\documentclass[a4paper,11pt]{article}
\usepackage{amsmath,amssymb}  
\usepackage{amsmath, amssymb}
\usepackage{latexsym}
\usepackage{amscd}
\usepackage{enumerate}
\usepackage{graphics}
\usepackage{graphicx}

\numberwithin{equation}{section}

\newtheorem{prop}{Proposition}[section]
\newtheorem{lem}[prop]{Lemma}
\newtheorem{thm}[prop]{Theorem}
\newtheorem{dfn}[prop]{Definition}

\newtheorem{rem}{Remark}[section]
\newtheorem{prf}{Proof.}

\newenvironment{proof}{\begin{prf}\rm }{\hfill \qed \end{prf}}
\newenvironment{remark}{\begin{rem}\rm }{\end{rem}}
\newenvironment{definition}{\begin{dfn}\rm }{\end{dfn}}

\newcommand{\abs}[1]{\left\vert{#1}\right\vert}

\newcommand{\R}{{\mathbb R}}
\newcommand{\C}{{\mathbb C}}

\renewcommand{\H}{{\mathbb H}}
\newcommand{\rsetvert}{\:\right\vert\:}
\newcommand{\set}[2]{\left\{\left.#1\vphantom{#2}\rsetvert#2\right\}}
\def\calA{{\mathcal A}}
\def\calS{{\mathcal S}}
\def\qed{\ifmmode\def\@qed{\quad\mbox{$\Box$}}\else\def\@qed{\hfill $\Box$}\fi \@qed}
\def\gerg{{\mathfrak g}}
\def\gerh{{\mathfrak h}}
\def\gerk{{\mathfrak k}}
\def\gers{{\mathfrak s}}

\def\gero{{\mathfrak o}}
\def\gerp{{\mathfrak p}}
\def\geru{{\mathfrak u}}

\def\gerH{{\mathfrak h}}

\def\gerl{{\mathfrak l}}

\newcommand{\sufcond}{\sharp}

\begin{document}
\title{Clifford quartic forms and local functional equations of non-prehomogeneous type}
\author{%
Fumihiro Sato%
\thanks{The first and second authors are partially supported by the grant in aid of scientific research of JSPS No.24540029, 25247001, 15K04800 and 24540049. }
\footnote{Department of Mathematics, Rikkyo University, 3-34-1 Nishi-Ikebukuro, Toshima-ku,
Tokyo, 171-8501, Japan. \newline e-mail: \texttt{sato@rikkyo.ac.jp }}
\quad and \quad Takeyoshi Kogiso${}^*$\!\!
\footnote{Department of Mathematics, Josai University, 1-1 Keyakidai, Sakado, Saitama, 350-0295, Japan.\newline e-mail: \texttt{kogiso@math.josai.ac.jp}}%
}
\date{}
\maketitle

\begin{abstract} 
It is known that one can associate local zeta functions satisfying a functional equation  to the irreducible relative invariant of an irreducible regular prehomogeneous vector space.  
We construct polynomials of degree $4$ that can not be obtained from prehomogeneous vector spaces, but, for which one can associate local zeta functions satisfying functional equations.  
Let $C_n $ be the Clifford algebra of the positive definite real quadratic form $v_1^2+\cdots+v_n^2$. For a $C_p \otimes C_q$-module $W$,   
we define a homogeneous polynomial $\tilde P$ (called a {\it Clifford quartic form}) of degree 4 on $W$  such that the associated local zeta functions satisfy a functional equation. 
The Clifford quartic forms $\tilde P$ can not be a relative invariants of any prehomogeneous vector space unless  $p+q$ or $\dim W$ are small. 
We also classify the exceptional cases of small dimension, namely, we determine all the prehomogeneous vector spaces with Clifford quartic forms as relative invariant.  
\end{abstract}


\section*{Introduction}

Let $(\mathbf{G},\rho,\mathbf{V})$ be an irreducible regular prehomogeneous vector space defined over $\R$ and $(\mathbf{G},\rho^*,\mathbf{V}^*)$ the dual of  $(\mathbf{G},\rho,\mathbf{V})$. Then, there exists an  irreducible homogeneous polynomial $P(v)$ (resp.\ $P^*(v^*)$) on $\mathbf{V}$ (resp.\ $\mathbf{V}^*$) such that the complement $\mathbf{\Omega}$ (resp.\ $\mathbf{\Omega}^*$) in $\mathbf{V}$ (resp.\ $\mathbf{V}^*$) of the hypersurface defined by $P(v)$ (resp.\ $P^*(v^*)$) is a single $\mathbf{G}$-orbit. Let $\Omega_1,\ldots,\Omega_\nu$ (resp.\ $\Omega^*_1,\ldots,\Omega^*_\nu$) be the connected components of $\mathbf{\Omega}\cap\mathbf{V}(\R)$ (resp.\ $\mathbf{\Omega^*}\cap\mathbf{V}^*(\R)$). 
For a rapidly decreasing function $\Phi$ (resp.\ $\Phi^*$) on $\mathbf{V}(\R)$ (resp.\ $\mathbf{V}^*(\R)$) and $i=1,\ldots,\nu$, the local zeta functions are defined by 
\[
\zeta_i(s,\Phi)= \int_{\Omega_i} \abs{P(v)}^s \Phi(v)\,dv, \quad
\zeta_i^*(s,\Phi^*)= \int_{\Omega^*_i} \abs{P^*(v^*)}^s \Phi^*(v^*)\,dv^*. 
\]
These integrals are absolutely convergent for $\Re(s)>0$ and can be continued to meromorphic functions of $s$ in $\C$. The fundamental theorem in the theory of prehomogeneous vector spaces (\cite{Sato-Shintani}, \cite{F.Sato},  \cite{KimuraBook}) states that they satisfy a functional equation 
\begin{equation}
\label{eqn:0.1}
\zeta_i^*(s,\hat\Phi)=\sum_{j=1}^\nu \gamma_{ij}(s)\zeta_i\left(-\frac nd -s,\Phi\right),
\end{equation}
where $\hat\Phi$ is the Fourier transform of $\Phi$, $n=\dim \mathbf{V}$, $d = \deg P$ and the gamma-factors $\gamma_{ij}(s)$ are meromorphic functions of $s$ independent of $\Phi$. 

Since the local zeta functions can be defined for an arbitrary polynomial $P$,  
it is natural to ask whether there are any polynomials other than the ones obtained from the theory of prehomogeneous vector spaces satisfying a functional equation
of the form $(\ref{eqn:0.1})$. 
Such polynomials, if exist, should be analytically and arithmetically interesting objects. 

The first example was found by Faraut and Koranyi \cite[Theorem 16.4.3]{Faraut-Koranyi} (see also Clerc \cite{Clerc}). 
They constructed polynomials having a functional equation of the form $(\ref{eqn:0.1})$ from representations of  Euclidean Jordan algebras and 
observed that the polynomials can not be obtained from prehomogeneous vector spaces for the simple Euclidean Jordan algebras of rank 2 (apart from some low-dimensional exceptions). 
However it seems that no complete determination of the non-prehomogeneous cases has not been done yet. 

In \cite{F.Sato4}, the first author considered the pullback of local functional equation by non-degenerate dual quadratic mappings. 
Let $P$ (resp.\ $P^*$) be a homogeneous polynomial of degree $d$ on a $\C$-vector space $\mathbf V$ (resp.\ $\mathbf V^*$) with $\R$-structure $V$ (resp.\ $V^*$). Here $\mathbf{V}^*$ is the vector space dual to $\mathbf {V}$. 
Let $Q:\mathbf{W}\rightarrow\mathbf{V}$ and $Q^*:\mathbf{W}^*\rightarrow\mathbf{V}^*$ be non-degenerate dual quadratic mappings (in the sense explained in \S 1). Then, if the local zeta functions attached to $P$ and $P^*$ satisfy a local functional equation of the form $(\ref{eqn:0.1})$, then the local zeta functions attached to the pullbacks $\tilde P=P\circ Q$ and  $\tilde P^*=P^*\circ Q^*$ also satisfy a similar functional equation (see Theorem \ref{thm:main}). 
This generalizes the earlier result of Faraut and Koranyi (see \cite[\S 2.2]{F.Sato4}). 

In the present paper, we examine the case where $\dim \mathbf{V} = p+q$ and $P(v)=v_1^2+\cdots+v_p^2-v_{p+1}^2-\cdots-v_{p+q}^2$ and classify the self-dual non-degenerate quadratic mappings  $Q$ to the quadratic space $(\mathbf{V},P)$. 
As a result, we can obtain many examples of polynomials having local functional equations, but not obtained from prehomogeneous vector spaces.   
Our main results are summarized as follows:
\begin{itemize}
\item 
The self-dual quadratic mappings to $(\mathbf{V},P)$ defined over $\R$ are in one to one correspondence to (not necessarily irreducible) representations of $C_p \otimes C_q$, where $C_p$ and $C_q$ are the real Clifford algebras of the positive definite quadratic forms $v_1^2+\cdots+v_p^2$ and $v_{p+1}^2+\cdots+v_{p+q}^2$ (Theorem \ref{thm:corresponds to tensor prod of Clif}). 
For a representation $\rho$ of $C_p \otimes C_q$ on a real vector space $W=\mathbf{W}(\R)=\R^m$, we may assume that the images $S_i=\rho(e_i)$ of the standard basis $e_1,\ldots,e_{p+q}$ are symmetric  matrices. 
Then the quadratic mapping $Q:W\rightarrow V$ given by $Q(w)=({}^tw S_1 w,\ldots,{}^tw S_{p+q} w)$ is self-dual and we call the polynomial
\[
\tilde P(w)=P(Q(w))=\sum_{i=1}^p ({}^tw S_i w)^2 - \sum_{j=p+1}^{p+q} ({}^tw S_j w)^2
\]
the {\it Clifford quartic form}\/ associated with $\rho$. 
In the special  case where $(p,q)=(1,q)$ and $S_1$ is the identity matrix, $\tilde P$ coincides with the polynomial that Faraut-Koranyi \cite{Faraut-Koranyi} obtained from a representation of the simple Euclidean Jordan algebra of rank 2.
\item 
The quadratic mapping $Q$ given as above is non-degenerate, if and only if the associated Clifford quartic form $\tilde P$ does not vanish identically and then the local zeta functions for $\tilde P$ satisfy functional equations of the form $(\ref{eqn:0.1})$ 
(Theorem \ref{thm:2.14}). 
For simplicity we give here an explicit formula for the functional equation for $p\geq q \geq 2$. 
For a rapidly decreasing function $\Psi$ on $W$, the local zeta functions are defined by 
\[
\tilde\zeta_+(s,\Psi) = \int_{\tilde P(w)>0} \abs{\tilde P(w)}^s \Psi(w)\,dw, \quad 
\tilde\zeta_-(s,\Psi) = \int_{\tilde P(w)<0} \abs{\tilde P(w)}^s \Psi(w)\,dw
\]
and they satisfy the functional equation
\begin{eqnarray*}
\begin{pmatrix}
\tilde\zeta_+\left(s,\hat\Psi\right)  \\
\tilde\zeta_-\left(s,\hat\Psi\right)  
\end{pmatrix}
 &=& 2^{4s+m/2}\pi^{-4s -2-m/2} \Gamma (s+1) 
          \Gamma \left(s+\frac n2 \right) \Gamma\left (s+1+\frac{m-2n}4 \right) 
          \Gamma \left(s+\frac m4\right)  \\
 & &   {}\times  \sin \pi s \begin{pmatrix}
     \sin\pi\left(s+\frac{q-p}2\right) 
        & -2\sin\frac{\pi p}2 \cos\frac{\pi q}2  \\
     -2\sin\frac{\pi q}2 \cos\frac{\pi p}2
       & \sin\pi\left(s+\frac{p-q}2\right) 
     \end{pmatrix}
                  \begin{pmatrix}
                  \tilde\zeta_+\left(-\frac m4 - s,\Psi\right)  \\
                  \tilde\zeta_-\left(-\frac m4 - s,\Psi\right)  
                   \end{pmatrix},
\end{eqnarray*}
where $m=\dim W$ and $\hat \Psi$ is the Fourier transform of $\Psi$.
The degenerate cases appear only for $(p,q,m)=(2,1,2),(3,1,4),(5,1,8),(9,1,16),(2,2,4),(3,3,8),(5,5,16)$, and $(p,q)=(1,1)$ with $S_1 = \pm S_2$ (Theorem \ref{thm:degenerate}). 
\item 
We completely determine when the Clifford quartic form $\tilde P(w)$ becomes a relative invariant of some prehomogeneous vector space (Theorem \ref{thm:pv}). 
In particular, if $p+q \geq 12$, then, there exist no prehomogeneous vector spaces having $\tilde P$ as relative invariant. If $p+q \leq 4$, then, $\tilde P$ is always a relative invariant of some prehomogeneous vector space. If $ 5 \leq p+q \leq 11$ and $p+q\ne6$, then, $\tilde P$ is a relative invariant of a prehomogeneous vector space only for very low-dimensional cases. 
Thus most of the Clifford quartic forms are non-prehomogeneous. 
\item 
To classify the cases related to prehomogeneous vector spaces, we need good knowledge on the group of symmetries of the Clifford quartic form $\tilde P$. 
The group contains $Spin(p,q) \times H_{p,q}(\rho)$ where $H_{p,q}(\rho)$ is the intersection of the orthogonal groups $O(S_1),\ldots,O(S_{p+q})$. 
Except for a few low-dimensional cases, the Lie algebra of the group of symmetries of $\tilde P$ coincides with $\mathfrak{so}(p,q) \times \mathrm{Lie}(H_{p,q}(\rho))$ (Theorem \ref{thm:3.1}) and the structure of $\mathrm{Lie}(H_{p,q}(\rho))$ can be determined explicitly (Theorem \ref{thm:3.4}). 
\end{itemize}

Thus our results show that the class of homogeneous polynomials that satisfy local functional equations of the form (\ref{eqn:0.1}) is broader than the class of relative invariants of regular prehomogeneous vector spaces. 
The characterization of such polynomials is an interesting open problem.  
In relation to this characterization problem (in a more general form), Etingof, Kazhdan and Polishchuk (\cite{EKP}) considered 
 the following condition for a homogeneous rational function $f$ on a finite-dimensional vector space $\mathbf V$:
\begin{quote}
$v \mapsto \mathrm{grad}\,f(v)$ defines a birational mapping of $\mathbb P(\mathbf V) \longrightarrow \mathbb P(\mathbf V^*)$.
\end{quote}
They called this condition {\it the projective semiclassical condition\/} (PSC). A function satisfying PSC is often called {\it homaloidal}. 
It is observed that the condition PSC is closely related to the existence of local functional equation.  
For example, regular prehomogeneous vector spaces has homaloidal relative invariant polynomials and it is difficult to construct non-prehomogeneous homaloidal polynomials.  
The classification of homaloidal polynomials is a difficult problem and of considerable interest in algebraic geometry (\cite{Bruno}, \cite{CRS}, \cite{Dolgachev}).

For a homaloidal homogeneous rational function $f$, there exists a rational function $f^*$ satisfying the identity $f^*(\mathrm{grad}\,\log f(v))=1/f(v)$, which is called the multiplicative Legendre transform of $f$.  
In \cite{EKP}, the authors raised the following question and answered to it affirmatively for cubic forms:
\begin{quote}
``Is it true that any homaloidal polynomial whose multiplicative Legendre transform is also a polynomial is a relative invariant of a prehomogeneous vector space?''  
\end{quote} 
However, it can be easily proved (Theorem \ref{thm:homaloidal}) that every Clifford quartic form is homaloidal and its multiplicative Legendre transform coincides with the original Clifford quartic form (up to constant). 
Hence the answer to the question above is negative, since Clifford quartic forms are nonprehomogeneous in general. 

The organization of this paper is as follows:
In \S 1, we recall the pullback theorem of local functional equations in \cite{F.Sato4}. In \S 2, the degenerate cases are classified and the functional equations satisfied by the local zeta functions of the Clifford quartic forms are calculated. In \S 3, we investigate the groups of symmetries of the Clifford quartic forms. The proofs of the main results in this section (Theorems \ref{thm:3.1}, \ref{thm:3.4}) will be given in \S 5 and 6. 
The classification of the prehomogeneous cases is done in \S 4.

As in the case of relative invariants of prehomogeneous vector spaces, the Clifford quartic forms are expected to enjoy rich arithmetic properties. 
In fact,  we can associate with the Clifford quartic forms global zeta functions satisfying a functional equation, which are analogues of genus zeta functions of quadratic forms. 
For the polynomials constructed by Faraut-Koranyi, Achab defined global zeta functions and proved their functional equations (\cite{Achab}, \cite{Achab2}).  
In her argument, it is crucial that $Q^{-1}(v)_\R$ $(P(v)\ne0)$ is compact and her method does not apply to our general setting.
Our method is based on the theory of automorphic pairs of distributions on prehomogeneous vector spaces (\cite{Tamura}, \cite{Sato-Tamura}). 
We discuss it in a separate paper (\cite{F.Sato5}). 

\subsection*{Acknowledgment} 
We thank H.~Ochiai, H.~Ishi and T.~Yoshino for helpful discussions. 


\subsection*{Notation} 

We denote by $\R, \C$ and $\H$, respectively, the field of real numbers, the field of complex numbers and the Hamilton quaternion algebra. 
For $\mathbb{K}=\R,\C,\H$, we write 
\begin{eqnarray*}
M(m;\mathbb{K}) & & \text{for the matrix algebra of size $m$ over $\mathbb{K}$}, \\
M(m,n;\mathbb{K}) & & \text{for the set of $m$ by $n$ matrices with entries in $\mathbb{K}$}, \\
\mathrm{Sym}(m;\mathbb{K}) &=& \set{X \in M (m;{\Bbb K})}{{}^t X =X}, \\
\mathrm{Alt}(m;\mathbb{K}) &=& \set{X \in M (m;{\Bbb K})}{{}^t X =-X}.
\end{eqnarray*}
For a $w \in {\R}^m$ and an $S \in \mathrm{Sym}(m,\R)$, we put $S[w]:={}^t w S w$.
We say that the signature of $S$ (or of the quadratic form $S[w]$) is $(p,q)$, if $S$ is non-degenerate and has exactly $p$ positive and $q$ negative eigenvalues. 
For square matrices $A \in M(m;\mathbb{K})$ and $B \in M(n;\mathbb{K})$,  we put  $A \perp B := \begin{pmatrix} A & 0 \\ 0 & B \end{pmatrix} \in M(m+n; \mathbb{K})$. 
The identity matrix and the zero matrix of size $m$ are denoted by $1_m$ and $0_m$, respectively. We write $1_{p,q}$ for $1_p\perp -1_q$. 
We put $\mathbf{e}[z]:=\exp(2\pi\sqrt{-1}z)$. 
For a real vector space $V$, we denote by $\calS(V)$ the space of rapidly decreasing functions on $V$. 
We use the same symbols as those in \cite[Chapter X, \S 2.1]{Helgason} to denote real classical Lie algebras. 
 
\section{Pullback of local functional equations by quadratic mappings}

In this section, we recall the main result of \cite{F.Sato4}. 

Let $\mathbf{V}$ (resp.\ $\mathbf{W}$) be a complex vector space  of dimension $n$ (resp.\ $m$) with real-structure $V$ (resp.\ $W$) and $\mathbf{V}^*$ (resp.\ $\mathbf{W}^*$) the vector space dual to $\mathbf{V}$ (resp.\ $\mathbf{W}$). 
The dual vector space $V^*$ (resp.\ $W^*$) of the real vector space $V$ (resp.\ $W$) can be regarded as a real-structure of $\mathbf{V}^*$(resp.\ $\mathbf{W}$). 
Let $P$ (resp.\ $P^*$) be an irreducible  homogeneous polynomial of degree $d$ on $\mathbf{V}$  
(resp.\ $\mathbf{V}^*$) defined over $\R$. 
We put
\begin{gather*}
\mathbf{\Omega}=\set{v \in \mathbf{V}}{P(v)\ne0}, \quad 
\Omega=\mathbf{\Omega}\cap V,\\ 
\mathbf{\Omega^*}=\set{v^* \in \mathbf{V}^*}{P^*(v^*)\ne0}, \quad 
\Omega^*=\mathbf{\Omega^*}\cap V^*.
\end{gather*}
We assume that 
\begin{description}
\item[(A.1)] there exists a biregular rational mapping $\phi:\mathbf{\Omega} \rightarrow \mathbf{\Omega^*}$ defined over $\R$. 
\end{description}
Let 
\[
\Omega=\Omega_1 \cup \cdots \cup \Omega_\nu, \quad
\Omega^*=\Omega_1^* \cup \cdots \cup \Omega_\nu^*
\]
be the decomposition of $\Omega$ and $\Omega^*$ into connected components. 
Note that (A.1) implies that the numbers of connected components of $\Omega$ and  $\Omega^*$ 
are the same and we may assume that 
\[
\Omega_j^*=\phi(\Omega_j) \quad (j=1,\ldots,\nu).
\]

Suppose that we are given quadratic mappings $Q:\mathbf{W} \rightarrow \mathbf{V}$ 
and $Q^*:\mathbf{W^*} \rightarrow \mathbf{V^*}$ defined over $\R$. 
The mappings $B_Q:\mathbf{W} \times \mathbf{W} \rightarrow \mathbf{V}$ 
and $B_{Q^*}:\mathbf{W^*} \times \mathbf{W^*} \rightarrow \mathbf{V^*}$ 
defined by 
\[
B_Q(w_1,w_2):=Q(w_1+w_2)-Q(w_1)-Q(w_2), \quad
B_{Q^*}(w_1^*,w_2^*):=Q^*(w_1^*+w_2^*)-Q^*(w_1^*)-Q^*(w_2^*)
\]
are bilinear. 
For given $v \in \mathbf{V}$ and $v^* \in \mathbf{V^*}$, the mappings $Q_{v^*}:\mathbf{W} \rightarrow \C$ and 
$Q^*_{v}:\mathbf{W^*} \rightarrow \C$ defined by 
\[
Q_{v^*}(w)=\langle Q(w),v^* \rangle, \quad   
Q^*_{v}(w^*)=\langle v,Q^*(w^*)\rangle
\]
are quadratic forms on $\mathbf{W}$ and $\mathbf{W^*}$, which take values in $\R$ on $W$ and $W^*$,  respectively. 
We assume that $Q$ and $Q^*$ are non-degenerate 
and dual to each other with respect to the biregular mapping $\phi$ in (A.1). 
This means that $Q$ and $Q^*$ satisfy the following: 
\begin{description}
\item[(A.2)] 
(i) (Nondegeneracy) The open set $\tilde{\mathbf{\Omega}}:=Q^{-1}(\mathbf{\Omega})$ 
(resp.\ $\tilde{\mathbf{\Omega^*}}={Q^*}^{-1}(\mathbf{\Omega^*})$) is not empty  and the rank of the differential of $Q$ (resp.\ $Q^*$) 
at $w \in \tilde{\mathbf{\Omega}}$ 
(resp.\ $w^* \in \tilde{\mathbf{\Omega^*}}$) is equal to $n$. 
(In particular, $m \geq n$.)  \\
(ii) (Duality) For any $v \in \mathbf{\Omega}$, 
the quadratic forms $Q_{\phi(v)}$ and $Q^*_v$ are dual to each other. 
Namely, fix a basis of $W$ and the basis of $W^*$ dual to it, and denote by $S_{v^*}$ and $S^*_v$ the 
matrices of the quadratic forms $Q_{v^*}$ and $Q^*_v$ with respect to the bases. 
Then $S_{\phi(v)}$ and $S^*_v$ $(v \in \mathbf{\Omega})$ are non-degenerate and $S_{\phi(v)}=(S^*_v)^{-1}$.   
\end{description}

\begin{remark}  
(1) By the assumption, there exist non-zero constants $\alpha, \beta$ satisfying
\[
\det(Q_v^*)=\alpha P(v)^{m/d}, \quad 
\det \left(\frac{\partial \phi(v)_i}{\partial v_j}\right)=\beta P(v)^{-2n/d}. 
\]

(2) 
In \cite{F.Sato4}, the assumptions (A.1) and (A.3) (=\,(A.2) in the present paper) are erroneously formulated only by referring to real structure. Moreover $\Omega$ and $\Omega^*$ in \cite[p.167, Lines 21 and 22]{F.Sato4}  should be $\mathbf{\Omega}$  and  $\mathbf{\Omega}^*$. 
\end{remark}

The main result in \cite{F.Sato4} is that, in the above setting, if $P(v)$ and $P^*(v^*)$ satisfy a local functional equation, then the pull backs $\tilde P(w):=P(Q(w))$ and $\tilde P^*(w^*):=P^*(Q^*(w^*))$ also satisfy a local functional equation.   Let us give a precise formulation.  

For an $s \in \C$ with $\Re(s)>0$, we  define a continuous function $\abs{P(v)}_j^s$ on $V$ by
\[
\abs{P(v)}_j^s
 = \begin{cases}
    \abs{P(v)}^{s}, & v \in \Omega_j, \\
   0, & v \not\in \Omega_j. 
   \end{cases}
\]
The function  $\abs{P(v)}_j^s$ can be extended to a tempered distribution depending on $s$ in $\C$ meromorphically. 
Similarly we define $\abs{P^*(v^*)}_j^s$ $(s \in \C)$.

We denote by $\calS(V)$ and $\calS(V^*)$ the spaces of rapidly decreasing functions on the real vector spaces 
$V$ and $V^*$, respectively. 
For $\Phi\in\calS(V)$ and $\Phi^*\in\calS(V^*)$, we define the local zeta functions 
by setting 
\[
\zeta_i(s,\Phi) = \int_V \abs{P(v)}_i^s \Phi(v)\,dv, \quad 
\zeta^*_i(s,\Phi^*) = \int_{V^*} \abs{P^*(v^*)}_i^s \Phi^*(v^*)\,dv^* \quad (i=1,\ldots,\nu).
\]
It is well-known that the local zeta functions $\zeta_i(s,\Phi)$, $\zeta^*_i(s,\Phi^*)$ are absolutely convergent for 
$\Re(s)>0$ and have analytic continuations to meromorphic functions of $s$ 
in $\C$.  
We assume the following:
\begin{description}
\item[(A.3)]  
A local functional equation of the form 
\begin{equation}
\label{eqn:FE}
\zeta^*_i(s,\hat\Phi) = \sum_{j=1}^\nu \Gamma_{ij}(s) \zeta_j(-\frac nd-s,\Phi)
\quad (i=1,\ldots,\nu)
\end{equation}
holds for every $\Phi \in \calS(V)$, 
where $\Gamma_{ij}(s)$ are meromorphic functions on $\C$ not depending on $\Phi$ with 
$\det(\Gamma_{ij}(s))\ne0$ and 
\[
\hat\Phi(v^*) = \int_V \Phi(v) \exp(-2\pi \sqrt{-1}\langle v, v^* \rangle)\,dv,
\]
the Fourier transform of $\Phi$.
\end{description}

We put 
\begin{gather*}
\tilde P(w) = P(Q(w)), \quad 
\tilde P^*(w^*) = P^*(Q^*(w^*)) \\
\tilde\Omega_i = Q^{-1}(\Omega_i), \quad 
\tilde\Omega^*_i = {Q^*}^{-1}(\Omega_i^*)\quad (i=1,\ldots,\nu). 
\end{gather*}
Some of $\tilde\Omega_i$'s and $\tilde\Omega^*_i$'s may be empty. 
We define $\abs{\tilde P(w)}_i^s$ and  $\abs{\tilde P^*(w^*)}_i^s$ in the same manner as above. 
For $\Psi \in \calS(W)$ and $\Psi^* \in \calS(W^*)$, we define the zeta functions associated with $\tilde P$ and $\tilde  P^*$ by
\[
\tilde\zeta_i(s, \Psi) = \int_{W} \abs{\tilde P(w)}_i^s \Psi(w)\,dw, \quad 
\tilde\zeta^*_i(s, \Psi^*) = \int_{W^*} \abs{\tilde P^*(w^*)}_i^s \Psi^*(w^*)\,dw^*.
\] 
We denote by $\hat\Psi$ the Fourier transform of $\Psi$:
\[
\hat\Psi(w^*) = \int_{W} \Psi(w) \exp(2\pi \sqrt{-1} \langle w, w^* \rangle)\,dw.
\]
Then our main result is that the local functional equation $(\ref{eqn:FE})$ for $P$ and $P^*$ implies a  
local functional equation for $\tilde P$ and $\tilde P^*$'s and the gamma factors in the new functional 
equation can be written explicitly.

\begin{thm}[\cite{F.Sato4}, Theorem 4]
\label{thm:main}
Under the assumptions {\rm (A.1), (A.2), (A.3)}, the zeta functions $\tilde\zeta_i(s, \Psi)$ and $\tilde\zeta^*_i(s, \Psi^*)$ satisfy
the local functional equation
\[
 \tilde\zeta^*_i\left(s, \hat\Psi\right)
 =  \sum_{j=1}^\nu \tilde\Gamma_{ij}(s) \tilde\zeta_j\left(-\frac{m}{2d}-s, \Psi\right), 
\]
where the gamma factors $\tilde \Gamma_{ij}(s)$ are given by 
\[
\tilde\Gamma_{ij}(s) = 2^{2ds+m/2} \abs{\alpha}^{1/2} \abs{\beta}^{-1}
\sum_{k=1}^\nu \mathbf{e}\left[\frac{p_k-q_k}8\right] \Gamma_{ik}(s)\Gamma_{kj}\left(s+\frac{m-2n}{2d}\right),
\]
where $\alpha,\beta$ are the constants defined in Remark 1.1 (1) and $(p_k,q_k)$ is the signature of the quadratic form $Q^*_v$ for $v \in \Omega_k$. 
\end{thm}

\begin{remark}  
(1) The signature $(p_k,q_k)$ of  $Q^*_v(w^*)$ does not depend on the choice of $v$, since $\Omega_k$ is connected. 

(2) In \cite{F.Sato4}, the theorem is formulated for multi-variable zeta functions. 
Here we restrict ourselves to single variable zeta functions for simplicity. 
\end{remark}

The theory of prehomogeneous vector spaces (see \cite{M.Sato0}, \cite{Sato-Shintani}, \cite{F.Sato}, \cite{KimuraBook}) provides a lot of examples of $P$ and $P^*$ satisfying (A.1) and (A.3).  
Therefore, if one can construct dual non-degenerate quadratic mappings to a prehomogeneous vector space, then by Theorem \ref{thm:main}, one obtains a new local functional equation. 
In \cite[Chap.\ 16]{Faraut-Koranyi}, Faraut and Koranyi proved that, starting from a representation of a Euclidean Jordan algebra, one can construct polynomials satisfying local functional equations (see also Clerc \cite{Clerc}). 
Theorem \ref{thm:main} generalizes their result (see \cite[\S 2.2]{F.Sato4}). 

The Faraut-Koranyi construction is especially interesting in the case of the simple Euclidean Jordan algebra of rank 2, since the polynomials $\tilde P$ obtained in this case are {\it not\/} relative invariants of prehomogeneous vector spaces
except for some low-dimensional cases, as is noticed in \cite{Clerc} (without specifying the low-dimensional exceptions explicitly). 
Let us explain this non-prehomogeneous example without referring to Jordan algebra. 
Let $V$ be the $q+1$-dimensional real quadratic space of signature $(1,q)$. 
We fix a basis $\{e_0,e_1,\ldots,e_q\}$ of $V$, for which the quadratic form is given by 
\[
P(x_0,x_1,\ldots,x_q)=x_0^2-x_1^2-\cdots-x_q^2.
\]
Denote by $C_q$ the Clifford algebra of the positive definite quadratic form $x_1^2+\cdots+x_q^2$ 
and consider a representation $S:C_q \rightarrow M(m;\R)$ of $C_q$ on an 
$m$-dimensional $\R$-vector space. 
We may assume that $S_i:=S(e_i)$ $(i=1,\ldots,q)$ are symmetric matrices.  
We denote by $W=\R^m$ the representation space of $S$ 
and define a quadratic mapping $Q:W \rightarrow V$ by 
\[
Q(w) = ({}^tw w)e_0 + \sum_{i=1}^q S_i [w] e_i.
\] 
Then, if $\tilde P(w)=P(Q(w))=({}^tw w)^2 - \sum_{i=1}^q (S_i [w] )^2$ does not vanish identically, 
$Q$ is a self-dual non-degenerate quadratic mapping and, by Theorem \ref{thm:main}, 
$\tilde P$ satisfies a local functional equation.  
In the next section, we generalize this construction by classifying the non-degenerate self-dual quadratic mappings to real non-degenerate quadratic spaces of arbitrary signature. 

\begin{remark} 
(1) In \cite{Ishi}, Ishi proved that, if $V$ is the underlying vector space of a semisimple (not necessarily Euclidean) Jordan algebra, $P$ is the determinant of the Jordan algebra and $Q:W \rightarrow V$ is a self-dual non-degenerate quadratic mapping, then the mapping $v \mapsto Q_v^*$ induces a representation of the Jordan algebra $V$. 

(2) In \cite{F.Sato6}, we obtained another local functional equations of non-prehomogeneous type in a different setting, which is motivated by the generalized Wishart distributions in multivariate statistics. 
\end{remark}

\section{Local functional equations of Clifford quartic forms} 

\subsection{Self-dual quadratic mappings and representations of Clifford algebras}

Let $p,q$ be non-negative integers, $V$ a real $p+q$-dimensional vector space and consider a quadratic form $P(v)$ of signature $(p,q)$. 
Fix a basis $\{e_1,\ldots,e_{p+q}\}$, which is called the standard basis, satisfying 
\[
P\left(\sum_{i=1}^{p+q} v_i e_i\right)=\sum_{i=1}^p v_i^2 - \sum_{j=1}^q v_{p+j}^2.
\]
We identify $V$ with $\R^{p+q}$ with the standard basis, and also with its dual vector space via the standard inner product 
$(x,y)=x_1y_1+\cdots+x_{p+q}y_{p+q}$. 
Put $\Omega=V\setminus\{P=0\}$. 
We determine the quadratic mappings $Q:W \rightarrow V$ that are self-dual with respect to 
the biregular mapping $\phi: \Omega \longrightarrow  \Omega$ defined by  
\[
\phi(v) := \frac 12 \mathrm{grad}\log P(v) = \frac{1}{P(v)}(v_1,\ldots,v_p,-v_{p+1},\ldots,-v_{p+q}). 
\]

By a quadratic mapping $Q$ of $W=\R^m$ to $V=\R^{p+q}$, we mean a mapping defined by 
\[
Q(w)=(S_1 [w],\ldots,  S_{p+q} [w])
\]
for some real symmetric matrices $S_1,\ldots,S_{p+q}$ of size $m$. 
For $v=(v_1,\dots,v_{p+q}) \in \R^{p+q}$, we put 
\[ 
S(v)=\sum_{i=1}^{p+q} v_i S_i, 
\]
then, by definition, the mapping $Q$ is self-dual with respect to $\phi$ if and only if 
\[
S(v)S(\phi(v))=1_m \quad (v \in \Omega). 
\]
If we put
\begin{equation}
\label{eqn:sign}
\varepsilon_i = \begin{cases} 
                     1 & (1 \leq i \leq p), \\
                     -1 & (p+1 \leq i \leq p+q),
                    \end{cases}
\end{equation}
this condition is 
equivalent to the polynomial identity
\[
\sum_{i=1}^p v_i^2 S_i^2 - \sum_{j=1}^q v_{p+j}^2 S_{p+j}^2 
 + \sum_{1\leq i<j \leq p+q} v_iv_j \left(\varepsilon_j S_iS_j+\varepsilon_i S_jS_i\right)
= P(v)1_m. 
\]
This identity holds if and only if 
\begin{eqnarray}
\label{eqn:clifford cond1}
S_i^2 &=& 1_m \quad  (1 \leq i \leq p+q), \\
\label{eqn:clifford cond2}
S_iS_j&=& \begin{cases} 
                S_jS_i & (1\leq i \leq p < j \leq p+q \ \text{or}\ 1\leq j \leq p < i \leq p+q )\\
                -S_jS_i & (1\leq i,j \leq p\ \text{or}\ p+1 \leq i,j \leq p+q). 
               \end{cases}.
\end{eqnarray}
This means that the linear map $S:V \rightarrow {\rm Sym}(m;\R)$ can be extended to a representation of the tensor product of the Clifford algebra $C_p$ of $v_1^2+\cdots+v_p^2$ and 
the Clifford algebra $C_q$ of $v_{p+1}^2+\cdots+v_{p+q}^2$. 

Conversely, if we are given a representation $\rho:C_p \otimes C_q \rightarrow M(m;\R)$, 
then we can obtain a self-dual quadratic mapping $Q:W=\R^m \rightarrow V$. 
Indeed, the images of the standard basis 
$S_1=\rho(e_1),\ldots,S_{p+q}=\rho(e_{p+q})$ satisfy the relations above. 
(We always identify $e_i\otimes 1$ (resp.\ $1 \otimes e_i$) with $e_i$ for $1 \leq i \leq p$ (resp.\ $p+1 \leq i \leq p+q$).)
Moreover, since $C_p\otimes C_q$ is semisimple, $\rho$ is a direct sum of irreducible representations. 
Any  irreducible representation of  $C_p\otimes C_q$ is a tensor product of an irreducible representation of $C_p$ and an irreducible representation of $C_q$. 
Hence,  $S_i$ is of the form $(\rho_1\otimes \rho'_1)(e_i) \perp \cdots \perp (\rho_r\otimes \rho'_r)(e_i)$  for some representations $\rho_1,\ldots,\rho_r$ of $C_p$ and some representations $\rho'_1,\ldots,\rho'_r$ of $C_q$. 
Therefore, by the following lemma (applied to the positive definite case), we may take symmetric matrices as $S_1,\ldots,S_{p+q}$ (by taking conjugate, if necessary) and then the mapping $Q(w)=(S_1[w],\ldots,S_{p+q}[w])$ is self-dual.   

\begin{lem}
\label{lem:finite group}
Let $P$ be a quadratic form on $V={\Bbb R}^{p+q} $ of signature $(p,q)$ and let $e_1, \dots , e_{p+q} $ be the standard basis of $V$ such that $P( \displaystyle{\sum_{i=1}^{p+q}  x_i e_i )= \sum_{i=1}^p x_i^2 - \sum_{j=1}^q x_{p+j}^2 .}$ Denote by $C_{p,q}$ the Clifford algebra of the quadratic form $P$ and let $\rho : C_{p,q} \rightarrow M(m;{\Bbb R})$ be a representation of $C_{p,q}.$ 
Then, in the equivalence class of $\rho$, there exists a representation with the property that $\rho(e_i )$ is a symmetric matrix for $1 \leq i \leq p $ and a skew-symmetric matrix for $p+1 \leq i \leq p+q.$
\end{lem}

\begin{proof} 
By the definition of the Clifford algebra $C_{p,q}$, we have 
$$e_i^2=1 \quad ( 1 \leq i \leq p), \quad e_i^2 =-1 \quad ( p+1 \leq i \leq p+q ) , \quad e_i e_j =-e_i e_j \quad (i \neq j ).$$
Hence, the multiplicative group $G$ generated by $\{ -1, e_1, \dots, e_{p+q} \}$ is a finite group and $\rho$ gives a group-representation of $G$ on ${\Bbb R}^m$. Therefore, if we replace $\rho$  by an equivalent representation if necessary, we may assume that every element in $G$ is represented by an orthogonal matrix. Then $\rho(e_i)=\rho(e_i)^{-1} ={}^t \rho(e_i) $ for $1 \leq i \leq p$, and $\rho(e_i)=-\rho(e_i)^{-1}=-{}^t \rho(e_i)$ for $p+1 \leq i \leq p+q$. 
\end{proof}

Thus we have proved the following theorem on the correspondence between self-dual quadratic mappings and representations of $C_p\otimes C_q$. 
 
\begin{thm}
\label{thm:corresponds to tensor prod of Clif}
Self-dual quadratic mappings $Q$ of $W=\R^m$ to the quadratic space $(V,P)$ 
correspond to representations $\rho$ of $C_p \otimes C_q$ such that $\rho(V)$ is contained in ${\rm Sym}(m;\R)$.  
\end{thm}

We call the symmetric matrices $S_1=\rho(e_1),\ldots,S_{p+q}=\rho(e_{p+q})$ the {\it basis matrices} of $\rho$. 

\begin{remark}
The construction above is a generalization of the quadratic mappings obtained from  representations of simple Euclidean Jordan algebra of rank 2 in the theory of Faraut-Koranyi \cite{Faraut-Koranyi}.  
In the case $(p,q)=(1,q)$, we have $C_1 \cong \R \oplus \R$ and $C_1 \otimes C_q \cong C_q \oplus C_q$. 
Hence representations of  $C_1 \otimes C_q$ can be identified with the direct sum 
of two $C_q$-modules $W_+$ and $W_-$. On $W_+$ (resp.\ $W_-$), $e_1$ acts as multiplication by $+1$ (resp. $-1$). 
The quadratic mappings given by the Faraut-Koranyi construction correspond 
to representations of $C_1 \otimes C_q$ for which $W_-=\{0\}$.    
\end{remark}

We put $R_{p,q}=C_p\otimes C_q$. We denote by $R_{p,q}^+$ the subalgebra of $R_{p,q}$ consisting of all the even elements, namely, the subalgebra generated by $e_ie_j$ $(1 \leq i < j \leq p+q)$. 

It is well-known that the structure of $C_p$ depends on $p \bmod 8$ as the following lemma shows (see \cite{Ono}, \cite{Porteous}, \cite{Varadarajan}): 
\begin{lem}
\label{lem:pos def Clif}
The structure of $C_p$ is given by the following table:
\[
\def\arraystretch{1.3}
\begin{array}{|c|c|}
\hline
p & C_p \\\hline 
p \equiv 0 \pmod 8 & M(2^{p/2} ; {\Bbb R}) \\\hline
 p \equiv 1 \pmod 8 & M(2^{(p-1)/2} ; {\Bbb R}) \oplus M(2^{(p-1)/2} ; {\Bbb R}) \\\hline
 p \equiv 2 \pmod 8 & M(2^{p/2} ; {\Bbb R}) \\\hline
 p \equiv 3 \pmod 8 & M(2^{(p-1)/2} ; {\Bbb C}) \\\hline
 p \equiv 4 \pmod 8 & M(2^{(p-2)/2} ; {\Bbb H}) \\\hline
 p \equiv 5 \pmod 8 & M(2^{(p-3)/2} ; {\Bbb H}) \oplus M(2^{(p-3)/2} ; {\Bbb H}) \\\hline
 p \equiv 6 \pmod 8 & M(2^{(p-2)/2} ; {\Bbb H}) \\\hline
 p \equiv 7 \pmod 8 & M(2^{(p-1)/2} ; {\Bbb C}) \\\hline 
\end{array}
\] 
\end{lem}

The structure of $R_{p,q}$ and $R^+_{p,q}$ is easily seen from Lemma \ref{lem:pos def Clif}.

\begin{lem}
\label{lem:T and T'}
Put $n=p+q$. Then the structure of $R_{p,q}$ and $R_{p,q}^+$ is given by the following table:
\[
\def\arraystretch{1.3}
\begin{array}{|c|c|c|c|c|c|}
\hline 
\text{Type} & (R_{p,q}, R_{p,q}^+) & \ell & r & (\mathbb K, \mathbb K') & \{p \bmod 8, q \bmod 8\} \\
\hline 
 &  & 2^{n/2} & 2^{n/2-1} & (\R,\C) &  \{0,2\}, \{4,6\} \\
\cline{3-6}
\smash{\lower3mm\hbox{I}}  & \smash{\lower3mm\hbox{$(T, T')$}}  & 2^{(n-1)/2} & 2^{(n-1)/2} & (\C,\R) & \{0,7\}, \{2,3\}, \{3,4\}, \{6,7\} \\
\cline{3-6}
     &   & 2^{(n-1)/2} & 2^{(n-1)/2-1}  &   (\C,\mathbb H) &  \{0,3\}, \{2,7\}, \{3,6\}, \{4,7\} \\
\cline{3-6}
   &   & 2^{n/2-1} & 2^{n/2-1} &   (\mathbb H,\C) &  \{0,6\}, \{2,4\} \\
\hline
\smash{\lower3mm\hbox{II}}  & \smash{\lower3mm\hbox{$(T,2T')$}}  & 2^{n/2} & 2^{n/2-1} &   (\R,\R) &  \{0,0\}, \{2,2\}, \{4,4\}, \{6,6\} \\
\cline{3-6}
   &   & 2^{n/2-1} & 2^{n/2-2} &   (\mathbb H,\mathbb H)  &  \{0,4\}, \{2,6\} \\
\hline
   &    & 2^{(n-1)/2} & 2^{(n-1)/2} &   (\R,\R) &  \{0,1\}, \{1,2\}, \{4,5\}, \{5,6\} \\
\cline{3-6}
III & (2T,T')  & 2^{n/2-1} & 2^{n/2-1} &   (\C,\C) &  \{1,3\}, \{1,7\}, \{3,5\}, \{5,7\} \\
\cline{3-6}
   &   & 2^{(n-3)/2} & 2^{(n-3)/2} &   (\mathbb H,\mathbb H) &  \{0,5\}, \{1,4\}, \{1,6\}, \{2,5\} \\
\hline
\smash{\lower3mm\hbox{IV}}  & \smash{\lower3mm\hbox{$(2T,2T')$}}  & 2^{n/2-1} & 2^{n/2-1} &   (\C,\R) & \{3,3\}, \{7,7\} \\
\cline{3-6}
   &   & 2^{n/2-1} & 2^{n/2-2} &   (\C,\mathbb H) &  \{3,7\} \\
\hline
\smash{\lower3mm\hbox{V}}  & \smash{\lower3mm\hbox{$(4T,2T')$}}  & 2^{n/2-1}  &  2^{n/2-1} &   (\R,\R) & \{1,1\}, \{5,5\} \\
\cline{3-6}
   &    & 2^{n/2-2} & 2^{n/2-2} &   (\mathbb H,\mathbb H) &  \{1,5\} \\
\hline
\end{array}
\]
where $T$ (resp.\ $T'$) denotes the matrix algebra $M(\ell;\mathbb K)$ (resp.\  $M(r;\mathbb K')$),  and $2T$ (resp.\ $2T'$, $4T$) denotes $T \oplus T$ (resp.\ $T' \oplus T'$, $T \oplus T \oplus T \oplus T$). 
\end{lem}

The number of inequivalent irreducible representations of $R_{p,q}$ is equal to the number of simple components, namely, $1$ for type I and II , $2$ for type III and IV, and $4$ for type V.  
The dimension over $\R$ of the irreducible representations of $R_{p,q}$ is given by $\ell \dim_\R \mathbb K$, which is a power of $2$. 
As in Lemma \ref{lem:finite group}, we denote by $C_{p,q}$ the Clifford algebra of the quadratic form $P$  and by $C_{p,q}^+$ the subalgebra of $C_{p,q}$ of even elements. 
Then the algebra $R_{p,q}^+$ is isomorphic to $C_{p,q}^+$, while $R_{p,q}$ is not necessarily isomorphic to $C_{p,q}$. 
The isomorphism of $R_{p,q}^+$ to $C_{p,q}^+$ is given by $e_ie_j \mapsto \varepsilon_j\tilde e_i\tilde e_j$, where $\tilde e_i$ denotes the element $e_i \in V$ viewed as an element in $C_{p,q}$ and $\varepsilon_i$'s are the same as in $(\ref{eqn:sign})$. 
 
Let $\gerk_{p,q}$ be the real vector space spanned by $e_ie_j$ $(1 \leq i < j \leq p+q)$. 
Then $\gerk_{p,q}$ is a Lie subalgebra of $R_{p,q}$ with bracket product $[X,Y]:=XY-YX$.

\begin{lem}
\label{lem:equivariance}
The Lie algebra $\gerk_{p,q}$ is isomorphic to 
\[
\gers\gero(p,q)=\set{X\in M({p+q};\R)}{{}^tX1_{p,q}+1_{p,q}X=0}
\]
and a representation $\rho$ of $R_{p,q}$ on a vector space $W$ induces a Lie algebra representation of $\gerk_{p,q}$, which is a direct sum of  real (half-) spin representations of $\gers\gero(p,q)$. 
The self-dual quadratic map $Q:W \rightarrow V$ corresponding to $\rho$ is $\mathrm{Spin}(p,q)$-equivariant. 
Here, for $p+q \geq 3$, $\mathrm{Spin}(p,q)$ denotes the real spin group, which is a double covering group of the orthogonal group $\mathrm{SO}(p,q)$. For $p+q=2$, we understand that $\mathrm{Spin}(p,q) = \mathrm{SO}(p,q)$. 
(If $p+q=1$, then $\gerk_{p,q}=\{0\}$ and there is nothing to prove.) 
\end{lem}

\begin{proof}
For $Y \in \gerk_{p,q}$, we have 
\[
\left.\frac{d}{dt}S_k[\exp(t\rho(Y))w]\right|_{t=0} = ({}^t\rho(Y)S_k+S_k\rho(Y))[w] \quad (w \in W).
\]
In case $Y=e_ie_j$ $(i \ne j)$, the relations (\ref{eqn:clifford cond1}) and (\ref{eqn:clifford cond2}) imply that
\[
{}^t\rho(Y)S_k+S_k\rho(Y) 
 = {}^t(S_iS_j)S_k+S_k(S_iS_j)
 = \begin{cases}
     0 & (k \ne i,j), \\
     2S_j & (k=i), \\
    -2\varepsilon_i\varepsilon_j S_i & (k=j).
    \end{cases}
\]
This shows that a Lie algebra isomorphism $f$ of $\gerk_{p,q}$ onto $\gers\gero(p,q)$  is defined by 
\[
f(e_ie_j)=X_{ij}:=2E_{ij}-2\varepsilon_i \varepsilon_j E_{ji}\quad (1 \leq i < j \leq p+q)
\] 
($E_{ij}$ is the matrix unit) and $Q$ satisfies the identity
\[
\left.\frac{d}{dt}Q(\exp(t\rho(Y))w)\right|_{t=0} 
 = \left.\frac{d}{dt}\exp(tf(Y))Q(w)\right|_{t=0}. 
\]
This shows that $Q$ is $\mathrm{Spin}(p,q)$-equivariant. 
The representation of $\gerk_{p,q}$ on $W$ generates the representation $\rho$ of the algebra $R^+_{p,q}$ which is isomorphic to the even Clifford algebra $C^+_{p,q}$. 
Hence the  representation of $\gerk_{p,q}$ on $W$ is equivalent to a direct sum of real (half-) spin representations.   
\end{proof}

The following distinction between representations will play an important role later. 

\begin{definition}
\label{def:pure and mixed}
A representation $\rho$ of $R_{p,q}$ is called  {\it  pure}, if the restriction  $\rho^+:=\rho|_{R^+_{p,q}}$ of $\rho$ to $R_{p,q}^+$ is isotypic, namely, $\rho^+$ is a direct sum of several copies of an irreducible representation of $R^+_{p,q}$. 
If $\rho^+$ contains inequivalent irreducible representations of $R^+_{p,q}$, then $\rho$ is called {\it mixed}. 
If $\rho^+\otimes\C$ is isotypic as a representation of $R_{p,q}^+\otimes\C$, then $\rho$ is called {\it pure over\/} $\C$.    
Otherwise, it is called {\it mixed over\/} $\C$. 
A mixed $\rho$ is obviously mixed over $\C$. 
\end{definition}


\subsection{Clifford quartic forms and classification of degenerate cases}

 Let $p,q $ be non-negative integers satisfying $p \geq q \geq 0$ and $ p+q \geq 1.$ Let  
$S_1, \dots , S_{p+q} $ be the basis matrices of an $m$-dimensional representation $\rho$ of $R_{p,q}$ and define the quadratic mapping $Q:W={\Bbb R}^m \rightarrow V={\Bbb R}^{p+q} $ by 
\begin{center}
$Q(w)=(S_1 [w], \dots , S_{p+q} [w]) \quad (w \in W).$
\end{center}
We put 
\begin{center}
$P(v)=\displaystyle{\sum_{i=1}^p v_i^2 -\sum_{i=p+1}^{p+q} v_i^2 , \quad \tilde{P}(w)=P (Q(w))= \sum_{i=1}^p S_i [w]^2 -\sum_{i=p+1}^{p+q} S_i [w]^2}.$
\end{center}
We call the polynomial $\tilde P(w)$ of degree $4$ the {\it Clifford quartic form\/} associated with $\rho$. 
The representation $\rho$ is called {\it non-degenerate\/} if $Q$ is non-degenerate in the sense of (A.2) in \S 1, and {\it degenerate\/} otherwise. 
By Theorem \ref{thm:main}, if $Q$ is non-degenerate, the complex power of the Clifford quartic form satisfies 
a local functional equation with explicit gamma factor. 
In this subsection we determine when $Q$ is non-degenerate. 

\begin{thm}
\label{thm:degenerate}
The following four conditions are equivalent:
\begin{enumerate}
\def\labelenumi{{\rm (\roman{enumi})}}
\item The quadratic mapping $Q$ is degenerate. 
\item The Clifford quartic form  $\tilde{P}(w)$ vanishes identically.
\item There exist linear forms $f_1 (w), \dots, f_{p+q} (w),$ not all of which are $0,$ satisfying the identity 
\begin{equation}
\label{eqn:2.4}
\sum_{i=1}^{p+q}  S_i [w] f_i (w) \equiv 0. 
\end{equation}
\item $(p,q,m)=(2,1,2),(3,1,4),(5,1,8),(9,1,16),(2,2,4),(3,3,8),(5,5,16)$, or $(p,q)=(1,1)$ with $S_1 = \pm S_2$. (Note that the condition $S_1 = \pm S_2$ for $(p,q)=(1,1)$ is equivalent to that the representation $\rho$ is pure.)
\end{enumerate}
\end{thm}

\begin{proof}
 ((i) $\Leftrightarrow$ (ii)) 
By Lemma \ref{lem:equivariance}, $GL(1) \times \text{Spin}(p,q)$ acts on $V$ and $W$ as the vector representation and (a direct sum of several copies of ) the spin representation and $Q$ is $GL(1) \times \mathrm{Spin}(p,q)$-equivariant. 
Moreover the action of $GL(1) \times \mathrm{Spin}(p,q)$ on $V$ gives a prehomogeneous vector space and the open orbit $\Omega$ is given by $\{ v \in V | P(v) \neq 0\}$.  
Hence, by \cite[Lemma 6]{F.Sato4}, the condition (i) implies the condition (ii). 
The opposite implication is obvious. 

((iv) $ \Rightarrow$ (ii)) In the case where $(p,q)=(1,1)$ with $S_1=\pm S_2$, the vanishing of $\tilde P$ is obvious. In the other cases, $\tilde P$ is a relative invariant of an irregular prehomogeneous vector space with no relative invariants; hence $\tilde P$ should vanish. 
The details will be given in \S \ref{section:classification}.
 
 ((ii) $\Rightarrow$ (iii)) Assume that $\tilde{P} (w)$ vanishes identically. Then, by differentiating $\tilde{P} (w),$ we obtain
\[ 
\frac{\partial \tilde{P}}{\partial w_j}(w) 
 = 2 \sum_{i=1}^p S_i [w] \frac{\partial}{\partial w_j} S_i [w] 
 -2 \sum_{i=p+1}^{p+q} S_i [w] \frac{\partial}{\partial w_j} S_i [w] \equiv 0.
\]
Since $S_i [w]$ are non-degenerate, ${{\partial} \over {\partial w_j} } S_i [w]$ are non-zero linear forms. Hence the condition (iii) holds.
 
 ((iii) $\Rightarrow$ (iv)) First we consider the case $p=1 \geq q \geq 0.$ It is obvious that the condition in (iii) does not hold for $(p,q)=(1,0)$. 
 In the case $(p,q)=(1,1),$ $R_{1,1}$ has 4 inequivalent irreducible representations
\[
\left\{
 \begin{array}{c} 
 \rho_1 (e_1) =1 , \\
 \rho_1 (e_2) =1, 
 \end{array} 
 \right. 
 \quad 
 \left\{ 
 \begin{array}{c} 
 \rho_2 (e_1) =-1 , \\
 \rho_2 (e_2) =-1, 
 \end{array} 
 \right. 
 \quad 
\left\{ 
 \begin{array}{c} 
 \rho_3 (e_1) =1 , \\
 \rho_3 (e_2) =-1, 
 \end{array} 
 \right. 
 \quad 
\left\{ 
 \begin{array}{c} 
 \rho_4 (e_1) =-1 , \\
 \rho_4 (e_2) =1. 
 \end{array} 
 \right. 
\]
Let $\rho$ be a representation of $R_{1,1} $ and we write 
\[
\rho =k_1 \rho_1 \oplus k_2 \rho_2 \oplus k_3 \rho_3 \oplus k_4 \rho_4 \quad (k_1, k_2, k_3, k_4 \geq 0,\ k_1 +k_2 +k_3 +k_4 \geq 1).
\]
Then, putting 
\[
1_{k_1, k_2} = 
 \begin{pmatrix} 
 1_{k_1} & 0 \\
 0 & -1_{k_2} 
 \end{pmatrix}, \quad 
 1_{k_3, k_4} = 
 \begin{pmatrix} 
 1_{k_3} & 0 \\
 0 & -1_{k_4} 
 \end{pmatrix},
\]
the basis matrices are given by
\begin{equation}
\label{eqn:bm for 1_1}
S_1 =  \begin{pmatrix} 
 1_{k_1,k_2} & 0 \\
 0 & 1_{k_3, k_4} 
 \end{pmatrix}, \quad 
 S_2 = 
 \begin{pmatrix} 
 1_{k_1, k_2} & 0 \\
 0 & -1_{k_3, k_4} 
 \end{pmatrix}, 
\end{equation}
and we have  
\begin{gather*}
S_1 [w]=q_1 (u) +q_2 (v) , \quad S_2 [w]=q_1 (u) -q_2 (v) , \\
q_1 (u )=1_{k_1, k_2} [u ], \quad q_2 (v) =1_{k_3, k_4} [v] ,
\end{gather*} 
where $u \in {\Bbb R}^{k_1 +k_2}$, $v \in {\Bbb R}^{k_3 + k_4} $ and 
$w= \begin{pmatrix} u \\ v \end{pmatrix}$.  
Let $f_1 (w)=f_{11} (u ) +f_{12} (v) $ and 
$f_2 (w)=f_{21} (u) +f_{22} (v) $ be linear forms. Then, since
\begin{eqnarray*}
f_1(w) S_1 [w] +f_2 (w) S_2 [w] 
 &=& 
 q_1 (u)( f_{11}(u) +f_{21}(u)) +q_2 (v) (f_{11}(u) -f_{21}(u)) \\
 & & {} + q_1 (u) (f_{21}(v) +f_{22} (v) ) +q_2 (v) (f_{12} (v) -f_{22} (v)),
\end{eqnarray*} 
the identity 
\[
f_1 (w) S_1 [w] +f_2 (w) S_2 [w] \equiv 0
\]
implies that 
\begin{eqnarray*}
q_1(u) (f_{11}(u) +f_{21}(u)) 
 &\equiv& q_2 (v) (f_{11}(u) -f_{21}(u) )  \\
 &\equiv& q_1 (u) (f_{21}(v) +f_{22}(v) )  \\
 &\equiv& q_2 (v) (f_{12}(v) -f_{22}(v) )  \equiv 0. 
\end{eqnarray*} 
If at least one of $f_1 (w) $ and $f_2 (w)$ is not 0, then $q_1 (u) $ or $q_2 (v)$ vanishes, hence $S_1 = \pm S_2$.

Now we assume that $ p \geq 2$. 
In this case we need the following lemma on canonical forms of basis matrices. 

\begin{lem}
\label{lem:canonical form}
Assume that $ p \geq 2$ and let $\rho$ be a representation of $R_{p,q}$. 
Then the dimension $m$ of $\rho$ is even. 
Put $d=m/2$. 
Let $S_1=\rho(e_1), \dots , S_{p+q}=\rho(e_{p+q})$ be the basis matrices of $\rho$.
Then the signature of $S_i$ $(1 \leq i \leq p)$ is $(d,d)$ and (by replacing $\rho$ by an equivalent representation, if necessary,) we can take the basis matrices of the form
\begin{gather*}
S_1 = \begin{pmatrix}
1_d & 0 \\
0 & -1_d 
\end{pmatrix}, \quad  
S_i = \begin{pmatrix}
0 & B_i \\
{}^t B_i & 0 
\end{pmatrix} \  (2 \leq i \leq p), \\  
S_{p+j} = \begin{pmatrix} 
A_j & 0 \\
0 & A_j 
\end{pmatrix} \  (1 \leq j \leq q).
\end{gather*}
Here 
$A_1 , \dots , A_q $ are the basis matrices of some $d$-dimensional representation of $R_{0,q}=C_q$, 
$B_2 =1_d$ and $B_3,\ldots, B_p$ are orthogonal and skew symmetric matrices. 
Moreover they satisfy the commutation relation 
\begin{gather*}
A_i A_j =-A_j A_i\ (1 \leq i < j \leq q), \quad
B_i B_j =-B_j B_i\ (3 \leq i < j \leq p), \\ 
A_i B_j =B_j A_i \quad (1 \leq i \leq q,\ 3 \leq j \leq p). 
\end{gather*}
\end{lem}

\begin{proof}
Since $S_1^2=1_m$, we may assume that $S_1=1_{r,m-r}$. 
Then, by the anti-commutativity $S_iS_j=-S_jS_i$, 
the symmetric matrices $S_i$ $(2 \leq i \leq p)$ are of the form 
\[
S_i=\begin{pmatrix} 0_r & B_i \\ {}^tB_i & 0_{m-r} \end{pmatrix}, 
\quad B_i \in M(r,m-r;\R).
\]
Since $S_i^2=1_m$, we have $r=m-r$ and $B_i$ is orthogonal. 
This shows that $m$ is even and the signatures of $S_1,\ldots,S_p$ are equal to $(d,d)$ $(d=m/2)$.  
Put $\tilde B_2= B_2 \perp {}^t B_2$.  
By replacing all the $S_i$ by $\tilde B_2 S_i \tilde B_2^{-1}$, 
we may assume that $S_1=1_{d,d}$ and $B_2=1_d$.   
The commutation relation $S_2S_i=-S_iS_2$ $(3 \leq i \leq p)$ implies that $B_i$ is skew-symmetric. 
The remaining part of the lemma is also a straightforward consequence of the commutation relations $(\ref{eqn:clifford cond1})$ and $(\ref{eqn:clifford cond2})$.  
\end{proof}

Let us return to the proof of (iii) $\Rightarrow$ (iv) for $p \geq 2$.  
Then,  by Lemma \ref{lem:canonical form}, $m$ is even. 
Put $d=m/2$ and we may assume that $S_1, \dots , S_{p+q}$ are of the form given in Lemma \ref{lem:canonical form}. 
Then the identity (\ref{eqn:2.4}) takes the form
\begin{equation}
\label{eqn:2.6}
0 \equiv (u^2 - v^2) f_1 (w) + 2 \sum_{i=2}^p (u, B_i v) f_i (w) -\sum_{j=1}^q 
(A_j [u] +A_j [v] ) f_{p+j} (w),   
\end{equation}
where we put $w =\begin{pmatrix} u \\ v \end{pmatrix}$ $(u,v \in {\Bbb R}^d)$ and  $(u,v)={}^t u v$, $u^2 =(u,u)$, $v^2=(v,v)$. 

Consider the case where $p \geq 2$ and $q=0$. Then, the term that do not contain the variable $v$ (resp.\ $u$) is given by $u^2 f_1 (u,0)$ (resp. $v^2 f_1 (0,v)$).
Hence, by (\ref{eqn:2.6}), we have $f_1 (u,0)=f_1 (0,v)=0$ and $f_1 (w) =f_1 (u,0)+f_1 (0,v)=0$. This gives a shorter relation $S_2 [w] f_2 (w) + \cdots + S_p [w] f_p (w) =0.$ Therefore we have no 
non-trivial relations of the form (\ref{eqn:2.4}) in the case $q=0.$

Next we consider the case where $p \geq 2$ and $q=1$. 
We may assume that $f_1 (w) \neq 0.$ 
Then $f_1 (u,0) \neq 0$ or $f_1 (0,v) \neq 0.$ 
Let us assume further that $f_1 (0,v) \neq 0,$ 
since the other case can be treated quite similarly. 
Comparing the terms of degree 3 with respect to $v$ on the left- and right-hand sides of the identity
\begin{equation}
\label{eqn:2.7}
0 \equiv (u^2 - v^2) f_1 (w) +2 \displaystyle{\sum_{i=2}^p} (u, B_i v) f_i (w) 
-(A_1 [u] +A_1 [v]) f_{p+1} (w), 
\end{equation}
we obtain 
\[
-v^2 f_1 (0,v) -A_1 [v] f_{p+1} (0,v) =0.
\]
Since $v^2$ is an irreducible polynomial over ${\Bbb R}$ and $A_1^2=1_d$, we have
\begin{equation}
\label{eqn:2.8}
A_1 = \pm 1_d, \quad f_{p+1} (0,v) = \mp f_1 (0, v). 
\end{equation}
Now compare the terms of degree 2 and 1 with respect to $u$ and $v$,  respectively,  in $(\ref{eqn:2.7})$. 
Then we get 
\[
0=u^2 f_1 (0,v) +2 \sum_{i=2}^p (u, B_i v) f_i (u,0) -A_1 [u] f_{p+1} (0,v). 
\]
From (\ref{eqn:2.8}), it follows that 
\[
u^2 f_1 (0,v) = -\sum_{i=2}^p (u,B_i v) f_i (u,0).
\]
Since $f_1(0,v) \neq 0,$ we can  choose $v \in {\Bbb R}^d $ such that $ f_1 (0,v) >0$. 
Then, for such a  $v$, the quadratic form of  $u$ on the left-hand side is positive definite and of rank $d$.  
The number of positive eigenvalues of the quadratic form of $u$ on the right-hand side is not greater than $p-1.$ 
 Thus we get 
 \begin{center}
 $d \leq p-1$ and $m=2d \leq 2(p-1)$.
 \end{center}
 Note that $S_1,\dots,S_p$ defines a representation of the Clifford algebra $C_p$ and the dimension $d_0$ of irreducible representations of $C_p$ is given by the table,
\[
\begin{array}{|c|c|c|c|c|c|c|c|c|c|} \hline
 p & 2 & 3 & 4 & 5 & 6 & 7 & 8 & 9 & 10 \leq p \\\hline 
 d_0 & 2 & 4 & 8 & 8 & 16 & 16 & 16 & 16 & 2^{(p-1)/2} \leq d_0 \\\hline 
 2(p-1) & 2 & 4 & 6 & 8 & 10 & 12 & 14 & 16 & 2(p-1) \\\hline 
 \end{array}
\]
which we can get easily from Lemma \ref{lem:pos def Clif}.
 If $ p \geq 10,$ then $2^{(p-1)/2} > 2(p-1) .$ Hence, the possibilities of $ p \geq 2$ and $m$ for $q=1$ are 
\[
(2,1,2), \quad (3,1,4) , \quad (5,1,8), \quad (9,1,16), 
\]
which are precisely the cases for $p \geq 2$ and $q=1$ found on the list in (iii).
 
 Finally we consider the case for $ p,q \geq 2$. We may assume that at least one of $  
f_{p+1} (w) , \dots , f_{p+q} (w) $ is not equal to 0; otherwise we have non-trivial relation 
 (\ref{eqn:2.6}) for $q=0$. Look at the terms in (\ref{eqn:2.6}) including only $u$ or only $v$ to get 
\begin{gather*}
0=u^2 f_1 (u,0) - \sum_{j=1}^q A_j [u] f_{p+j}(u,0), \\
0=-v^2 f_1 (0,v) - \sum_{j=1}^q A_j [v] f_{p+j}(0,v).
\end{gather*}
One of these 2 identities should give a non-trivial relation for the $d$-dimensional representation of $R_{q,1}$ defined by 
$S_j =A_j \ (j=1, \dots , q)$ and $S_{q+1}=1_d$.  
By what we have proved for $q=1$, this implies that $(q,d)=(2,2),(3,4),(5,8),(9,16)$.  
If we change the role of $p$ and $q$, we also have $(p,d)=(2,2),(3,4),(5,8),(9,16)$.  Hence the possibilities of $(p,q,m)$ are 
\[
(2,2,4), \quad (3,3,8), \quad (5,5,16), \quad (9,9,32). 
\]
The case $(9,9,32)$ can not occur, since the dimension of irreducible representation of $R_{9,9}$ is equal to $2^8$ (see Lemma \ref{lem:T and T'}), much bigger than 32. 
Thus we have obtained the list of $(p,q,m)$ in (iii). 
\end{proof}

\begin{thm}
\label{thm:squared P}
If $(p,q,m)$ is equal to one of 
\begin{gather*}
(1,0,m)\ (m \geq 1), (2,0,2), (1,1,2), (3,0,4), (2,1,4), \\ 
(5,0,8), (4,1,8), (3,2,8), (9,0,16), (8,1,16), (5,4,16),
\end{gather*} 
then, there exists a quadratic form $q(w)$ on $W$ such that $\tilde P(w) = q(w)^2$ up to constant multiple.  
\end{thm}

\begin{proof}
Let $\rho$ be a degenerate representation of $R_{p,q}$ and $S_1,\ldots,S_{p+q}$ the basis matrices of $\rho$. 
Then, $(S_1[w]^2+\cdots+S_p[w]^2)-(S_{p+1}[w]^2+\cdots+S_{p+q}[w]^2)=0$, and 
\begin{gather*}
S_{p+q}[w]^2=(S_1[w]^2+\cdots+S_p[w]^2)-(S_{p+1}[w]^2+\cdots+S_{p+q-1}[w]^2), \\ 
-S_p[w]^2 = (S_1[w]^2+\cdots+S_{p-1}[w]^2)-(S_{p+1}[w]^2+\cdots+S_{p+q}[w]^2)
\end{gather*}
are the quartic forms obtained from the representations $\rho|_{R_{p,q-1}}$ and $\rho|_{R_{p-1,q}}$.  
Hence, the theorem is an immediate consequence of Theorem  \ref{thm:degenerate}. 
\end{proof}

\begin{remark} 
(1) We prove later in \S \ref{section:classification} that Theorem \ref{thm:squared P} gives the complete list of the representations for which the Clifford quartic form $\tilde P(w)$ is the square of a quadratic form. 

(2)   \label{remarlk:Ono}
In the cases above, the mapping $Q:w \mapsto (S_1[w],\ldots,S_{p+q}[w])$ is a  quadratic spherical map in the sense of Ono \cite[Chapter 5]{Ono}. In particular, the cases $(p,q,m)=(3,0,4)$, $(5,0,8)$, $(9,0,16)$ correspond to the Hopf fibrations $S^3\rightarrow S^2$, $S^7\rightarrow S^4$, $S^{15}\rightarrow S^8$, respectively. 
\end{remark}


\subsection{Explicit formulas for the local functional equations}

Now let us describe the explicit formula for the local functional equations satisfied by the local zeta functions attached to the Clifford quartic forms. 

\begin{lem} 
\label{lemma:2.10}
For $p \geq q \geq 0$ with $p+q >0$, we put 
$$ \Omega=\{ x \in {\Bbb R}^{p+q} |P(x) \neq 0 \} , \quad P(x)=\sum_{i=1}^p x_i^2 -\sum_{j=1}^q x_{p+j}^2 .$$
Then, the decomposition of $\Omega $ into connected components is given as follows:
\begin{enumerate}
\def\labelenumi{$(\arabic{enumi})$}
\item 
If $(p,q)=(1,0),$ then we have $\Omega= \Omega_+ \cup \Omega_- , $ where 
$$ \Omega_+ =\{ x \in {\Bbb R} | x < 0 \} , \quad \Omega_- =\{ x \in {\Bbb R} | x <0 \} .$$
\item
If $(p,q)=(1,1),$ then we have $\Omega =\Omega_{++} \cup \Omega_{+-} \cup \Omega_{-+} \cup \Omega_{--} , $ where 
\begin{gather*} 
\Omega_{++}=\{ (x_1, x_2) \in {\Bbb R}^2 | P(x) >0 , x_1 >0 \} ,  
 \quad \Omega_{+-}=\{ (x_1, x_2) \in {\Bbb R}^2 | P(x) >0 , x_1 <0 \}, \\
\Omega_{-+}=\{ (x_1,x_2) \in {\Bbb R}^2 | P(x) <0 , x_2 >0 \} , 
 \quad \Omega_{--}=\{ (x_1, x_2) \in {\Bbb R}^2 | P(x) <0, x_2 <0 \}.
\end{gather*} 
\item 
If $(p,q)=(p,0)$ with $p \geq 2$, then $\Omega$ is connected and we have $\Omega=\Omega_+$.
\item 
If $(p,q)=(p,1)$ with $p \geq 2$, then we have $\Omega =\Omega_+ \cup \Omega_{-+} \cup \Omega_{--}$, 
where 
\begin{gather*} 
\Omega_{+}=\{ x \in {\Bbb R}^{p+1} | P(x) >0 \} , \\
\Omega_{-+}=\{ x \in {\Bbb R}^{p+1} | P(x) <0 , x_{p+1} >0 \} , \quad \Omega_{--}=\{ x \in {\Bbb R}^{p+1} | P(x) <0 , x_{p+1} <0 \}.
\end{gather*}
\item  
If $p,q \geq 2$, then we have $\Omega =\Omega_+ \cup \Omega_-, $ where 
$$
\Omega_+ =\{ x \in {\Bbb R}^{p+q} | P(x) >0 \} , \quad \Omega_- =\{ x \in {\Bbb R}^{p+q} | P(x) <0 \}.
$$
\end{enumerate}
\end{lem}

For $v \in {\Bbb R}^{p+q} $ with $P(v) \neq 0$, the signature of the symmetric matrix $S(v) = \sum_{i=1}^{p+q} v_i S_i$ depends only on the connected component to which $v$ belongs. 
For each connected component $\Omega_{\eta}$, $\Omega_{\eta\tau}$ $(\eta , \tau = \pm 1)$ we define $\gamma=\gamma_\eta, \gamma_{\eta\tau}$ by 
\[
\gamma =\mathbf{e}\left[ \frac{\sigma_+ - \sigma_-}8 \right], 
\]
where $ \sigma_+$ and $\sigma_-$, respectively, are the numbers of positive and negative eigenvalues of $S(v)$ for a point $v \in \Omega_{\eta\tau}.$ 
An explicit formula for the constants $\gamma$ is given by the following lemma. 

\begin{lem}
\label{lemma:2.11}
Assume that $p \geq q \geq 0$. 
\begin{enumerate}
\def\labelenumi{$(\arabic{enumi})$}
\item 
If $(p,q)=(1,0),$ then $\gamma_\pm=\mathbf{e}\left[ \frac{\pm (k_+ -k_-) }8\right]$, where $k_{\pm}=\dim\set{w\in W}{\rho(e_1)w=\pm w}$.
\item 
If $(p,q)=(1,1)$,  then 
\[
\gamma_{\eta, \tau}= \mathbf{e}\left[\frac{\tau ( k_{++} -k_{--}) + \tau\eta (k_{+-} -k_{-+}) }4 \right],
\]
where $k_{\sigma_1 \sigma_2}=\dim\set{w \in W}{\rho(e_1)w=\sigma_1 w,\ \rho(e_2)w=\sigma_2 w}$ for $\sigma_1, \sigma_2 \in \{\pm\}$. 
\item 
If $ p \geq 2, q=0,$ then $\gamma =1$. 
\item 
If $p \geq 2, q=1$,  then 
\[ 
\gamma_+ =1,\quad  \gamma_{-, \pm} = 
\left\{ \begin{array}{cc} 
(\sqrt{-1})^{\pm (k_+ -k_- )} & (p=2), \\
(-1)^{k_+ -k_-} & ( p=3),  \\
1 & ( p \geq 4),
\end{array} \right. 
\]
where $k_+$ (resp.\ $k_-$) is the multiplicity in $\rho$ of the irreducible representations of $R_{p,1}$ for which $e_{p+1}$ acts as multiplication by $+1$ (resp.\ $-1$).
\item 
If $p \geq q \geq 2, $ then $\gamma_+= \gamma_- =1$. 
\end{enumerate}
\end{lem}

\begin{proof} 
The constant $\gamma$ does not depend on the choice of a representative $v$ of $\Omega_{\eta\tau}$. 
Hence we may take $v={}^t(\pm1,0,\ldots,0)$ or $v={}^t(0,\ldots,0,\pm1)$. 
Then $S(v)=\pm S_1$ or $\pm S_{p+q}$. 

(1) For the above choice of $v$, we have $S(v)=\pm S_1=\pm 1_{k_+.k_-}$.  
This implies the first assertion.

(2) The second assertion is an immediate consequence of  $(\ref{eqn:bm for 1_1})$. 

(3) If $p \geq 2$, by Lemma \ref{lem:canonical form}, 
the signature of $S_1$ is $({m \over 2} , {m \over 2})$. 
Hence we have $ \varepsilon =1$. 

(4) The proof of $\varepsilon_+=1$ is quite the same as that for (3). 
Let $d$ be the dimension of irreducible representations of $C_p$. 
(Note that all the irreducible representations are of the same dimension.) 
Then $S_{p+1}=I_{dk_+, dk_-}$ and we have $\sigma_+-\sigma_- = d(k_+ -k_-)$. 
Since $d=2 $ for $p=2$, $d=4$ for $p=3$, and $8 | d$ for $ p \geq 3$,  this implies the fourth assertion. 

(5) The last assertion is obvious from Lemma \ref{lem:canonical form}. 
\end{proof}

\begin{lem}
\label{lem:alpha beta}
The constants $\alpha,\beta$ defined in Remark 1.1 (1) are given by
\[
\alpha=\pm1, \quad \beta=(-1)^{q+1}.
\]
\end{lem}

\begin{proof}
By $(\ref{eqn:clifford cond1})$ and $(\ref{eqn:clifford cond2})$, we have 
\[
S(v)^2=\sum_{i=1}^n v_i^2S_i^2+\sum_{i<j} v_iv_j(S_iS_j+S_jS_i) = P(v)1_m.
\]
Hence, $\det S(v)^2=P(v)^m$ and $\det S(v)=\pm P(v)^{m/2}$. 
This proves that $\alpha=\pm1$. 
The Jacobian 
\begin{eqnarray*}
\det\left(\frac{\partial \phi(v)_i}{\partial v_j}\right)
 &=& \det \left\{
           \frac{1}{P(v)} \begin{pmatrix} 
                  \varepsilon_1  & 0 & \cdots & 0 \\
                  0 & \varepsilon_2  & \ddots  & \vdots \\
                  \vdots  & \ddots  & \ddots  &  0 \\
                  0 & \cdots  & 0  & \varepsilon_n \\
                  \end{pmatrix} \right. \\
      & &      -  \frac{2}{P(v)^2}  \left. \begin{pmatrix} 
                  (\varepsilon_1v_1)^2 & (\varepsilon_1v_1)(\varepsilon_2v_2) 
                            & \cdots & (\varepsilon_1v_1)(\varepsilon_n v_n) \\
                  (\varepsilon_2v_2)(\varepsilon_1v_1) & (\varepsilon_2v_2)^2 & \ddots  & \vdots \\
                  \vdots  & \ddots  & \ddots  &  (\varepsilon_{n-1}v_{n-1})(\varepsilon_n v_n) \\
                  (\varepsilon_n v_n)(\varepsilon_1v_1) & \cdots  & (\varepsilon_n v_n)(\varepsilon_{n-1} v_{n-1})  & (\varepsilon_n v_n)^2 
                  \end{pmatrix} \right\}
\end{eqnarray*}
is an $SO(p,q)$-invariant homogeneous rational function of degree $-2n$ and is equal to $\beta P(v)^{-n}$, where $\varepsilon_i$'s are given by $(\ref{eqn:sign})$. 
We can obtain $\beta$ easily by taking $v={}^t(1,0,\ldots,0)$.   
\end{proof}

We define the local zeta functions for the quadratic form $P(v)$ by (the analytic continuations of) the integrals
\[
\zeta_\pm(s,\Phi)=\int_{\Omega_\pm}\abs{P(v)}^s\Phi(v)\,dv, \quad
\zeta_{\pm\pm}(s,\Phi)=\int_{\Omega_{\pm\pm}}\abs{P(v)}^s\Phi(v)\,dv 
\quad (\Phi \in \calS(\R^n)).
\]

\begin{lem}
\label{lem:LFE for quad form}
For any $\Phi \in \calS(\R^{n})$, the following functional equations hold:
\\
$(1)$ Assume that $(p,q)=(n,0)$. 
\[
\zeta_+\left(s,\hat \Phi\right) 
 = -\pi^{-(2s+n/2+1)} 
        \Gamma(s+1) \Gamma\left(s+\frac n2\right) 
        \sin\left(s\pi\right) \zeta_+\left(-s-\frac n2,\Phi\right).
\]
$(2)$ Assume that $(p,q)=(n-1,1)$. 
\begin{eqnarray*}
\begin{pmatrix}
\zeta_+\left(s,\hat \Phi\right) \\
\zeta_{-+}\left(s,\hat \Phi\right) \\
\zeta_{--}\left(s,\hat \Phi\right) 
\end{pmatrix}
 &=& \pi^{-(2s+n/2+1)} 
        \Gamma(s+1) \Gamma\left(s+\frac n2\right)  \\
 & & \times 
\begin{pmatrix}
-\cos\left(s\pi\right) & -\cos\left(\frac{n\pi}2\right)  & -\cos\left(\frac{n\pi}2\right) \\[0.5ex]
\frac12 & \frac12\mathbf{e}\left[-\frac{2s+n}4\right]    & \frac12\mathbf{e}\left[\frac{2s+n}4\right]   \\[0.5ex]
\frac12  & \frac12\mathbf{e}\left[\frac{2s+n}4\right]   & \frac12\mathbf{e}\left[-\frac{2s+n}4\right]   
\end{pmatrix}
\begin{pmatrix}
\zeta_+\left(-s-\frac n2,\Phi\right) \\
\zeta_{-+}\left(-s-\frac n2,\Phi\right) \\
\zeta_{--}\left(-s-\frac n2,\Phi\right) 
\end{pmatrix}.
\end{eqnarray*}
$(3)$ Assume that $p,q\geq1$ and put $n=p+q$. 
\begin{eqnarray*}
\begin{pmatrix}
\zeta_+\left(s,\hat\Phi\right) \\
\zeta_{-}\left(s,\hat\Phi\right) 
\end{pmatrix}
 &=& \pi^{-(2s+n/2+1)} 
        \Gamma(s+1) \Gamma\left(s+\frac n2\right) \\
 & & \times 
\begin{pmatrix}
-\sin\left(\frac{(2s+q)\pi}2\right) & \sin\left(\frac{p\pi}2\right) \\
\sin\left(\frac{q\pi}2\right) & -\sin\left(\frac{(2s+p)\pi}2\right)
\end{pmatrix}
\begin{pmatrix}
\zeta_+\left(-s-\frac n2,\Phi\right) \\
\zeta_{-}\left(-s-\frac n2,\Phi\right) 
\end{pmatrix}.
\end{eqnarray*}
\end{lem}

\begin{proof}
The first and the third functional equations are well-known (see, e.g., \cite[\S 4.2]{KimuraBook}). The second functional equation can be derived from the two-variable functional equations given in \cite[\S, Theorem 2]{Muro} or \cite[Theorem 3.6]{F.Sato2} by specialization of a variable. (It is also contained in \cite{Suzuki1}.) 
\end{proof}

Now we have all the necessary data for the description of the local functional equations satisfied by the Clifford quartic forms. 
 
Let $\rho$ be a non-degenerate representation of $R_{p,q}$ on $W=\R^m$ and $Q:W \rightarrow V$ the associated non-degenerate self-dual quadratic mapping. 
Let $S_1, S_2, \dots , S_{p+q}$ be the basis matrices of $\rho$. 
Then the associated Clifford quartic form is given by $\tilde P(w)=S_1[w]^2+\cdots+S_p[w]^2-S_1[w]^2-\cdots-S_{p+q}[w]^2$. 
Put $\tilde\Omega_\pm=Q^{-1}(\Omega_\pm)$ and $\tilde\Omega_{\pm,\pm}=Q^{-1}(\Omega_{\pm\pm})$. 
For $\Psi \in \calS(W)$, we define the local zeta functions by 
\[
\begin{cases}
\tilde\zeta_\pm(s,\Psi)
 = \int_{\tilde\Omega_\pm} \abs{\tilde P(w)}^s \Psi(w)\,dw 
 & (p \geq q \ne 1), \\
\tilde\zeta_{\pm\pm}(s,\Psi) 
 = \int_{\tilde\Omega_{\pm\pm}} \abs{\tilde P(w)}^s \Psi(w)\,dw 
 & (p \geq q = 1).
\end{cases}
\]
By the general theory, the local zeta functions $\tilde\zeta_\pm(s;\Psi)$, $\tilde\zeta_{\pm\pm}(s;\Psi)$ can be continued to meromorphic functions of $s$ in $\C$ and the functional equations of these local zeta functions are immediate consequences of Lemmas \ref{lemma:2.11}, \ref{lem:alpha beta}, \ref{lem:LFE for quad form} and Theorem \ref{thm:main}. 
For simplicity, we give explicit formulas for the local functional equations under the assumption that 
\begin{equation}
\text{``the constants $\gamma$ are equal to $1$ and $m = \dim W \geq 8$.''}
\label{eqn:trivial gamma}
\end{equation}

\begin{thm} 
\label{thm:2.14}
Under the assumption $(\ref{eqn:trivial gamma})$, 
the local zeta functions $\tilde\zeta_\pm(s,\Psi)$, $\tilde\zeta_{\pm\pm}(s,\Psi)$ $(\Psi \in \calS(W))$ satisfy the following local functional equations.

If $(p,q)=(n,0)$, then 
\begin{eqnarray*}
\tilde\zeta_+\left(s,\hat\Psi\right)
 &=& 2^{4s+m/2}\pi^{-4s -2-m/2} \Gamma (s+1) 
          \Gamma \left(s+\frac n2 \right) \Gamma\left (s+1+\frac{m-2n}4 \right) 
          \Gamma \left(s+\frac m4\right)  \\
 & &   {}\times \sin(\pi s) \sin \pi\left(s-\frac{n}2\right) 
                  \tilde\zeta_+\left(-\frac m4 - s,\Psi\right). 
\end{eqnarray*}

If $p > q = 1$, then 
\begin{eqnarray*}
\begin{pmatrix}
\tilde\zeta_+\left(s,\hat\Psi\right)  \\
\tilde\zeta_{-+}\left(s,\hat\Psi\right)  \\
\tilde\zeta_{--}\left(s,\hat\Psi\right)  
\end{pmatrix}
 &=& 2^{4s+m/2}\pi^{-4s -2-m/2} \Gamma (s+1) 
          \Gamma \left(s+\frac n2 \right) \Gamma\left (s+1+\frac{m-2n}4 \right) 
          \Gamma \left(s+\frac m4\right)  \\
 & &   {}\times  \sin \pi s 
\begin{pmatrix}
 -\sin \pi\left(s-\frac {n }2\right) 
& 0 
& 0 \\ 
-\sin \left(\frac {n \pi}2\right)   
& -\sin \pi\left(s+\frac {n }2\right)    
& 0    \\ 
-\sin \left(\frac {n \pi}2\right) 
& 0 
& -\sin \pi\left(s+\frac {n }2\right) 
\end{pmatrix} 
                  \begin{pmatrix}
                  \tilde\zeta_+\left(-\frac m4 - s,\Psi\right)  \\
                  \tilde\zeta_{-+}\left(-\frac m4 - s,\Psi\right)  \\
                  \tilde\zeta_{--}\left(-\frac m4 - s,\Psi\right)  
                   \end{pmatrix}. 
\end{eqnarray*}

If $p \geq q \geq 2$, then 
\begin{eqnarray*}
\begin{pmatrix}
\tilde\zeta_+\left(s,\hat\Psi\right)  \\
\tilde\zeta_-\left(s,\hat\Psi\right)  
\end{pmatrix}
 &=& 2^{4s+m/2}\pi^{-4s -2-m/2} \Gamma (s+1) 
          \Gamma \left(s+\frac n2 \right) \Gamma\left (s+1+\frac{m-2n}4 \right) 
          \Gamma \left(s+\frac m4\right)  \\
 & &   {}\times  \sin \pi s \begin{pmatrix}
     \sin\pi\left(s+\frac{q-p}2\right) 
        & -2\sin\frac{\pi p}2 \cos\frac{\pi q}2  \\
     -2\sin\frac{\pi q}2 \cos\frac{\pi p}2
       & \sin\pi\left(s+\frac{p-q}2\right) 
     \end{pmatrix}
                  \begin{pmatrix}
                  \tilde\zeta_+\left(-\frac m4 - s,\Psi\right)  \\
                  \tilde\zeta_-\left(-\frac m4 - s,\Psi\right)  
                   \end{pmatrix}. 
\end{eqnarray*}
\end{thm}

\begin{remark}
The non-degenerate cases that are excluded by the assumption  $(\ref{eqn:trivial gamma})$ are
\[
(p,q)=(1,0),(1,1),(2,1),(3,1)\quad \text{and} \quad (p,q,m)=(3,0,4).
\]
In these cases, the Clifford quartic forms are relative invariants of well-known prehomogeneous vector spaces (see the classification in \S 4).  
\end{remark}


\section{The structure of the groups of symmetries of the Clifford quartic forms}

Let $\rho$ be a non-degenerate representation of $R_{p,q}$ and $\tilde P(w)$ the associated Clifford quartic form. 
Then, $\tilde P(w)$ is not a relative invariant of any prehomogeneous vector space except for some low-dimensional cases and the local functional equation satisfied by $\tilde P(w)$ in Theorem \ref{thm:2.14} is not covered by the theory of prehomogeneous vector spaces. 
To see this, we need to know the structure of the group 
\[
G_{p,q}(\rho) :=\set{ g \in GL(W)}{\tilde{P} ( g \cdot w ) =\tilde{P}(w)}, 
\]
the group of symmetries of $\tilde P$,  
and to determine when $(GL(1)\times G_{p,q}(\rho),W)$ is a prehomogeneous vector space. 
If $(GL(1)\times G_{p,q}(\rho),W)$ is not a prehomogeneous vector space, then there exist no prehomogeneous vector spaces with $\tilde P$ as relative invariant. 

Let $\mathfrak{g}_{p,q}(\rho)$ be the Lie algebra of $G_{p,q}(\rho)$. 
Differentiating the identity $\tilde{P} (\exp(tX) \cdot w) = \tilde{P} (w)$ $(X \in \mathfrak{gl}(W))$, we have 
\[
\gerg_{p,q}(\rho)=\set{X\in\mathfrak{gl}(W)}{\sum_{i=1}^{p+q} \varepsilon_i S_i [w] ({}^t X S_i +S_i X) [w]=0}, 
\]
where $\varepsilon_i$'s are as in $(\ref{eqn:sign})$.

First note that the group $\mathrm{Spin}(p,q)$ is contained in $G_{p,q}(\rho)$, since, by Lemma \ref{lem:equivariance}, we have 
\[
\tilde P(gw)=P(Q(gw))=P(gQ(w))=P(w)\quad (g \in \mathrm{Spin}(p,q)). 
\]
Hence the Lie algebra $\gerk_{p,q}$ $({}\cong \gers\gero(p,q))$ of $\mathrm{Spin}(p,q)$ is contained in $\gerg_{p,q}(\rho)$. 
There exists another group contained in $G_{p,q}(\rho)$.  
Put 
\[
H_{p,q}(\rho):=\set{g \in GL(W)}{S_i[gw]=S_i[w]\ (1 \leq i \leq p+q)} 
 = \bigcap_{i=1}^{p+q} O(S_i).
\]
Then, it is obvious that $H_{p,q}(\rho)$ is also contained in $G_{p,q}(\rho)$.  
Since $g \mapsto S_i\,{}^t\!g^{-1}S_i$ $(1 \leq i \leq p+q)$ are involutions commuting with each other, $H_{p,q}(\rho)$ is a reductive Lie group. 
The Lie algebra $\mathfrak{h}_{p.q} (\rho)$ of  $H_{p,q}(\rho)$ is given by
\[
\mathfrak{h}_{p.q} (\rho) 
 = \set{X \in \mathfrak{gl}(W)}{{}^t X S_i +S_i X=0 \ (i=1, \dots , p+q)}.
\]

\begin{thm}
\label{thm:3.1}
Let $\rho$ be an $m$-dimensional representation of $R_{p,q}$. 
Then, we have 
\[
\gerg_{p,q} (\rho) \cong \gers\gero(p,q) \oplus \gerh_{p,q}(\rho)
\]
except for the following low dimensional cases. 
\[
\begin{array}{|c|c|c|c|c|c|c|c|c|c|c|} \hline 
p+q &2 & 3 & 4 & 5 & 6 & 7 & 8 & 9 & 10 & 11  \\\hline 
m & \text{\bf pure} & \mathbf{2},4 & \mathbf{4},8 & 8 & \mathbf{8},16 & 16 & 16 & 16 & \mathbf{16},32 & 32 \\\hline 
\end{array} 
\]
(The cases written with boldface letters are degenerate cases (see Theorem \ref{thm:degenerate}).)
\end{thm} 

The following lemma gives a sufficient condition for $\gerg_{p,q} (\rho)$ to be isomorphic to $\gers\gero(p,q) \oplus \gerh_{p,q}(\rho)$. 

\begin{lem}
\label{lem:sufficient}
If the basis matrices $S_1, S_2, \dots , S_{p+q}$ of a representation $\rho$ of $R_{p,q}$ satisfy the condition $(\sharp)$ below, 
then $\gerg_{p,q} (\rho)$ is isomorphic to $\gers\gero(p,q) \oplus \gerh_{p,q}(\rho)$:
\begin{description}
\addtocounter{equation}{1}
\item[$(\sharp)$]   if\/  $\displaystyle{\sum_{i=1}^{p+q}} S_i [w]  X_i [w] \equiv 0$ for $X_1 , \dots , X_{p+q} \in {\rm Sym}(m;{\Bbb R})$ $(m=\mathrm{deg}\,\rho)$, 
 then we have $X_i = \displaystyle{\sum_{j=1}^{p+q}}  a_{ij} S_j $ $(i=1,2, \dots , p+q)$
for some $a_{ij}$ with $a_{ij} =-a_{ji}$.
\end{description}
\end{lem}

\begin{proof}
For $X \in \gerg_{p,q}(\rho)$, the assumption $(\sufcond)$ implies that   
\[
{}^t X S_i +S_i X = \sum_{j=1}^{p+q} a_{ij} S_j \ (i=1, \dots , p+q), \quad a_{ij}=-\varepsilon_i \varepsilon_j a_{ji} \ ( 1 \leq i <j \leq p+q). 
\]
The coefficients $a_{ij}$ defines an element $(a_{ij})$ in $\mathfrak{so}(p,q)$ and the mapping 
\[
f : \mathfrak{g}_{p,q} ( \rho )  \longrightarrow \mathfrak{so}(p,q) , \quad X \mapsto (a_{ij} ) 
\]
gives a Lie algebra homomorphism.  The mapping $\mathfrak{so}(p,q) \ni (a_{ij}) \mapsto \sum_{i<j} a_{ij} S_iS_j \in \rho(\gerk_{p,q})$ is a section of $f$ (see Lemma \ref{lem:equivariance}).  
It is obvious that
\[
\mathrm{Ker}(f) =\{ X \in \mathfrak{g}_{p,q}( \rho ) | \ {}^t X S_i +S_i X = 0 \ (1 \leq i \leq p+q) \} =\mathfrak{h}_{p,q}(\rho).  
\]
This proves the isomorphism $ \mathfrak{g}_{p,q} ( \rho) \cong \mathfrak{so}(p,q) \oplus \mathfrak{h}_{p,q} (\rho)$. 
\end{proof}

Theorem \ref{thm:3.1} is an immediate consequence of  
Lemma \ref{lem:sufficient} and the following lemma.
 
\begin{lem}
\label{lem:key}
The condition $(\sufcond)$ holds except for the following cases.
\[
\begin{array}{|c|c|c|c|c|c|c|c|c|c|c|} 
\hline 
p+q &2 & 3 & 4 & 5 & 6 & 7 & 8 & 9 & 10 & 11  \\
\hline 
m & \text{pure} & 2,4 & 4,8 & 8 & 8,16 & 16 & 16 & 16 & 16,32 & 32 \\
\hline 
\end{array} 
\]
\end{lem}

We postpone the proof of Lemma \ref{lem:key} until \S \ref{section:proof of key} and consider the Lie algebra $\gerh_{p,q}(\rho)$. 
The structure of $\gerh_{p,q} ( \rho )$ depends on the structure of the algebras $R_{p,q}$ and $R_{p,q}^+$ given in Lemma \ref{lem:T and T'}.

\begin{thm}
\label{thm:3.4} 
The Lie algebra $\gerh_{p,q}(\rho) $ is isomorphic to the reductive Lie algebra given in the following table: 
\[
\def\arraystretch{1.1}
\begin{array}{|c|c|c|c|}
\hline 
(R_{p,q}, R_{p,q}^+) & (\mathbb K, \mathbb K') & \gerh_{p,q}(\rho) & \{p \bmod 8, q \bmod 8\} \\
\hline 
 &  (\R,\C) & \mathfrak{so}(k,\C) &  \{0,2\}, \{4,6\} \\
\cline{2-4}
\smash{\lower3mm\hbox{$(T, T')$}} & (\C,\R) & \mathfrak{sp}(k,\R) & \{0,7\}, \{2,3\}, \{3,4\}, \{6,7\} \\
\cline{2-4}
 &   (\C,\mathbb H) & \mathfrak{so}^*(2k) &  \{0,3\}, \{2,7\}, \{3,6\}, \{4,7\} \\
\cline{2-4}
 &   (\mathbb H,\C) & \mathfrak{sp}(k,\C) &  \{0,6\}, \{2,4\} \\
\hline
\smash{\lower3mm\hbox{$(T,2T')$}} &   (\R,\R) & \mathfrak{gl}(k,\R) &  \{0,0\}, \{2,2\}, \{4,4\}, \{6,6\} \\
\cline{2-4}
 &   (\mathbb H,\mathbb H) & \mathfrak{gl}(k,\mathbb H) &  \{0,4\}, \{2,6\} \\
\hline
 &   (\R,\R) & \mathfrak{so}(k_1,k_2) &  \{0,1\}, \{1,2\}, \{4,5\}, \{5,6\} \\
\cline{2-4}
(2T,T') &   (\C,\C) & \mathfrak{u}(k_1,k_2) &  \{1,3\}, \{1,7\}, \{3,5\}, \{5,7\} \\
\cline{2-4}
 &   (\mathbb H,\mathbb H) & \mathfrak{sp}(k_1,k_2) &  \{0,5\}, \{1,4\}, \{1,6\}, \{2,5\} \\
\hline
\smash{\lower3mm\hbox{$(2T,2T')$}} &   (\C,\R) & \mathfrak{sp}(k_1,\R)\oplus \mathfrak{sp}(k_2,\R) & \{3,3\}, \{7,7\} \\
\cline{2-4}
 &   (\C,\mathbb H) & \mathfrak{so}^*(2k_1)\oplus \mathfrak{so}^*(2k_2) &  \{3,7\} \\
\hline
\smash{\lower3mm\hbox{$(4T,2T')$}} &   (\R,\R) & \mathfrak{so}(k_1,k_2)\oplus \mathfrak{so}(k_3,k_4) & \{1,1\}, \{5,5\} \\
\cline{2-4}
 &   (\mathbb H,\mathbb H) & \mathfrak{sp}(k_1,k_2)\oplus \mathfrak{sp}(k_3,k_4) &  \{1,5\} \\
\hline
\end{array}
\]
Here $k,k_1,k_2,k_3,k_4$ are the multiplicities of irreducible representations in $\rho$.  
More precisely, if $R_{p,q}=T$, then $R_{p,q}$ has only one irreducible representation  $\rho_1$ and $k$ is the multiplicity of $\rho_1$ in $\rho$;  
if $R_{p,q}=T \oplus T$, then $R_{p,q}$ has two irreducible representations $\rho_1,\rho_2$ and $k_1,k_2$ are the multiplicities of $\rho_1,\rho_2$ in $\rho$;  
if $R_{p,q}=T \oplus T \oplus T \oplus T$, then $R_{p,q}$ has four irreducible representations $\rho_1,\rho_2,\rho_3,\rho_4$ and $R_{p,q}^+$ has two irreducible representations (the  even and the odd half-spin representations) $\Lambda_e$, $\Lambda_o$. We may assume that $\rho_1|_{R^+_{p,q}}=\rho_2|_{R^+_{p,q}}=\Lambda_e$ and $\rho_3|_{R^+_{p,q}}=\rho_4|_{R^+_{p,q}}=\Lambda_o$. 
Then, 
$k_1,k_2,k_3,k_4$ are the multiplicities of $\rho_1,\rho_2,\rho_3,\rho_4$ in $\rho$.
\end{thm}

\begin{rem}\rm 
In \cite[Proposition 4.4.4, Table 4.4.3]{KobayashiYoshino}, Kobayashi and Yoshino gave a list of classical groups, which is quite similar to that in Theorem \ref{thm:3.4}.  
However the relation between these two lists is not clear. 
\end{rem}

We shall prove Theorem \ref{thm:3.4} in \S \ref{section:h_{p.q}}. 
From the proof, we can easily read how the Lie algebra $\mathfrak{so}(p,q)\oplus\gerh_{p,q}(\rho)$ acts on $W$. 
To describe the result, we need some notational preliminaries. 
We denote by $\Lambda_1$ the standard representation of $\mathfrak{gl}(k,\mathbb K)$ on $\mathbb K^k$ for $\mathbb K = \R,\C,\H$. 
We also denote by $\Lambda_1$ the representations of the classical Lie algebras obtained from the following natural inclusions
\begin{gather*}
\mathfrak{so}(k_1,k_2) \hookrightarrow \mathfrak{gl}(k_1+k_2,\R), \quad
\mathfrak{so}(k,\C) \hookrightarrow \mathfrak{gl}(k,\C), \quad
\mathfrak{so}^*(2k) \hookrightarrow \mathfrak{gl}(k,\H), \\
\mathfrak{sp}(k_1,k_2) \hookrightarrow \mathfrak{gl}(k_1+k_2,\H), \quad
\mathfrak{sp}(k,\R) \hookrightarrow \mathfrak{gl}(2k,\R), \quad
\mathfrak{sp}(k,\C) \hookrightarrow \mathfrak{gl}(2k,\C), \\
\mathfrak{u}(k_1,k_2) \hookrightarrow \mathfrak{gl}(k_1+k_2,\C).
\end{gather*}
In case $R^+_{p,q}$ is a simple algebra, then we denote by $\Lambda$ the representation of $\gerk_{p,q}\cong \mathfrak{so}(p,q)$ induced by the unique irreducible representation of $R^+_{p,q}$. Namely $\Lambda$ is the spin representation. 
In case  $R^+_{p,q}$ is the direct sum of two simple algebras, then we denote by $\Lambda_e$ and $\Lambda_o$ the irreducible representations of  $\gerk_{p,q}\cong \mathfrak{so}(p,q)$ induced by the two irreducible representations of $R^+_{p,q}$. Namely $\Lambda_e$  and $\Lambda_o$ are the even and odd half-spin representations. 
Moreover we denote by $1$ the trivial representation of a Lie algebra.

\begin{thm}
\label{thm:3.5}
Put $\gerg'_{p,q}(\rho)=\mathfrak{so}(p,q)\oplus\gerh_{p,q}(\rho)$. 
Then 
$\mathfrak{g}'_{p,q} ( \rho)$ acts on the representation space $W$ of $\rho$ as follows.
{\small
\[
\def\arraystretch{1.2}
\hspace*{-0.3cm}
\begin{array}{|c|c|c|c|} \hline 
(R_{p,q} , R_{p,q}^+) & ({\Bbb K}, {\Bbb K}') & \{p \bmod 8, q \bmod 8\}   & \text{\em Representation of $ \mathfrak{g}'_{p,q}{(\rho)}$} \\
\hline 
 & ({\Bbb R}, {\Bbb C} ) &   \{0,2\}, \{4,6\}  & ( \mathfrak{so}(p,q) \oplus \mathfrak{so}(k,\C), \Lambda \otimes \Lambda_1)\\ \cline{2-4} 
 \smash{\lower5mm\hbox{$(T, T')$}}  & ({\Bbb C}, {\Bbb R} )  &  
\begin{array}{c} \{0,7\}, \{2,3\} \\[-1ex] \{3,4\}, \{6,7\} \end{array}  & ( \mathfrak{so}(p,q) \oplus \mathfrak{sp}(k,\R) , \Lambda \otimes \Lambda_1 ) \\ 
\cline{2-4} 
       & ({\Bbb C}, {\Bbb H} ) & \begin{array}{c} \{0,3\}, \{2,7\} \\[-1ex] \{3,6\}, \{4,7\}  \end{array} & ( \mathfrak{so}(p,q) \oplus \mathfrak{so}^*(2k) , \Lambda \otimes \Lambda_1 ) \\ 
\cline{2-4} 
       & ({\Bbb H}, {\Bbb C} ) &  \{0,6\}, \{2,4\}  & ( \mathfrak{so}(p,q) \oplus \mathfrak{sp}(k,\C), \Lambda \otimes \Lambda_1) \\ 
\hline
\smash{\lower3mm\hbox{$(T,2T')$}}   & ({\Bbb R}, {\Bbb R})  & \begin{array}{c} \{0,0\}, \{2,2\}  \\[-1ex]  \{4,4\}, \{6,6\}  \end{array} & ( \mathfrak{so}(p,q) \oplus \mathfrak{gl}(k,\R) , \Lambda_e \otimes \Lambda_1 + \Lambda_o \otimes \Lambda_1^*) \\
\cline{2-4} 
                   & ({\Bbb H}, {\Bbb H})  &  \{0,4\}, \{2,6\}  & ( \mathfrak{so}(p,q) \oplus \mathfrak{gl}(k,\H) , \Lambda_e \otimes \Lambda_1 + \Lambda_o \otimes \Lambda_1^*) \\ 
\hline
 & ({\Bbb R}, {\Bbb R}) & \begin{array}{c}  \{0,1\}, \{1,2\} \\[-1ex]  \{4,5\}, \{5,6\}  \end{array} & ( \mathfrak{so}(p,q) \oplus \mathfrak{so}(k_1,k_2) , \Lambda \otimes \Lambda_1 )\\
\cline{2-4} 
(2T, T')   & ({\Bbb C}, {\Bbb C}) & \begin{array}{c} \{1,3\}, \{1,7\}  \\[-1ex]  \{3,5\}, \{5,7\}  \end{array} & ( \mathfrak{so}(p,q) \oplus \mathfrak{u}(k_1,k_2) , \Lambda \otimes \Lambda_1)\\
\cline{2-4}    
                 & ({\Bbb H}, {\Bbb H}) &   \begin{array}{c} \{0,5\}, \{1,4\}  \\[-1ex]  \{1,6\}, \{2,5\}  \end{array} &( \mathfrak{so}(p,q) \oplus \mathfrak{sp}(k_1,k_2) , \Lambda \otimes \Lambda_1 ) \\ 
\hline 
\smash{\lower3mm\hbox{$(2T,2T')$}}   & ({\Bbb C}, {\Bbb R}) &  \{3,3\}, \{7,7\}  & ( \mathfrak{so}(p,q) \oplus \mathfrak{sp}(k_1,\R) \oplus \mathfrak{sp}(k_2,\R) , \Lambda_e \otimes \Lambda_1 \otimes 1 + \Lambda_o \otimes 1 \otimes \Lambda_1 ) \\
\cline{2-4} 
              & ({\Bbb C}, {\Bbb H}) &  \{3,7\}  & ( \mathfrak{so}(p,q) \oplus \mathfrak{so}^*(2k_1) \oplus \mathfrak{so}^*(2k_2) , \Lambda_e \otimes \Lambda_1 \otimes 1 + \Lambda_o \otimes 1 \otimes \Lambda_1 )\\  
\hline     
 \smash{\lower3mm\hbox{$(4T,2T')$}}    & ({\Bbb R}, {\Bbb R}) & \{1,1\}, \{5,5\}  & ( \mathfrak{so}(p,q) \oplus \mathfrak{so}(k_1,k_2) \oplus \mathfrak{so}(k_3,k_4) , \Lambda_e \otimes \Lambda_1 \otimes 1 + \Lambda_o \otimes 1 \otimes \Lambda_1 ) \\
\cline{2-4} 
                & ({\Bbb H}, {\Bbb H}) &  \{1,5\}  & (\mathfrak{so}(p,q) \oplus \mathfrak{sp}(k_1,k_2) \oplus \mathfrak{sp}(k_3,k_4) , \Lambda_e \otimes \Lambda_1 \otimes 1 + \Lambda_o \otimes 1 \otimes \Lambda_1 ) \\
\hline                  
\end{array} 
\]}
\end{thm}


\section{Clifford quartic forms and prehomogeneous vector spaces}

\label{section:classification}

Let $\rho$ be a representation of $R_{p,q}$ on a real vector space $W$. 
As in the previous sections, we put $m=\dim W$ and $\gerg_{p,q}'(\rho)=\mathfrak{so}(p,q)\oplus \gerh_{p,q}(\rho)$.  

\subsection{Classification of the degenerate, generic and prehomogeneous cases}
The aim of  this subsection is to prove the following by using Theorems \ref{thm:3.1}, \ref{thm:3.4} and \ref{thm:3.5}: 
\begin{enumerate}
\def\labelenumi{(\thesection.\alph{enumi})}
\item For the cases classified in Theorem \ref{thm:degenerate}, namely, for $(p,q,m)=(2,1,2)$, $(2,2,4)$, $(3,1,4)$, $(3,3,8)$, $(5,1,8)$, $(5,5,16)$, $(9,1,16)$,  and $(p,q)=(1,1)$ with $S_1 = \pm S_2$, the Clifford quartic forms actually vanish. 
This proves that the statement (iv) implies the statement (ii) in Theorem \ref{thm:degenerate} and completes the classification of degenerate representations. 
\label{prob:degenerate}
\item For the non-degenerate cases which belong to the exceptions in Theorem \ref{thm:3.1}, namely, for $(p,q,m)=(3,0,4)$, $(2,1,4)$, $(4,0,8)$, $(3,1,8)$, $(2,2,8)$, $(5,0,8)$, $(4,1,8)$, $(3,2,8)$, $(6,0,16)$, $(5,1,16)$, $(4,2,16)$, $(3,3,16)$, $(7,0,16)$, $(6,1,16)$, $(5,2,16)$, $(4,3,16)$, $(8,0,16)$, $(7,1,16)$, $(5,3,16)$, $(4,4,16)$, $(9,0,16)$, $(8,1,16)$, $(5,4,16)$, $(10,0,32)$, $(9,1,32)$, $(8,2,32)$, $(7,3,32)$, $(6,4,32)$, $(5,5,32)$, $(10,1,32)$, $(9,2,32)$, $(6,5,32)$,  the Lie algebra $\gerg_{p,q}(\rho)$ is strictly larger than $\gerg_{p,q}'(\rho)$. We call these representations $\rho$ {\it exceptional\/}. The representations that are non-degenerate and not exceptional are called {\it generic\/}. For generic representations we have $\gerg_{p,q}(\rho)=\gerg_{p,q}'(\rho)$ by Theorem \ref{thm:3.1}. 
\label{prob:exceptional}
\item  We classify the non-degenerate representations $\rho$ of $R_{p,q}$ for which the associated Clifford quartic form is a relative invariant of some prehomogeneous vector space. 	  
\label{prob:classification}
\end{enumerate}

A prehomogeneous vector space is, by definition, a triple $(\mathbf{G},\pi,\mathbf{W})$ of a connected linear algebraic group $\mathbf G$, a finite-dimensional vector space $\mathbf W$ and a rational representation $\pi$ of $\mathbf G$ on $\mathbf V$ with a Zariski-open $\mathbf G$-orbit. 
We often suppress $\pi$ or $\mathbf W$ and write $(\mathbf{G},\mathbf{W})$ or $(\mathbf{G},\pi)$, if there is no fear of confusion. 
We say that a finite dimensional representation $(\gerg,\pi,W)$ of a real Lie algebra $\gerg$ is a prehomogeneous vector space, if it is a real form of the infinitesimal representation of a prehomogeneous vector space. 
In the following, by abuse of notation, we use a same symbol to denote a representation of a linear algebraic group (or a Lie group) $G$ and the corresponding infinitesimal representation of $\mathrm{Lie}(G)$.   
This convention also applies to the table in Theorem \ref{thm:3.5}.  
For basic results on prehomogeneous vector spaces, refer to \cite{KimuraBook} and \cite{Sato-Kimura}. 

Denote by $\tau$ the representation of $\gerg_{p,q}'(\rho)$ on $W$ given in Theorem \ref{thm:3.5} and consider the triple $(\mathfrak{gl}(1,\R)\oplus\gerg'_{p,q}(\rho),\tau,W)$,  where $\mathfrak{gl}(1,\R)$ acts as scalar multiplication on $W$. 
Note that if $(\mathfrak{gl}(1,\R)\oplus\gerg'_{p,q}(\rho),\tau,W)$ is a prehomogeneous vector space, then $(\mathfrak{gl}(1,\R)\oplus\gerg_{p,q}(\rho),W)$ is a prehomogeneous vector space. 

First we consider the low-dimensional cases $p+q\leq6$ and $p+q=8$, where the (half-) spin representations of $\mathfrak{so}(p,q)$ are equivalent to the standard or rather simple tensor representations of classical Lie algebras (see \cite[Chapter X, \S 6.4]{Helgason}). 

\subsubsection*{The case $p+q=1$}
In this case, the representation $\rho$ of $R_{1,0}$ is generic and the representation $\tau$  in question is of the form $(\mathfrak{gl}(1,\R)\oplus\mathfrak{so}(k_1,k_2), \Lambda_1,\R^k)$ $(k=k_1+k_2)$, This is a typical example of prehomogeneous vector spaces. 
The fundamental relative invariant is a quadratic form of signature $(k_1,k_2)$ and the Clifford quartic form is its square. 

\subsubsection*{The case $p+q=2$}
In this case, the representation $\tau$ is of the form 
\begin{eqnarray*}
(p,q)=(1,1) & & (\mathfrak{gl}(1,\R)\oplus\mathfrak{so}(1,1)\oplus\mathfrak{so}(k_1,k_2)\oplus\mathfrak{so}(k_3,k_4), \tau,\R^k)\quad (k=k_1+k_2+k_3+k_4), \\ 
              & & \quad \cong (\mathfrak{gl}(1,\R)\oplus\mathfrak{so}(k_1,k_2),\Lambda_1, \R^{k_1+k_2}) \oplus (\mathfrak{gl}(1,\R)\oplus\mathfrak{so}(k_3,k_4),\Lambda_1, \R^{k_3+k_4}), \\ 
(p,q)=(2,0) & & (\mathfrak{gl}(1,\R)\oplus\mathfrak{u}(1)\oplus\mathfrak{so}(k,\C), \C^k)=(\mathfrak{gl}(1,\C)\oplus\mathfrak{so}(k,\C), \Lambda_1, \C^k).
\end{eqnarray*}
These are prehomogeneous vector spaces. 
The Clifford quartic form is the product of two quadratic forms in the former case and the norm of a quadratic form with coefficients in $\C$ in the latter case. 
The representation $\rho$ of $R_{1,1}$ is degenerate if $(k_1+k_2)(k_3+k_4)=0$, and non-degenerate and generic otherwise. 
We have already seen the vanishing of the Clifford quartic forms in the case $(k_1+k_2)(k_3+k_4)=0$ in the proof of (iii) $\Rightarrow$ (iv) of Theorem \ref{thm:degenerate}. 
The representation $\rho$ of $R_{2,0}$ is always non-degenerate and generic. 

\subsubsection*{The case $p+q=3$}
In this case, the representation $\tau$ is of the form 
\begin{eqnarray*}
(p,q)=(2,1) & & (\mathfrak{gl}(1,\R)\oplus\mathfrak{so}(2,1)\oplus\mathfrak{so}(k_1,k_2), \tau,\R^{2k})\quad (k=k_1+k_2), \\ 
              & & \quad \cong (\mathfrak{gl}(2)\oplus\mathfrak{so}(k_1,k_2), \Lambda_1\otimes \Lambda_1,M(2,k;\R)), \\ 
(p,q)=(3,0) & & (\mathfrak{gl}(1,\R)\oplus\mathfrak{so}(3)\oplus\mathfrak{so}^*(2k), \C^{2k}) \\
 & & \quad \cong (\mathfrak{gl}(1,\R)\oplus\mathfrak{sl}(1,\H)\oplus\mathfrak{so}^*(2k), \Lambda_1\otimes \Lambda_1, \H^{k}).
\end{eqnarray*}
These are prehomogeneous vector spaces. 
The representation $\rho$ of $R_{2,1}$ is degenerate for $k=1$, exceptional for $k=2$ (and then $\gerg_{2,1}(\rho)=\mathfrak{sl}(2,\R)\oplus \mathfrak{sl}(2,\R) \supsetneq \gerg'_{2,1}(\rho)$),  and generic for $k \geq 3$. 
The Clifford quartic form actually vanishes for $k=1$, since the prehomogeneous vector space $(\mathfrak{gl}(2,\R), \Lambda,\R^2)$ has no non-trivial relative invariants.  
The representation $\rho$ of $R_{3,0}$ is exceptional for $k=1$ (and then $\gerg_{3,0}(\rho)=\mathfrak{sl}(1,\H)\oplus \mathfrak{sl}(1,\H)\supsetneq \gerg'_{3,0}(\rho)$),  and generic for $k \geq 2$. 
In the two exceptional cases $(p,q,m)=(2,1,4)$, $(3,0,4)$, the Clifford quartic forms are the squares of quadratic forms (see Theorem \ref{thm:squared P}). 

\subsubsection*{The case $p+q=4$}
In this case, the representation  $\tau$ is of the form 
\begin{eqnarray*}
(p,q)=(2,2) & & (\mathfrak{gl}(1,\R)\oplus\mathfrak{so}(2,2)\oplus\mathfrak{gl}(k), \tau,  \R^{4k}), \\ 
              & & \quad \cong (\mathfrak{gl}(1,\R)\oplus\mathfrak{sl}(2,\R)\oplus\mathfrak{sl}(2,\R)\oplus\mathfrak{gl}(k,\R),  \\
     & & \hspace{2cm} (\Lambda_1\otimes 1 \otimes \Lambda_1) \oplus (1\otimes \Lambda_1 \otimes \Lambda_1^*), M(2,k;\R)\oplus M(2,k;\R)), \\ 
(p,q)=(3,1) & & (\mathfrak{gl}(1,\R)\oplus\mathfrak{so}(3,1)\oplus\mathfrak{u}(k_1,k_2), \tau, \C^{2k}) \\
 & & \quad \cong (\mathfrak{gl}(1,\R)\oplus\mathfrak{sl}(2,\C)\oplus\mathfrak{u}(k_1,k_2), \Lambda_1\otimes \Lambda_1, M(2,k_1+k_2;\C)), \\
(p,q)=(4,0) & & (\mathfrak{gl}(1,\R)\oplus\mathfrak{so}(4,0)\oplus\mathfrak{gl}(k,\mathbb{H}), \tau, \H^{2k}), \\ 
              & & \quad \cong (\mathfrak{gl}(1,\R)\oplus\mathfrak{sl}(1,\H)\oplus\mathfrak{sl}(1,\H)\oplus\mathfrak{gl}(k,\mathbb{H}),  \\
     & & \hspace{2cm} (\Lambda_1\otimes 1 \otimes \Lambda_1) \oplus (1\otimes \Lambda_1 \otimes \Lambda_1^*), M(1,k;\H)\oplus M(1,k;\H)). 
\end{eqnarray*}
These are prehomogeneous vector spaces. 
The representation $\rho$ of $R_{2,2}$ is degenerate for $k=1$, exceptional for $k=2$ (and then $\gerg_{2,2}(\rho)=\mathfrak{sl}(2,\R) \oplus \mathfrak{sl}(2,\R) \oplus \mathfrak{sl}(2,\R) \oplus \mathfrak{sl}(2,\R) \oplus \mathfrak{gl}(1,\R)$),  and generic for $k \geq 3$.
The representation $\rho$ of $R_{3,1}$ is degenerate for $k=1$, exceptional for $k=2$ (and then $\gerg_{3,1}(\rho)=\mathfrak{sl}(2,\C) \oplus \mathfrak{sl}(2,\C) \oplus \geru(1)$),  and generic for $k \geq 3$.
The representation $\rho$ of $R_{4,0}$ is exceptional for $k=1$ (and then $\gerg_{4,0}(\rho)=\mathfrak{sl}(1,\H) \oplus \mathfrak{sl}(1,\H) \oplus \mathfrak{sl}(1,\H) \oplus \mathfrak{sl}(1,\H) \oplus \mathfrak{gl}(1,\R)$),  and generic for $k \geq 2$.
The Clifford quartic forms actually vanish for $(p,q)=(2,2), (3,1)$ and $k=1$, since the prehomogeneous vector space $(\mathfrak{gl}(2,\mathbb K), \Lambda_1,\mathbb K^2)$ $(\mathbb K=\R, \C)$ has no non-trivial relative invariants.  
The Clifford quartic form for an exceptional representation is the product of two quadratic forms in 4 variables or the square of the absolute value of a complex quadratic form in 4 variables according as $(p,q,m)=(2,2,4)$, $(4,0,4)$ or $(p,q,m)=(3,1,4)$.

\subsubsection*{The case $p+q=5$}
In this case, the representation $\tau$ is of the form 
\begin{eqnarray*}
(p,q)=(3,2) & & (\mathfrak{gl}(1,\R)\oplus\mathfrak{so}(3,2)\oplus\mathfrak{sp}(k,\R),\tau, \R^{8k}), \\ 
              & & \quad \cong (\mathfrak{gl}(1,\R)\oplus\mathfrak{sp}(2,\R)\oplus\mathfrak{sp}(k,\R), \Lambda_1\otimes \Lambda_1, M(4,2k;\R)), \\ 
(p,q)=(4,1) & & (\mathfrak{gl}(1,\R)\oplus\mathfrak{so}(4,1)\oplus\mathfrak{sp}(k_1,k_2),\tau, \H^{2k}) \\
 & & \quad \cong (\mathfrak{gl}(1,\R)\oplus\mathfrak{sp}(1,1)\oplus\mathfrak{sp}(k_1,k_2),  \Lambda_1\otimes \Lambda_1, M(2,k;\H)), \\
(p,q)=(5,0) & & (\mathfrak{gl}(1,\R)\oplus\mathfrak{so}(5)\oplus\mathfrak{sp}(k_1,k_2), \tau, \H^{2k}) \\
 & & \quad \cong (\mathfrak{gl}(1,\R)\oplus\mathfrak{sp}(2)\oplus\mathfrak{sp}(k_1,k_2),  \Lambda_1\otimes \Lambda_1, M(2,k;\H)).
\end{eqnarray*}
By the classification of irreducible prehomogeneous vector spaces (\cite{Sato-Kimura}), these are prehomogeneous vector spaces  only when $k=1$ (the unique  exceptional case) and then it is easy to see that the Clifford quartic form is the square of a quadratic forms (cf.\ Theorem \ref{thm:squared P}). 
Hence  we have  
$\gerg_{3,2}(\rho)=\mathfrak{so}(4,4)$, $\gerg_{4,1}(\rho)=\mathfrak{so}(4,4)$, and $\gerg_{5,0}(\rho)=\mathfrak{so}(8)$. They are strictly larger than $\gerg'_{p,q}(\rho)$. 

\subsubsection*{The case $p+q=6$}
In this case, the representation $\tau$ is of the form 
\begin{eqnarray*}
(p,q)=(3,3) & & (\mathfrak{gl}(1,\R)\oplus\mathfrak{so}(3,3)\oplus\mathfrak{sp}(k_1,\R)\oplus\mathfrak{sp}(k_2,\R), \tau, \R^{8k}), \\ 
              & & \quad \cong (\mathfrak{gl}(1,\R)\oplus\mathfrak{sl}(4)\oplus\mathfrak{sp}(k_1,\R)\oplus\mathfrak{sp}(k_2,\R), \\ 
 & & \hspace{2cm} (\Lambda_1\otimes \Lambda_1 \otimes 1) \oplus (\Lambda_1^* \otimes 1 \otimes \Lambda_1), M(4,2k_1;\R)\oplus M(4,2k_2;\R)), \\ 
(p,q)=(4,2) & & (\mathfrak{gl}(1,\R)\oplus\mathfrak{so}(4,2)\oplus\mathfrak{sp}(k,\C),  \tau, \C^{8k}), \\ 
              & & \quad \cong (\mathfrak{gl}(1,\R)\oplus\mathfrak{su}(2,2)\oplus\mathfrak{sp}(k,\C), \Lambda_1\otimes \Lambda_1, M(4,2k;\C)), \\ 
(p,q)=(5,1) & & (\mathfrak{gl}(1,\R)\oplus\mathfrak{so}(5,1)\oplus\mathfrak{sp}(k_1,k_2)\oplus\mathfrak{sp}(k_3,k_4),  \tau, \H^{2k}) \\
 & & \quad \cong (\mathfrak{gl}(1,\R)\oplus\mathfrak{sl}(2,\H)\oplus\mathfrak{sp}(k_1,k_2)\oplus\mathfrak{sp}(k_3,k_4), \\ 
 & & \hspace{2cm} (\Lambda_1\otimes \Lambda_1 \otimes 1) \oplus (\Lambda_1^* \otimes 1 \otimes \Lambda_1), M(2,k_1+k_2;\H)\oplus M(2,k_3+k_4;\H)), \\
(p,q)=(6,0) & & (\mathfrak{gl}(1,\R)\oplus\mathfrak{so}(6)\oplus\mathfrak{sp}(k,\C),  \tau, \C^{8k}) \\
 & & \quad \cong (\mathfrak{gl}(1,\R)\oplus\mathfrak{su}(4)\oplus\mathfrak{sp}(k,\C), \Lambda_1\otimes \Lambda_1,  M(4,2k;\C)).
\end{eqnarray*}
These are prehomogeneous vector spaces  only when 
\[
\begin{cases}
k=1 & ((p,q)=(4,2), (6,0)) \\
k_1=k_2=1\ \text{or}\ k_1k_2=0  & ((p,q)=(3,3)) \\
k_1+k_2=k_3+k_4=1\ \text{or}\ (k_1+k_2)(k_3+k_4)=0  & ((p,q)=(5,1)). 
\end{cases}
\]
The non-prehomogeneous cases are characterized as the cases where $m=\dim_{\R}W>16$, and the representation $\rho$ is mixed over $\C$ in the sense in Definition \ref{def:pure and mixed}.
For $(p,q)=(3,3), (5,1)$ and $m=8$, the Clifford quartic form vanishes since the prehomogeneous vector space $(GL(1) \times SL(4) \times SL(2), \Lambda_1\otimes\Lambda_1)$ has no non-trivial relative invariants. 
For the case $m=16$, then $\gerg_{p,q}(\rho)\otimes\C$ is isomorphic to $\mathfrak{sl}(4,\C) \oplus \mathfrak{sl}(4,\C)$ or $\mathfrak{gl}(1,\C) \oplus \mathfrak{sl}(4,\C) \oplus \mathfrak{sl}(2,\C) \oplus \mathfrak{sl}(2,\C)$ according as $\rho$ is pure over $\C$ or mixed over $\C$ and is strictly larger than $\gerg_{p,q}'(\rho)\otimes\C$.  

\subsubsection*{The case $p+q=8$}
In this case, the representation $\tau$ is of the form 
\begin{eqnarray*}
(p,q)=(4,4) & & (\mathfrak{gl}(1,\R)\oplus\mathfrak{so}(4,4)\oplus\mathfrak{gl}(k,\R), \tau, \R^{16k})  \\ 
& & \quad \cong (\mathfrak{gl}(1,\R)\oplus\mathfrak{so}(4,4)\oplus\mathfrak{gl}(k,\R), (\Lambda_e\otimes \Lambda_1) \oplus (\Lambda_o\otimes \Lambda_1^*), M(8,k;\R)\oplus M(8,k;\R)),  \\ 
(p,q)=(5,3) & & (\mathfrak{gl}(1,\R)\oplus\mathfrak{so}(5,3)\oplus\mathfrak{u}(k_1,k_2),  \tau, \C^{8k}) \\ 
& & \quad \cong (\mathfrak{gl}(1,\R)\oplus\mathfrak{so}(5,3)\oplus\mathfrak{u}(k_1,k_2),  \Lambda \otimes \Lambda_1, M(8,k;\C)), \\ 
(p,q)=(6,2) & & (\mathfrak{gl}(1,\R)\oplus\mathfrak{so}(6,2)\oplus\mathfrak{gl}(k,\H),  \tau,  \H^{8k}) \\ 
& & \quad \cong (\mathfrak{gl}(1,\R)\oplus\mathfrak{so}(6,2)\oplus\mathfrak{gl}(k,\H),  (\Lambda_e\otimes \Lambda_1) \oplus (\Lambda_o\otimes \Lambda_1^*),  M(4,k;\H) \oplus M(4,k;\H)), \\ 
(p,q)=(7,1) & & (\mathfrak{gl}(1,\R)\oplus\mathfrak{so}(7,1)\oplus\mathfrak{u}(k_1,k_2),  \tau,  \C^{8k}) \\
& & \quad \cong (\mathfrak{gl}(1,\R)\oplus\mathfrak{so}(7,1)\oplus\mathfrak{u}(k_1,k_2),  \Lambda\otimes \Lambda_1,  M(8,k;\C)), \\
(p,q)=(8,0) & & (\mathfrak{gl}(1,\R)\oplus\mathfrak{so}(8)\oplus\mathfrak{gl}(k,\R),  \tau, \R^{16k}) \\
& & \quad \cong (\mathfrak{gl}(1,\R)\oplus\mathfrak{so}(8)\oplus\mathfrak{gl}(k,\R),  (\Lambda_e\otimes \Lambda_1) \oplus (\Lambda_o\otimes \Lambda_1^*), M(8,k;\R)\oplus M(8,k;\R)).
\end{eqnarray*}
By the classification of irreducible prehomogeneous vector spaces (\cite{Sato-Kimura}, \cite{KimuraI}, \cite{KimuraII}), these are prehomogeneous vector spaces  only when $k=1$ and $(p,q)\ne(6,2)$ (the unique  exceptional case) and then the Clifford quartic form is the product of two quadratic forms in 8 variables or the square of the absolute value of a complex quadratic form in 8 variables according as $(p,q,m)=(4,4,16),(8,0,16)$ or $(p,q,m)=(5,3,16),(7,1,16)$. 
Hence we have 
$\gerg_{4,4}(\rho)=\mathfrak{so}(4,4)\oplus\mathfrak{so}(4,4)\oplus\mathfrak{gl}(1,\R)$, $\gerg_{5,3}(\rho)=\mathfrak{so}(8,\C)\oplus\geru(1)$, $\gerg_{7,1}(\rho)=\mathfrak{so}(8,\C)\oplus\geru(1)$, and $\gerg_{8,0}(\rho)=\mathfrak{so}(8)\oplus\mathfrak{so}(8)\oplus\mathfrak{gl}(1,\R)$. They are strictly larger than $\gerg'_{p,q}(\rho)$.

\subsubsection*{The case $p+q=7$ or $p+q \geq 9$}
By \cite{Sato-Kimura} (see also \cite[Theorem 1.5]{KimuraI}), the irreducible prehomogeneous vector spaces (defined over $\C$) containing the spin or half-spin representations of $\mathbf{Spin}(p+q)$ ($p+q=7,\ \text{or}\ \geq 9$) are given by 
\begin{equation}
\label{eqn:spin pv}
\begin{array}{ll}
(\mathbf{Spin}(7)\times \mathbf{GL}(k),\Lambda \otimes \Lambda_1) & (k=1,2,3,5,6,7), \\
(\mathbf{Spin}(9)\times \mathbf{GL}(k),\Lambda \otimes \Lambda_1) & (k=1,15), \\
(\mathbf{Spin}(10)\times \mathbf{GL}(k),\Lambda_\sharp \otimes \Lambda_1) & (k=1,2,3,13,14,15,\ \sharp=e,o), \\
(\mathbf{Spin(11)}\times \mathbf{GL}(k),\Lambda \otimes \Lambda_1) & (k=1,31), \\
(\mathbf{Spin}(12)\times \mathbf{GL}(k),\Lambda_\sharp \otimes \Lambda_1) & (k=1,31,\ \sharp=e,o), \\
(\mathbf{Spin(14)}\times \mathbf{GL(k)},\Lambda_\sharp \otimes \Lambda_1) & (k=1,63,\ \sharp=e,o), \\
(\mathbf{Spin}(p+q)\times \mathbf{GL}(k),\Lambda_\sharp \otimes \Lambda_1) & (k\geq \deg\Lambda_\sharp,\ \sharp=e,o).
\end{array}
\end{equation}

First note that, by Theorem \ref{thm:3.1}, there exists only one degenerate representation for $p+q=7$ or $\geq 9$, namely the case $p+q=10$ and $m=16$.  
In this case the complexification of the representation $\tau$ is $(\mathbf{Spin}(10)\times \mathbf{GL}(1),\Lambda_\sharp)$ and this is a prehomogeneous vector space without non-trivial relative invariants (see \cite{Sato-Kimura}). Hence the Clifford quartic form vanishes in this case.

\begin{lem}
Let $\rho$ be a non-degenerate representation of  $R_{p,q}$ on a real vector space $W$. 
Assume that $p+q=7$ or $p+q \geq 9$. 
If $\rho$ is generic,  then $(\mathfrak{gl}(1,\R) \oplus \mathfrak{g}_{p,q}(\rho),W)$ is not a prehomogeneous vector space. 
If $\rho$ is exceptional, then $(\mathfrak{gl}(1,\R) \oplus \mathfrak{g}_{p,q}(\rho),W)$ is a prehomogeneous vector space except the case where $p+q=10$ and $\rho$ is mixed over $\C$. Moreover, as the following table shows, the Lie algebra $\gerg_{p,q}(\rho)$ is strictly larger than $\gerg'_{p,q}(\rho)$ in the exceptional cases. 
\[
\begin{array}{|c|c|c|c|c|}
\hline
p+q & \dim W & \mathfrak{g}'_{p,q}(\rho)\otimes\C & \mathfrak{g}_{p,q}(\rho)\otimes\C & \text{\rm pv?}\\
\hline
\hline
7 & 16 & (\mathfrak{so}(7)\oplus\mathfrak{sl}(2),\Lambda\otimes\Lambda_1) & \mathfrak{so}(8)\oplus\mathfrak{sl}(2) & {\rm pv} \\ 
\hline
9 & 16 & (\mathfrak{so}(9),\Lambda) & \mathfrak{so}(16)  & {\rm pv} \\
\hline
10 & 32\ (\rho=\text{pure}/\C) & (\mathfrak{so}(10)\oplus\mathfrak{so}(2),\Lambda_\sharp\otimes\Lambda_1) & \mathfrak{so}(10)\oplus\mathfrak{sl}(2) & {\rm pv}  \\
\hline
10 & 32\ (\rho=\text{mixed}/\C) & (\mathfrak{so}(10),\Lambda_e\oplus\Lambda_o) & \mathfrak{so}(10)\oplus\mathfrak{gl}(1)  & \text{\rm not pv} \\
\hline
11 & 32 & (\mathfrak{so}(11),\Lambda) & \mathfrak{so}(12)  & {\rm pv} \\
\hline
\end{array}
\]
In particular, if $p+q \geq 12$, then  $(\mathfrak{gl}(1) \oplus \mathfrak{g}_{p,q}(\rho),W)$ is not a prehomogeneous vector space. 
\end{lem}

\begin{proof}
By Theorem \ref{thm:3.5}, the complexification of the representation $\tau$ of $\mathfrak{gl}(1) \oplus \gerg'_{p,q}(\rho)$ on $W$ are equivalent to one of the following:
\begin{enumerate}
\def\labelenumi{(\Roman{enumi})}
\item $(\mathfrak{gl}(1) \oplus \mathfrak{so}(p+q)\oplus\mathfrak{so}(k), \Lambda_e\otimes\Lambda_1)$ \quad $(p+q \equiv 1, 3 \pmod 8)$,
\item $(\mathfrak{gl}(1) \oplus \mathfrak{so}(p+q)\oplus\mathfrak{sp}(k), \Lambda_e\otimes\Lambda_1)$ \quad $(p+q \equiv 5, 7 \pmod 8)$,
\item $(\mathfrak{gl}(1) \oplus \mathfrak{so}(p+q)\oplus\mathfrak{gl}(k), \Lambda_e\otimes\Lambda_1 \oplus \Lambda_o\otimes\Lambda_1^*)$ \quad $(p+q \equiv 0, 4 \pmod 8)$,
\item $(\mathfrak{gl}(1) \oplus \mathfrak{so}(p+q)\oplus\mathfrak{so}(k_1)\oplus\mathfrak{so}(k_2), \Lambda_e\otimes\Lambda_1\otimes 1 \oplus \Lambda_o\otimes 1 \otimes\Lambda_1^*)$ \quad $(p+q \equiv 2 \pmod 8)$,
\item $(\mathfrak{gl}(1) \oplus \mathfrak{so}(p+q)\oplus\mathfrak{sp}(k_1)\oplus\mathfrak{sp}(k_2), \Lambda_e\otimes\Lambda_1\otimes 1 \oplus \Lambda_o\otimes 1 \otimes\Lambda_1^*)$ \quad $(p+q \equiv 6 \pmod 8)$.
\end{enumerate} 
Assume that $(\mathfrak{gl}(1) \oplus \mathfrak{g}'_{p,q}(\rho),\tau,W)$ is a prehomogeneous vector space. 
Then,  every direct summand is also a prehomogeneous vector spaces and it must coincide with one of the representations in $(\ref{eqn:spin pv})$. 
Hence, in the case (I), by the congruence condition on $p+q\equiv 1,3\pmod 8$, we have $p+q=9,11$ and $k=1,2$. By the classification in \cite{Kimura} (see also \cite[Proposition 1.3]{KimuraI}), if $k=2$, then the representation does not give a prehomogeneous vector space. Thus we obtain a prehomogeneous vector space only for $k=1$ (the unique exceptional case for $p+q=9,11$), namely, for $(p,q,m)=(9,0,16)$, $(8,1,16)$, $(5,4,16)$, $(10,1,32)$, $(9,2,32)$, $(6,5,32)$. 
In the case $(p,q,m)=(9,0,16)$, $(8,1,16)$, $(5,4,16)$, the Clifford quartic form is the square of a quadratic form and we have $\gerg_{9,0}(\rho)=\mathfrak{so}(16)$, $\gerg_{8,1}(\rho)=\mathfrak{so}(8,8)$, $\gerg_{5,4}(\rho)=\mathfrak{so}(8,8)$, which is larger than $\gerg'_{p,q}(\rho)$  (cf.\ Theorem \ref{thm:squared P}, \cite[\S 7, I), (19)]{Sato-Kimura}). 
In the case $(p,q,m)=(10,1,32)$, $(9,2,32)$, $(6,5,32)$, the Clifford quartic form is the irreducible relative invariant of the space (22) in \cite[\S 7, I)]{Sato-Kimura}, which is the same as the irreducible relative invariant of the space (23) in \cite[\S 7, I)]{Sato-Kimura}. 
This shows that $\gerg_{10,1}(\rho)=\gerg_{9,2}(\rho)=\mathfrak{so}(10,2)$ and $\gerg_{6,5}(\rho)=\mathfrak{so}(6,6)$, which is larger than $\gerg'_{p,q}(\rho)$. 

In the case (II),  we obtain a prehomogeneous vector space only for $p+q=7$, $k=1$ and $m=16$ (the unique exceptional case for $p+q=7$), namely, for $(p,q,m)=(7,0,16)$, $(6,1,16)$, $(5,2,16)$, $(4,3,16)$.  
In these cases, the Clifford quartic form is the irreducible relative invariant of the space (17) in \cite[\S 7, I)]{Sato-Kimura}, which is the same as the irreducible relative invariant of the space (15) in \cite[\S 7, I)]{Sato-Kimura}. 
This shows that  
 $\gerg_{7,0}(\rho) = \mathfrak{so}(8)\oplus\mathfrak{sl}(2,\R)$,  
 $\gerg_{4,3}(\rho) = \mathfrak{so}(4,4)\oplus\mathfrak{sl}(2,\R)$,  
$\gerg_{6,1}(\rho) \cong \gerg_{5,2}(\rho) \cong \mathfrak{so}^*(8)\oplus\mathfrak{su}(2)$, which is larger than $\gerg'_{p,q}(\rho)$.  

In the case (III),  if $\deg \Lambda_e \leq k$, by \cite[Proposition 1.15]{KimuraII}, the representation does not give a prehomogeneous vector space. 
If $\deg \Lambda_e > k$, by the classification of $2$-simple prehomogenous vector spaces in \cite{KimuraI}, we can easily check that the representation does not give a prehomogeneous vector space. 

In the case (IV),  if the representation gives a prehomogeneous vector space, then $n=10$ and $k_1,k_2=0,1,2$. If $k_1\cdot k_2 \ne 0$, from the classification in \cite{Kimura} and \cite{KimuraI}, we see that the representation does not give a prehomogeneous vector space. If $(k_1,k_2)=(1,0)$ or $(0,1)$, this is a degenerate case, which we have already examined just before this lemma. 
If $(k_1,k_2)=(2,0)$ or $(0,2)$, then the representation gives a prehomogeneous vector space of type (17) in the list in \cite[\S 3]{Kimura}. Hence the Clifford quartic form is the same as the irreducible relative invariant of the space (20)  in \cite[\S 7, I)]{Sato-Kimura}. 
Hence we have $\gerg_{9,1}(\rho)=\mathfrak{so}(9,1)\oplus\mathfrak{sl}(2,\R)$, $\gerg_{5,5}(\rho)=\mathfrak{so}(5,5)\oplus\mathfrak{sl}(2,\R)$,  $\gerg_{7,3}(\rho)=\mathfrak{so}(7,3)\oplus\mathfrak{su}(2)$,  which is larger than 
$\gerg'_{p,q}(\rho)$. 

In the case (V), by the congruence condition $p+q\equiv 6 \pmod 8$, we have $p+q=14$.  
However the irreducible representation contained in (V) does not appear in $(\ref{eqn:spin pv})$. 
Hence, we obtain no prehomogeneous vector spaces in this case. 

Thus all the cases for which $(\mathfrak{gl}(1) \oplus \mathfrak{g}'_{p,q}(\rho),W)$ is a prehomogeneous vector space are the 4 cases given in the lemma and, by Theorem \ref{thm:3.1},  all of them are of exceptional type.  
Therefore, if $\rho$ is of generic type (and $p+q=7$ or $p+q\geq9$), then  $(\mathfrak{gl}(1) \oplus \mathfrak{g}_{p,q}(\rho),W)$ is not a prehomogeneous vector space. 
In these 4 exceptional cases the Lie algebras $\gerg_{p,q}(\rho)$ are well known in the theory of  prehomogeneous spaces (see \cite{Sato-Kimura} and \cite{Kimura}). 
There remains one more representation of exceptional type, namely, the case where $p+q=10$, $\dim W=16$ and $\rho$ is mixed over $\C$. 
In this case, we have $\gerg_{p,q}(\rho)=\mathfrak{so}(10)\oplus \mathfrak{gl}(1)$ by direct calculation and hence $(\mathfrak{gl}(1) \oplus \mathfrak{g}_{p,q}(\rho),W)$ is not a prehomogeneous vector space.   
\end{proof}

Summing up the results above, we see that (4.a) and (4.b) are established. 
As for the classification of prehomogeneous cases (4.c), we obtain the following theorem. 

\begin{thm}
\label{thm:pv} 
$(1)$ A Clifford quartic form is not a relative invariant of any prehomogeneous vector space if and only if 
\[
\left\{
\begin{array}{l}
p+q=5,\ m>8;\\
p+q=6,\ m>16\ \text{and $\rho$ is mixed over $\C$}; \\
p+q=7,8,9,\ m>16;\\
p+q=10,\ m>32,\ \text{or $m=32$ and $\rho$ is mixed over $\C$}; \\
p+q=11,\ m>32; \\
p+q\geq12.
\end{array}
\right.
\]

$(2)$ 
The prehomogeneous vector spaces having Clifford quartic forms as relative invariant are the spaces listed in Table 1 in p.\pageref{table:1}.  
\end{thm}
\begin{table}
\begin{center}
{\small
$
\def\arraystretch{1.3}
\begin{array}{|c|c|}
\hline
(p,q) &  \text{prehomogeneous vector space}\\
\hline
\hline
(1,0) & (GL(1,\R)\times SO(k_1,k_2), \Lambda_1) \\
\hline
(2,0) & (GL(1,\C)\times SO(k,\C), \Lambda_1) \\
\hline
(1,1) & (GL(1,\R)\times SO(k_1,k_2), \Lambda_1)\oplus (GL(1,\R)\times SO(k_3,k_4), \Lambda_1) \\
\hline
(3,0) & (GL(1,\H) \times SO^*(2k), \Lambda_1\otimes\Lambda_1) \\
\hline
(2,1) & (GL(2,\R)\times SO(k_1,k_2), \Lambda_1\otimes\Lambda_1)\\
\hline
(4,0) & (GL(1,\mathbb H)\times GL(1,\mathbb H)\times GL(k,\mathbb H), (\Lambda_1\otimes1\otimes\Lambda_1)\oplus(1\otimes\Lambda_1\otimes\Lambda_1^*) )\\
\hline
(3,1) & (GL_2(\C)\times SU(k_1,k_2), \Lambda_1\otimes\Lambda_1)\\
\hline
(2,2) & (GL_2(\R)\times GL_2(\R)\times SL(k,\R), (\Lambda_1\otimes1)\oplus(1\otimes\Lambda_1\otimes\Lambda_1))\\
\hline
(5,0) & (GL(1,\R)\times SO(8), \Lambda_1)\\
\hline
(4,1) & (GL(1,\R)\times SO(4,4), \Lambda_1)\\
\hline
(3,2) & (GL(1,\R)\times SO(4,4), \Lambda_1)\\
\hline
(6,0) & (GL(2,\C)\times SU(4), \Lambda_1 \otimes \Lambda_1)\\
\hline
\smash{\lower3mm\hbox{$(5,1)$}} & (GL_2(\mathbb H)\times Sp(k_1,k_2),\Lambda_1 \otimes \Lambda_1)\ (k_1+k_2 \geq 2) \\
\cline{2-2}
& (GL_1(\R)\times SL(2,\H)\times SU(2)\times SU(2),(\Lambda_1\otimes\Lambda_1\otimes1)\oplus(\Lambda_1^*\otimes1\otimes\Lambda_1))  \\
\hline
(4,2) & (GL(2,\C)\times SU(2,2), \Lambda_1 \otimes \Lambda_1)\\
\hline
\smash{\lower3mm\hbox{$(3,3)$}} &  (GL_4(\R)\times Sp(k,\R),\Lambda_1\otimes\Lambda_1)  \quad (k \geq 2) \\
\cline{2-2}
& (GL_1(\R)\times SL(4,\R)\times SL(2,\R)\times SL(2,\R),(\Lambda_1\otimes\Lambda_1\otimes1)\oplus(\Lambda_1^*\otimes1\otimes\Lambda_1))  \\
\hline
(7,0) & (GL_2(\R)\times SO(8),\Lambda_1\otimes\Lambda_1) \\
\hline
(6,1) & (GL_1(\H) \times SO^*(8),\Lambda_1\otimes\Lambda_1) \\
\hline
(5,2) & (GL_1(\H) \times SO^*(8),\Lambda_1\otimes\Lambda_1) \\
\hline
(4,3) & (GL_2(\R)\times SO(4,4),\Lambda_1\otimes\Lambda_1) \\
\hline
(8,0) & (GL_1(\R)\times GL_1(\R)\times SO(8)\times SO(8),(\Lambda_1\otimes1)\oplus(1\otimes\Lambda_1)) \\
\hline
(7,1) & (GL_1(\C)\times SO(8,\C), \Lambda_1) \\
\hline
(5,3) & (GL_1(\C)\times SO(8,\C), \Lambda_1) \\
\hline
(4,4) & (GL_1(\R)\times GL_1(\R)\times SO(4,4)\times SO(4,4),(\Lambda_1\otimes1)\oplus(1\otimes\Lambda_1)) \\
\hline
(9,0) & (GL_1(\R)\times SO(16),\Lambda_1) \\
\hline
(8,1) & (GL_1(\R)\times SO(8,8),\Lambda_1) \\
\hline
(5,4) & (GL_1(\R)\times SO(8,8),\Lambda_1) \\
\hline
(9,1) & (GL_2(\R)\times Spin(9,1),\Lambda_1\otimes\Lambda_\sharp) 
    \quad (\sharp=e,o) \\
\hline
(7,3) & (GL_1(\H) \times Spin(7,3),\Lambda_1\otimes\Lambda_\sharp)     \quad (\sharp=e,o) \\
\hline
(5,5) & (GL_2(\R)\times Spin(5,5),\Lambda_1\otimes\Lambda_\sharp)    \quad (\sharp=e,o) \\
\hline
(10,1) & (GL_1(\R)\times Spin(10,2),\Lambda_\sharp)    \quad (\sharp=e,o) \\
\hline
(9,2) & (GL_1(\R)\times Spin(10,2),\Lambda_\sharp)    \quad (\sharp=e,o) \\
\hline
(6,5) & (GL_1(\R)\times Spin(6,6),\Lambda_\sharp)    \quad (\sharp=e,o) \\
\hline
\end{array}
$
} 
\end{center}
\caption{Prehomogeneous vector spaces having Clifford quartic forms as relative invariant}
\label{table:1}
\end{table}

\begin{remark}
Among prehomogeneous vector spaces in Table 1, the cases (9,1), (7,3), (5,5), (10,1), (9,2), (6,5) are interesting, since these spaces have rather complex structures. Our theorem gives a construction of the irreducible relative invariants and explicit formulas for the local functional equations in a unified manner. 
For the split case (5,5) (resp.\ (6,5)) the explicit formula for the functional equations was calculated earlier by Igusa \cite{Igusa} (resp.\ Suzuki \cite{Suzuki1}). 
The other cases seem to be new.    
\end{remark}

\subsection{Non-prehomogeneous example: $(p,q)=(3,2)$}
As is shown in the theorem above, non-prehomogeneous Clifford quartic forms appear first for $p+q=5$. We consider here this non-prehomogeneous case for $(p,q)=(3,2)$ in some detail. 
The algebra $R_{3,2}$ has a unique irreducible representation $\rho_0$ of degree $8$. 
We may choose the basis matrices $S_i=\rho_0(e_i)$ $(1\leq i \leq 5)$ as follows:
\begin{gather*}
S_1=
\begin{pmatrix}
0 & 0 & 0 & 1_2 \\ 
0 & 0 & -1_2 & 0 \\ 
0 & -1_2 & 0 & 0 \\ 
1_2 & 0 & 0 & 0 
\end{pmatrix}, \quad
S_2=
\begin{pmatrix}
0 & 0 & J & 0 \\ 
0 & 0 & 0 & -J \\ 
-J & 0 & 0 & 0 \\ 
0 & J  & 0 & 0 
\end{pmatrix}, \quad
S_3=
\begin{pmatrix}
0 & 0 & 0 & J \\ 
0 & 0 & J & 0 \\ 
0 & -J & 0 & 0 \\ 
-J & 0  & 0 & 0 
\end{pmatrix}, \\
S_4=
\begin{pmatrix}
0 & 0 & 0 & H \\ 
0 & 0 & -H & 0 \\ 
0 & -H & 0 & 0 \\ 
H & 0  & 0 & 0 
\end{pmatrix},  \quad
S_5=
\begin{pmatrix}
0 & 0 & 0 & K \\ 
0 & 0 & -K & 0 \\ 
0 & -K & 0 & 0 \\ 
K & 0  & 0 & 0 
\end{pmatrix},
\end{gather*}
where we put $J=\begin{pmatrix} 0 & -1 \\ 1 & 0 \end{pmatrix}$, $H=\begin{pmatrix} -1 & 0 \\ 0 & 1 \end{pmatrix}$, $K=\begin{pmatrix} 0 & 1 \\ 1 & 0 \end{pmatrix}$.
Then we have 
\[
\rho_0(\gerk_{3,2})=\set{\begin{pmatrix} X & 0 \\ 0 & X \end{pmatrix}}{X \in \mathfrak{sp}(2;\R)}, \quad  \mathfrak{sp}(2;\R)=\set{X \in M(4;\R)}{{}^tXJ_2+J_2X=0},
\]
where $\gerk_{3,2}$ is the Lie algebra spanned by $e_ie_j$ $(1 \leq i < j \leq 5)$ and $J_2=J \perp J$. 
We identify the representation space $W_0$ of $\rho_0$ with $M(4,2;\R)$ by 
$w=\begin{pmatrix} u \\ v \end{pmatrix} \mapsto (u v)$. Then the action of $\gerk_{3,2}$ is given by the left multiplication of $\mathfrak{sp}(2;\R)$.  
Consider a reducible representation $\rho=\rho_0^{\oplus k}$. 
If $k \geq 2$, then $\rho$ is non-degenerate, generic and non-prehomegeneous, and the Clifford quartic form 
\[
\tilde P(w) = \sum_{i=1}^3 S_i^{(k)}[w]^2 - \sum_{i=4}^5 S_i^{(k)}[w]^2, 
\quad S_i^{(k)}=\overbrace{S_i\perp \cdots\perp S_i}^k
\]
satisfies the functional equation (Theorem \ref{thm:2.14})
\begin{eqnarray*}
\begin{pmatrix}
\tilde\zeta_+\left(s,\hat\Psi\right)  \\
\tilde\zeta_-\left(s,\hat\Psi\right)  
\end{pmatrix}
 &=& 2^{4s+4k-1}\pi^{-4s -4k-2} \Gamma (s+1) 
          \Gamma \left(s+\frac 52 \right) \Gamma\left (s+2k-\frac 32 \right) 
          \Gamma \left(s+2k\right) \\
  & & {} \times  \begin{pmatrix}
                     -\sin (2\pi s) & -4\sin(\pi s)  \\
                       0 & \sin(2\pi s)
                   \end{pmatrix} 
                  \begin{pmatrix}
                  \tilde\zeta_+\left(-2k - s,\Psi\right)  \\
                  \tilde\zeta_-\left(-2k - s,\Psi\right)  
                   \end{pmatrix}.
\end{eqnarray*}
In the present case, the Clifford quartic form has an expression in terms of more popular invariants. 
By Theorems \ref{thm:3.4} and \ref{thm:3.5}, $\tilde P$ is invariant under the action of the group $G_{3,2}(\rho)$ with Lie algebra
\[
\gerg_{3,2}(\rho)= \mathfrak{sp}(2;\R)\oplus \mathfrak{sp}(k;\R), 
\quad \mathfrak{sp}(k;\R)=\set{X \in M(2k;\R)}{{}^tXJ_k+J_kX=0},
\]
where $J_k=\overbrace{J \perp \cdots \perp J}^k$.
We identify the representation space of $\rho$ with $M(4,2k;\R)$.  
Then the action on $M(4,2k;\R)$ is given by $w \mapsto Xw+w\,{}^tY$ $(x \in \mathfrak{sp}(2;\R),\ Y \in \mathfrak{sp}(k;\R))$ (see \S 6.2). 

\begin{lem}
\label{lem:sp invariant}
We define the polynomials $P_1,P_2$ on $M(4,2k)$ by 
\[
P_1(w)= \text{the Pfaffian of $w J_k \,{}^tw$}, \quad
P_2(w)= \mathrm{tr}(J_2w J_k \,{}^tw).
\]
Then the ring $\C[M(4,2k)]^{Sp(2)\times Sp(k)}$ of $Sp(J_2)\times Sp(k)$-invariants is generated by $P_1,P_2$ and is isomorphic to the polynomial ring of $2$ variables.
\end{lem}

\begin{proof}
The $Sp(2)$-equivariant mapping $\phi:M(4,2k) \rightarrow \mathrm{Alt}(4)$ $(\phi(w)=w J_k \,{}^tw)$ induces an isomorphism $M(4,2k)//Sp(k) \cong \mathrm{Alt}(4)$, namely, $\phi^*:\C[\mathrm{Alt}(4)] \rightarrow \C[M(4,2k)]^{Sp(k)}$ is a $\C$-algebra isomorphism. Hence,  $\phi^*:\C[\mathrm{Alt}(4)]^{Sp(2)} \stackrel{\cong}{\longrightarrow} \C[M(4,2k)]^{Sp(2)\times Sp(k)}$.   
Put $V_1=\C J_2$ and $V_2=\set{Y \in \mathrm{Alt}(4)}{\mathrm{tr}(J_2Y)=0}$. 
Then $V_1$ and $V_2$ are simple $Sp(J_2)$-modules and $\mathrm{Alt}(4)=V_1\oplus V_2$. It is known that $(GL(1)\times GL(1) \times Sp(J_2),\mathrm{Alt}(4))$ is a regular prehomogeneous vector space and the fundamental relative invariants are given by the Pfaffian $Pf(Y)$ and $\mathrm{tr}(J_2Y)$ (see \cite{Kimura}).  Here the first (resp.\ second) factor of $GL(1)\times GL(1)$ acts of $V_1$ (resp.\ $V_2$) as scalar multiplication. 
This shows that $\C[M(4,2k)]^{Sp(2)}=\C[Pf(Y),\mathrm{tr}(J_2Y)]$ and hence 
$\C[M(4,2k)]^{Sp(2)\times Sp(k)}=\C[P_1,P_2]$, since $P_1=\phi^*(Pf(Y))$ and $P_2=\phi^*(\mathrm{tr}(J_2Y))$. 
The fundamental relative invariants are algebraically independent (\cite[\S 4, Lemma 4]{Sato-Kimura}, \cite[Lemma 2.8]{KimuraBook}) and this implies that $\C[M(4,2k)]^{Sp(J_2)\times Sp(J_k)}$ is isomorphic to the polynomial ring of $2$ variables. 
\end{proof}

The polynomial $P_1$ is of degree $4$ and the polynomial $P_2$ is a quadratic form of signature $(4k,4k)$. 
By Lemma \ref{lem:sp invariant}, the Clifford quartic form is written as a linear combination of  $P_1$ and $P_2^2$. Indeed we have 
\[
\tilde P(w) = -16 P_1(w)+P_2(w)^2.
\]
Note that $P_1(w)$ is the irreducible relative invariant of the prehomogeneous vector space $(GL(4)\times Sp(k), \Lambda_1\otimes\Lambda_1, M(4,2k))$ and is viewed as the Clifford quartic form for $(p,q)=(3,3)$.  
Hence $P_1$ also satisfies a local functional equation. 
The polynomial $P_2^2$ also satisfy a local functional equation, since $P_2$ is a quadratic form. 
We define the local zeta functions $\tilde\zeta_{1,\pm}(f;s,\Psi)$ (resp.\ $\tilde\zeta_{2,+}(f;s,\Psi)$) for $P_1$ (resp.\  $P_2^2$) as in \S 2. 
Then we obtain the following functional equations from Theorem \ref{thm:2.14} for $P_1$ and Theorem \ref{lem:LFE for quad form} for $P_2^2$:
\begin{eqnarray*}
\begin{pmatrix}
\tilde\zeta_{1,+}\left(s,\hat\Psi\right)  \\
\tilde\zeta_{1,-}\left(s,\hat\Psi\right)  
\end{pmatrix}
 &=& 2^{4s+4k}\pi^{-4s -4k-2} \Gamma (s+1) 
          \Gamma \left(s+3 \right) \Gamma\left (s+2k-2 \right) 
          \Gamma \left(s+2k\right) \\
 & & {} \times     \begin{pmatrix}
                     \sin( \pi s)^2 & 0  \\
                       0 & \sin( \pi s)^2
                   \end{pmatrix} 
                  \begin{pmatrix}
                  \tilde\zeta_{1,+}\left(-2k - s,\Psi\right)  \\
                  \tilde\zeta_{1,-}\left(-2k - s,\Psi\right)  
                   \end{pmatrix}, \\
\tilde\zeta_{2,+}\left(s,\hat\Psi\right) 
 &=&  - 2^{4s+4k-1}\pi^{-4s -4k-2}   \Gamma \left(s+\frac12\right) 
          \Gamma \left(s+1 \right) \Gamma\left (s+2k \right) 
          \Gamma \left(s+2k+\frac12\right) \\
 & & \times 
    \sin (2\pi s)     \tilde\zeta_{2,+}\left(-2k - s,\Psi\right). 
\end{eqnarray*}

\subsection{The Cilifford quartic forms are homaloidal}

A homogeneous rational function $f$ on a fininte-dimensional vector space $\mathbf V$ is called {\it homaloidal}, if
\begin{quote}
$v \mapsto \mathrm{grad}\,f(v)$ defines a birational mapping of $\mathbb P(\mathbf V) \longrightarrow \mathbb P(\mathbf V^*)$.
\end{quote}
For a homaloidal homogeneous rational function $f$, there exists a rational function $f^*$ satisfying the identity $f^*(\mathrm{grad}\,\log f(v))=1/f(v)$, which is called the {\it multiplicative Legendre transform}\/ of $f$. 
Following \cite{Dolgachev}, we call a  polynomial $f$ a {\it homaloidal EKP-polynomial}\/ if $f$ is homaloidal and its multiplicative Legendre transform $f^*$ is also a polynomial.  

By definition, a regular prehomogeneous vector space has homaloidal relative invariant polynomials. 
As we mentioned in the introduction, it is rather difficult to construct homaloidal polynomials that are not relative invariants of prehomogeneous vector spaces and the classification of homaloidal polynomials has been done only for some special cases:
\begin{itemize}
\item
Cubic homaloidal EKP-polynomials are classified by Etingof-Kazhdan-Polishchuk  (\cite{EKP}). 
\item 
Homaloidal polynomials in 3 variables without multiple factors are classified by Dolgachev (\cite{Dolgachev}).
\item 
In \cite{Bruno}. Bruno determined when a product of linear forms is homaloidal. 
\end{itemize}
All the homaloidal polynomials classified in these works are relative invariants of prehomogeneous vector spaces and 
 Etingof, Kazhdan and Polishchuk  (\cite[\S 3.4, Question 1]{EKP}) asked whether homaloidal EKP-polynomials are relative invariants of regular prehomogeneous vector spaces. 

Now we show that Clifford quartic forms are homaloidal EKP-polynomials and give a negative answer to the question raised by Etingof, Kazhdan and Polishchuk.

\begin{thm}
\label{thm:homaloidal}
Let $S_1,\ldots,S_{p+q}$ be the basis matrices of a representation $R_{p,q}$. 
Then the Clifford quartic form $\tilde P(w)=S_1[w]^2+\cdots+S_{p+q}[w]^2$ is a homaloidal EKP-polynomials. The multiplicative Legendre transform of $\tilde P$ coincides with $\tilde P$ itself up to a constant factor. 
\end{thm}

\begin{proof}
It is easy to see that 
\[
\mathrm{grad}\, \tilde P(w) ={}^t\left(\frac{\partial \tilde  P(w)}{\partial w_1},\ldots,\frac{\partial P(w)}{\partial w_m}\right) 
 = 2^2 \sum_{i=1}^{p+q} \epsilon_i S_i[w]\cdot S_iw, 
\]
where $\epsilon_i$ $(1\leq i \leq p+q)$ are defined by $(\ref{eqn:sign})$. 
Let us calculate 
\[
\tilde P(\mathrm{grad}\, \tilde P(w)) 
 = \sum_{i=1}^{p+q} \epsilon_i S_{i}[\mathrm{grad}\, \tilde P(w)]^2 
\]
For any $i$, we have 
\begin{eqnarray*}
2^{-8}S_{i}[\mathrm{grad}\, \tilde P(w)]
 &=& S_{i}\left[\sum_{j=1}^{p+q} \epsilon_j S_j[w] S_jw\right] \\
 &=& \sum_{j=1}^{p+q} S_j[w]^2 (S_jS_{i}S_j)\left[w\right] 
   + \sum_{1\leq j < k \leq p+q} \epsilon_j \epsilon_k S_j[w]S_k[w] (S_jS_{i}S_k+S_kS_{i}S_j)\left[w\right]. 
\end{eqnarray*}
From the commutation relations $(\ref{eqn:clifford cond1})$ and $(\ref{eqn:clifford cond2})$, it follows that
\begin{eqnarray*}
S_jS_{i}S_j&=&\begin{cases}
              S_i & (i=j), \\
              -\epsilon_i\epsilon_jS_i & (i \ne j),
              \end{cases}\\
S_jS_{i}S_k+S_kS_{i}S_j
             &=&\begin{cases} 
                   0 & (j \ne k,\ j \ne i, k \ne i), \\
                   2S_j & (k = i), \\  
                   2S_k & (j = i).  
                   \end{cases}
\end{eqnarray*}
Hence 
\begin{eqnarray*}
2^{-8}S_{i}[\mathrm{grad}\,\tilde  P(w)]
 &=&  S_i[w]^3 -\epsilon_i S_i[w] \sum_{j\ne i} \epsilon_j S_j[w]^2
   + 2\epsilon_i S_i[w]\sum_{j\ne i} \epsilon_j S_j[w]^2 \\
 &=& S_i[w]^3 + \epsilon_i S_i[w] \sum_{j\ne i} \epsilon_j S_j[w]^2 \\
 &=& \epsilon_i S_i[w] \tilde P(w). 
\end{eqnarray*}
Thus we obtain 
\[
\tilde P(\mathrm{grad}\,\tilde  P(w)) 
 = 2^8 \tilde P(w)^3. 
\]
In other words, 
\[
\tilde P(\mathrm{grad}(\log \tilde P)(w)) = \tilde P(\tilde P(w)^{-1}\mathrm{grad}\,\tilde P(w)) =2^8 \tilde P(w)^{-1}.
\]
This shows that the multiplicative Legendre transform of $\tilde P$ coincides with $\tilde P$ itself (up to constant factor). 
Consequently, By \cite[Proposition 3.6]{EKP}, $\tilde P$ is a homaloidal EKP-polynomial. 
\end{proof}

\begin{rem}
\rm
Non-prehomogeneous homaloidal polynomials are constructed by Ciliberto, Russo, and Simis (\cite[\S 3]{CRS}). (In \cite[\S 4]{CRS} the authors proved that the determinants of subHankel-matrices are homaloidal. But the polynomials are also relative invariants of non-reductive regular prehomogeneous vector spaces.) 
Another examples of non-prehomogeneous (reducible) homaloidal rational functions are given by Letac and Massam (\cite{LM}). 
\end{rem}


\section{Proof of Lemma \ref{lem:key}}

\label{section:proof of key}
 
We prove Lemma \ref{lem:key} by induction on $n=p+q$ and complete the proof of Theorem \ref{thm:3.1}. 
It is obvious that the condition $(\sufcond)$ in Lemma \ref{lem:sufficient} holds for $p+q=1$. 
Assume that $p+q=2$. 
Unless $p=q=1$ and $S_1=\pm S_2$, then $S_1[x]$ and $S_2[x]$ are coprime  and the condition $(\sufcond)$ holds.   
When $p+q \geq 3$, Lemma \ref{lem:key} is equivalent to the following:
\begin{lem}
\label{lem:5.1}
If $p \geq q \geq 0$, $p+q \geq 3$ and $m > 4(p+q)-8$,  then the condition $(\sufcond)$  in Lemma \ref{lem:sufficient} holds for every $m$-dimensional representation of $R_{p,q}$. 
\end{lem}

Indeed, denoting by $m_0$ the minimal dimension of representations of $R_{p,n-p}$   $(p=0,1, \dots , n)$, we have by Lemma \ref{lem:T and T'}
\[
\begin{array}{|c|c|c|c|c|c|c|c|c|c|c|c|} \hline 
n=p+q & 3 & 4 & 5 & 6 & 7 & 8 & 9 & 10 & 11 & 12 & \cdots \\\hline 
4(p+q)-8 & 4 & 8 & 12 & 16 & 20 & 24 & 28 & 32 & 36 & 40 & \cdots \\\hline 
m_0 & 2 & 4 & 8 & 8 & 16 & 16 & 16 & 16 & 32 & 64 & \cdots \\\hline 
\end{array} 
\]
Moreover $m_0 > 4(p+q)-8$ for $n=p+q \geq 12$. 
Hence, when $p+q \geq 3$, Lemma \ref{lem:key} follows immediately from Lemma \ref{lem:5.1}.

Now we consider the case $p+q \geq 3$ with $p \geq q \geq 0$. 
Since $p \geq 2$,  we may choose the basis matrices $S_1,\ldots,S_{p+q}$ as given in Lemma \ref{lem:canonical form}.
Then, the assumption $\sum_{i=1}^{p+q} S_i[w]X_i[w] = 0$ in $(\sufcond)$ 
 can be written as follows:
\begin{equation}
\{(u,u) -(v,v)\} X_1 [w] + 2 \sum_{i=2}^{p} (u,B_i v) X_i [w] 
 + \sum_{j=1}^q (A_j [u] +A_j [v] ) Y_j [w] =0, \label{eqn:5.1}
\end{equation}
where we put $w= \begin{pmatrix} u \\ v \end{pmatrix}$ $(u, v \in {\Bbb R}^d)$ with $d=m/2$ and $Y_j=X_{p+j}$.
We write the matrices $X_1, \dots, X_p, Y_1, \dots, Y_q$ in the form
\begin{gather*}
X_i=\left( \begin{array}{cc}
X_i^{(1)} & X_i^{(2)} \\
{}^t X_i^{(2)} & X_i^{(3)} 
\end{array} \right) , \quad X_i^{(1)} , X_i^{(3)} \in {\rm Sym}(d,{\Bbb R}),\ X_i^{(2)} \in M(d,{\Bbb R}), \\
Y_j=\left( \begin{array}{cc}
Y_j^{(1)} & Y_j^{(2)} \\
{}^t Y_j^{(2)} & Y_j^{(3)} 
\end{array} \right) , \quad Y_j^{(1)} , Y_j^{(3)} \in {\rm Sym}(d,{\Bbb R}),\ Y_j^{(2)} \in M(d,{\Bbb R}).
\end{gather*}
Then the identity (\ref{eqn:5.1}) is equivalent to the following 5 identities:
\begin{gather}
(u,u) X_1^{(1)} [u] + \displaystyle{\sum_{j=1}^q} A_j [u] Y_j^{(1)} [u] = 0 , 
\label{eqn:5.2} \\
         -(v,v) X_1^{(3)} [v] + \displaystyle{\sum_{j=1}^q} A_j [v] Y_j^{(3)} [v] = 0, \label{eqn:5.3} \\
(u,u) X_1^{(3)} [v] -(v,v)X_1^{(1)} [u] + 4 \sum_{i=2}^p  (u, B_i v) (u, X_i^{(2)} v) 
 + \sum_{j=1}^q (A_j [u] Y_j^{(3)} [v] +A_j [v] Y_j^{(1)} [u] ) = 0, \nonumber \\
(u,u) (u,X_1^{(2)} v) +(u,v) X_2^{(1)} [u] + \sum_{i=3}^p (u,B_i v) X_i^{(1)} [u] 
 + \sum_{j=1}^q A_j [u](u,Y_j^{(2)} v)  = 0, \label{eqn:5.4} \\
-(v,v) (u,X_1^{(2)} v) +(u,v) X_2^{(3)} [v] + \sum_{i=3}^p (u,B_i v) X_i^{(3)} [v] 
 + \sum_{j=1}^q  A_j [v] (u, Y_j^{(2)} v ) = 0. \label{eqn:5.5}
\end{gather}
It is convenient to rewrite the third identity as follows:
\begin{eqnarray}
\lefteqn{
-4 \sum_{i=2}^p (u, B_i v) (u, X_i^{(2)} v) 
}\nonumber\\
 & & = (u,u) X_1^{(3)} [v] -(v,v) X_1^{(1)} [u] 
   + \sum_{j=1}^q (A_j [u] Y_j^{(3)} [v] +A_j [v] Y_j^{(1)} [u]) . \label{eqn:5.6}
\end{eqnarray}

The following lemma is a consequence of the identities $(\ref{eqn:5.2})$, $(\ref{eqn:5.3})$ and $(\ref{eqn:5.6})$. 

\begin{lem}
\label{lem:5.2} 
There exist $a_{ij} \in{\Bbb R}$ satisfying $a_{ij} +a_{ji}=0$ and
\begin{gather} 
X_1^{(1)} = \displaystyle{\sum_{j=1}^q} a_{1, p+j} A_j , \quad 
Y_j^{(1)} =a_{p+j,1} 1_d + \displaystyle{\sum_{k=1}^q} a_{p+j, p+k} A_k \quad (1 \leq j \leq q),  \label{eqn:5.7}
\\
X_1^{(3)} =\displaystyle{\sum_{j=1}^q} a_{1,p+j} A_j, \quad 
Y_j^{(3)} =-a_{p+j,1} 1_d + \displaystyle{\sum_{k=1}^q} a_{p+j, p+k} A_k
\quad (1 \leq j \leq q), \label{eqn:5.8}
\\
X_i^{(2)} = \displaystyle{\sum_{j=2}^p} a_{ij} B_j \quad (2 \leq i \leq p). \label{eqn:5.9}
\end{gather}
\end{lem}   

\begin{proof}
First note that, from $(\ref{eqn:5.7})$ and $(\ref{eqn:5.8})$, we see  that the right-hand side of (\ref{eqn:5.6}) is identically zero and 
\[
\sum_{i=2}^p (u,B_i v) (u, X_i^{(2)} v) = 0. 
\]
This implies the identity (\ref{eqn:5.9}), since, by the induction assumption, the condition $(\sufcond)$ holds for the $2d$-dimensional representation of $C_{p-1}$ determined by $S_2, \dots , S_p$.
Hence it is sufficient to prove the identities (\ref{eqn:5.7}), (\ref{eqn:5.8}). 
In the subsequent discussion, we have to distinguish two cases, namely, the case where $q=1$ and $A_1 = \pm 1_d$ and the case where $q \ne 1$ or $A_1 \ne \pm 1_d$. 
We refer to the first case as Case A and the second case as Case B. 
\\[2ex]
{\bf Case A:} We put $A_1= \varepsilon 1_d$ $(\varepsilon = \pm 1)$. 
Then the identities (\ref{eqn:5.2}) and (\ref{eqn:5.3}) become 
\[
(u,u) X_1^{(1)} [u] + \varepsilon (u,u) Y_1^{(1)} [u] =0, \quad 
-(v,v) X_1^{(3)} [v] + \varepsilon (v,v) Y_1^{(3)} [v] =0.
\]
This implies that
\begin{equation}
Y_1^{(1)} = - \varepsilon X_1^{(1)} , \quad Y_1^{(3)} = \varepsilon X_1^{(3)}. \label{eqn:5.10}
\end{equation}
We substitute (\ref{eqn:5.10}) to (\ref{eqn:5.6}) to get
\[
-2 \sum_{i=2}^p (u,B_i v) (u,X_i^{(2)} v) = (u,u) X_1^{(3)} [v] -(v,v) X_1^{(1)} [u].
\]
Fix $v$ and consider the left-hand side of this identity as a quadratic form of $u$. 
Then the rank of the quadratic form is not greater than $2(p-1)$. 
Since $d > 2(p+1)-4 =2(p-1)$ by the assumption of Lemma \ref{lem:5.1}, 
it is a degenerate quadratic form of $u$.  
Hence by the expression on the right-hand side we have $\det (X_1^{(3)} [v] 1_d -(v,v) X_1^{(1)} ) =0.$ 
Since $v$ is arbitrary, this shows that $X_1^{(3)} = \lambda 1_d$ for some eigenvalue $\lambda$ of $X_1^{(1)}$, namely, $X_1^{(3)}$ is a scalar matrix. 
Similarly we can prove that $X_1^{(1)} $ is also a scalar matrix and we get 
\[
X_1^{(1)} =a_{1,p+1} A_1, \quad X_1^{(3)} = b_{1, p+1} A_1
\]
for some constants $a_{1,p+1}, b_{1, p+1}$. 
Then, since $(u,u) X_1^{(3)} [v] -(v,v) X_1^{(1)} [u] $ is degenerate, we have 
\[
a_{1,p+1} =b_{1,p+1}. 
\]
Furthermore (\ref{eqn:5.10}) implies that
\[
Y_1^{(1)} = - \varepsilon a_{1,p+1} A_1 = -a_{1, p+1} 1_d, \quad 
 Y_1^{(3)} = \varepsilon a_{1,p+1} A_1 = a_{1, p+1} 1_d.
\]
Hence putting $a_{p+1, 1}=-a_{1, p+1}$, we obtain 
\[
Y_1^{(1)} =a_{p+1,1} 1_d , \quad Y_1^{(3)} =-a_{p+1, 1} 1_d 
\]
and $a_{p+1, p+1} =b_{p+1, p+1}=0$.  
This proves the identities (\ref{eqn:5.7}) and (\ref{eqn:5.8}).
\\[2ex]
{\bf Case B:}  
In this case, since $ q \ne 1$ or $ A_1 \neq \pm 1_d$, by the induction assumption, 
the $d$-dimensional representation of $R_{1,q}$ determined by $1_d, A_1, \dots , A_q$ satisfies the condition $(\sufcond)$.   
Hence from the identities (\ref{eqn:5.2})  and (\ref{eqn:5.3}) we have
\begin{gather}
X_1^{(1)} = \sum_{j=1}^q a_{1, p+j} A_j , \quad 
Y_j^{(1)} = a_{p+j, 1} 1_d + \sum_{k=1}^q a_{p+j, p+k} A_k  \quad (1 \leq j \leq q), \label{eqn:5.11}\\
X_1^{(3)} = \sum_{j=1}^q b_{1,p+j} A_j , \quad 
Y_j^{(3)} =-b_{p+j, 1} 1_d + \sum_{k=1}^q b_{p+j, p+k} A_k   \quad (1 \leq j \leq q), \label{eqn:5.12}
\end{gather}
where $a_{ij}, b_{ij} $ satisfy that $ a_{ij} +a_{ji}=0$, $b_{ij} +b_{ji}=0$. 
We rewrite the right-hand side of (\ref{eqn:5.6}) by using (\ref{eqn:5.11}) and (\ref{eqn:5.12}) to get
\begin{eqnarray}
\lefteqn{
-4 \sum_{i=2}^p (u,B_i v) (u, X_i^{(2)} v) 
}\nonumber\\
 & & = - (u,u) \sum_{j=1}^q (a_{p+j,1} + b_{1, p+j} ) A_j [v] \nonumber\\
 & &  \hspace{5mm} {} + \sum_{j=1}^q A_j [u] \left\{(a_{1,p+j} + b_{p+j,1} ) (v,v) 
    - \sum_{k=1}^q (a_{p+j, p+k} + b_{p+k, p+j} ) A_k [v] \right\}.
\label{eqn:5.13}
\end{eqnarray}
As in Case A, the quadratic form of $u$ (for an arbitrarily fixed $v$) defined by the left-hand side of  (\ref{eqn:5.13})  is degenerate. 
Indeed, if $q=0$, then the right-hand side vanishes, and 
if $q>0$, then the rank of the quadratic form is not greater than $2(p-1) \leq 2(p+q)-4 < d$.  
By the self-duality of the quadratic mapping defined by $1_d, A_1, \dots , A_q$, the linear combination of $(u,u), A_1 [u], \dots , A_q [u]$ on the right-hand side of (\ref{eqn:5.13}) is degenerate, only when  
\begin{eqnarray}
\lefteqn{
\left(\sum_{j=1}^q (a_{p+j,1} + b_{1,p+j}) A_j [v]  \right)^2 
} \nonumber\\
 & & - \sum_{j=1}^q \left\{ (a_{1,p+j} + b_{p+j,1} ) (v,v) - \sum_{k=1}^q (a_{p+j,p+k} +b_{p+k, p+j} ) A_k [v] \right\}^2 =0. \label{eqn:5.14}
\end{eqnarray}
Furthermore in Case B the quadratic mapping defined by $1_d, A_1, \dots , A_q$ is non-degenerate (see Theorem \ref{thm:degenerate}), and hence the quadratic forms $(u,u), A_1 [u], \dots , A_q [u]$ are algebraically independent. 
Therefore, comparing the coefficients of $(v,v)^2$ on 
the both sides of (\ref{eqn:5.14}), we obtain
\[
\sum_{j=1}^q (a_{1,p+j} +b_{p+j,1} )^2 =0.
\]
Since the constants $a_{1,p+j}$, $b_{p+j,1}$ are real numbers, we have
\[
a_{1,p+j} +b_{p+j,1} =0 \quad (1 \leq j \leq q).
\]
and, by the skew-symmetry of $a_{ij}$, $b_{ij}$, 
\[
a_{p+j,1} +b_{1,p+j}=0 \quad (1 \leq j \leq q).
\]
Hence, by (\ref{eqn:5.14}),  we get 
\[
\sum_{j=1}^q \left\{\sum_{k=1}^q (a_{p+j,p+k} +b_{p+k, p+j}) A_k [v] \right\}^2 = 0.
\]
This implies 
\[
\sum_{k=1}^q  (a_{p+j, p+k} +b_{p+k,p+j} ) A_k [v] =0 \quad (1 \leq j \leq q) .
\]
Since the basis matrices $A_1 , \dots , A_q$ are linearly independent, we have
\[
a_{p+j,p+k} +b_{p+k,p+j} =0 \quad (1 \leq j,k \leq q).
\]
By using the skew-symmetry of $a_{ij}, b_{ij}$, we also have 
\begin{equation}
a_{p+j,1}=b_{p+j,1}, \quad a_{p+j,p+k} =b_{p+j,p+k}. \label{eqn:5.15}
\end{equation}
Now the identities (\ref{eqn:5.7}) and (\ref{eqn:5.8}) follow from (\ref{eqn:5.11}), (\ref{eqn:5.12}) and (\ref{eqn:5.15}). 
\end{proof}
  
To proceed further, we need the following lemma.

\begin{lem}
\label{lem:5.3} 
We assume that the identity 
\begin{equation}
X[u] (u,v) -(u,u)(u,Xv) = 
\sum_{i=1}^r Y_i [u] (u,B_i v) \label{eqn:5.16}
\end{equation}
holds for $X, Y_1 , \dots , Y_r \in {\rm Sym}(d,{\Bbb R})$ and  $B_1, \dots , B_r \in {\rm Alt}(d,{\Bbb R})$.  
Let $s$ be the number of symmetric matrices $Y_i$ that are not scalar matrices.  
If $d > 2s$, 
then $X$ is a scalar matrix and the both sides of {\rm (\ref{eqn:5.16})} are identically zero.
\end{lem}
 
\begin{proof}
We may assume that $X$ is a diagonal matrix with diagonal entries $\lambda_1 \geq \cdots  \geq \lambda_d$. 
Then, since 
\[
(X[u] 1_d -(u,u) X) u = -\displaystyle{ \sum_{i=1}^r } Y_i [u] B_i u, 
\]
by comparing the $k$-th entries of the left- and right-hand sides, we have 
\[
\left(\sum_{i \neq k} ( \lambda_i - \lambda_k ) u_i^2 \right) u_k 
  = - \sum_{i=1}^r Y_i [u] \psi_{ik} (u), 
\]
where $ \psi_{ik} (u) $ is the linear form of $u$ which gives the $k$-th entry of $B_i u$.  
The left-hand side of this identity is of degree 1 with respect to $u_k$. 
Since $B_i $ is a skew-symmetric matrix, $u_k$ does not appear in $\psi_{ik} (u)$. 
We consider $Y_i [u]$ as a quadratic polynomial of $u_k$ and denote by $\phi_{ik} (u)$ the linear form of $u_1,\ldots,u_{k-1},u_{k+1},\ldots,u_d$ appearing as the coefficient of $u_k$. 
We may assume that $Y_i$ are scalar matrices for $i \geq s+1$. 
Then $\phi_{ik}(u)=0$ for $i \geq s+1$ and we have
\[
{\sum_{i \neq k} } ( \lambda_i - \lambda_k ) u_i^2 
 = -\sum_{i=1}^r  \psi_{ik} (u) \phi_{ik} (u). 
\]
Note that the number of positive eigenvalues and that of negative eigenvalues of the quadratic form on the right-hand side are not greater than $r$. 
Consider the case where $k=1$. 
Then the left-hand side of this identity is negative semidefinite. 
Hence the multiplicity of $\lambda_1 $ is not smaller than $d-r$.  
Similarly, if we consider the case where $k=d$, the left-hand side of this identity is positive semidefinite and the multiplicity of $\lambda_d$ is at least $d-r$.  
Therefore, if $X$ has distinct eigenvalues, then $d \geq 2(d-r) $ namely $2r \geq d$.
This contradicts the assumption $d > 2r$. 
Hence,  $X$ is a scalar matrix and the both sides of (\ref{eqn:5.16}) are identically zero. 
\end{proof}

Let us return to the proof of Lemma \ref{lem:5.1} and examine the identities  $(\ref{eqn:5.4})$ and $(\ref{eqn:5.5})$.
Substitute $v=u$ to (\ref{eqn:5.4}) and  (\ref{eqn:5.5}). 
Then, since $B_i$ $(i \geq 3)$ are skew-symmetric matrices and hence $(u, B_i u) =0$,  we obtain 
\begin{gather}
(u,u) ( X_1^{(2)} + X_2^{(1)} ) [u] + \displaystyle{\sum_{j=1}^q } A_j [u] Y_j^{(2)} [u] =0,  \label{eqn:5.17}
\\
(u,u) ( X_2^{(3)} - X_1^{(2)} ) [u] + \displaystyle{\sum_{j=1}^q } A_j [u] Y_j^{(2)} [u] =0.  
\label{eqn:5.18}
\end{gather}
Here we have to treat Case A and Case B separately as in the proof of Lemma \ref{lem:5.2}. 
\\[2ex]
{\bf Case A:} 
In this case, $q=1$, $A_1 = \varepsilon 1_d$ $(\varepsilon = \pm 1)$  
and the identities (\ref{eqn:5.17}) and (\ref{eqn:5.18}) become
\[
(u,u) (X_1^{(2)} +X_2^{(1)} + \varepsilon Y_1^{(2)} ) [u] =0, \quad 
 (u,u) (X_2^{(3)} -X_1^{(2)} + \varepsilon Y_1^{(2)} ) [u] =0.
\]
Hence, putting 
\begin{equation}
Z^+ =X_1^{(2)} + \varepsilon Y_1^{(2)} , \quad 
Z^- =X_1^{(2)} - \varepsilon Y_1^{(2)}, \quad
Z_s^{\pm} ={1 \over 2} ( Z^{\pm} + {}^t Z^{\pm}), \quad 
Z_a^{\pm} ={1 \over 2} (Z^{\pm} - {}^t Z^{\pm} ), \label{eqn:5.19}
\end{equation}
we have 
\[
Z_s^{+} =-X_2^{(1)} , \quad Z_s^- =X_2^{(3)}. 
\]
This yields that
\begin{equation}
X_1^{(2)}=Z_a^{+}-X_2^{(1)}-\varepsilon Y_1^{(2)}
           =Z_a^{-}+X_2^{(3)}+\varepsilon Y_1^{(2)}. \label{eqn:5.20}
\end{equation}
By substituting (\ref{eqn:5.20}) to (\ref{eqn:5.4}) and (\ref{eqn:5.5}), we get 
\begin{gather*}
(u,u) (u, Z_a^+ v) +(u,v) X_2^{(1)} [u] -(u,u) (u,X_2^{(1)} v) 
 + \sum_{i=3}^p (u,B_i v) X_i^{(1)} [u] =0, \\
-(v,v) (u, Z_a^- v) +(u,v) X_2^{(3)} [v] -(v,v) (u,X_2^{(3)} v) 
 + \sum_{i=3}^p (u,B_i v) X_i^{(3)} [v] =0. 
\end{gather*}
We rewrite these identities as follows: 
\begin{gather*}
(u,v) X_2^{(1)} [u] -(u,u) (u,X_2^{(1)} v) 
 = - (u,u) (u,Z_a^+ v) - \sum_{i=3}^p (u,B_i v) X_i^{(1)} [u], \\
(u,v) X_2^{(3)} [v] -(v,v) (u,X_2^{(3)} v) 
 = (v,v) (u,Z_a^- v) - \sum_{i=3}^p (u,B_i v) X_i^{(3)} [v]. 
\end{gather*}
Lemma \ref{lem:5.3} can apply to these identities, 
since  $s$ in Lemma \ref{lem:5.3} is not greater than $p-2$ and 
$d > 2(p+q) -4 > 2s$.  
Thus we obtain for some $\alpha,\beta$ 
\begin{gather}
X_2^{(1)} = \alpha 1_d, \quad X_2^{(3)} =\beta 1_d, \label{eqn:5.21} \\
(u,u) (u, Z_a^+ v) + \sum_{i=3}^p (u, B_i v) X_i^{(1)} [u] =0, \label{eqn:5.22}\\
(v,v) (u, Z_a^- v) - \sum_{i=3}^p (u, B_i v) X_i^{(3)} [v] =0. \label{eqn:5.23}
\end{gather}
Summing (\ref{eqn:5.22}) (resp.\ (\ref{eqn:5.23})) and the identity obtained by exchanging $u$ and $v$ in (\ref{eqn:5.22}) (resp.\ (\ref{eqn:5.23})), we get 
\begin{gather}
\left\{ (u,u)-(v,v) \right\} (u,Z_a^+ v) + \sum_{i=3}^p (u,B_i v) (X_i^{(1)} [u] -X_i^{(1)} [v]) =0, \label{eqn:5.24} \\
\left\{ (u,u)-(v,v) \right\} (u,Z_a^- v) 
 - \sum_{i=3}^p (u,B_i v) (X_i^{(3)} [u] -X_i^{(3)} [v]) =0. \label{eqn:5.25}
\end{gather} 
Since the condition $(\sufcond)$ holds for the $2d$-dimensional representation of $R_{p-1,0} $ determined by $S_1, S_3, \dots, S_p $ by induction hypothesis, 
there exist constants $\gamma_{i,1} , \gamma_{1,i} , \delta_{i,1} , \delta_{1,i} \in {\Bbb R}$ $(3 \leq i \leq p)$ satisfying 
$ \gamma_{i,1}=-\gamma_{1,i}$, $\delta_{i,1}=-\delta_{1,i}$ and  
\begin{gather}
Z_a^+ =\displaystyle{\sum_{i=3}^p} \gamma_{1,i} B_i , \quad 
Z_a^- = \sum_{i=3}^p \delta_{1,i} B_i, \label{eqn:5.26} \\ 
X_i^{(1)} =\gamma_{i,1} 1_d , \quad X_i^{(3)} = -\delta_{i,1} 1_d \quad (3 \leq i \leq p). \label{eqn:5.27}
\end{gather}
From (\ref{eqn:5.20}), (\ref{eqn:5.21}) and (\ref{eqn:5.26}), we have 
\[
Z^+ =- \alpha 1_d + \displaystyle{ \sum_{i=3}^p } \gamma_{1,i} B_i ,  \quad 
Z^- = \beta 1_d + \displaystyle{ \sum_{i=3}^p} \delta_{1,i} B_i. 
\]
Then, since $X_1^{(2)} ={1 \over 2} (Z^+ + Z^-)$, $Y_1^{(2)} ={ {\varepsilon} \over 2} (Z^+ -Z^- )$,  we obtain 
\begin{equation}
X_1^{(2)} =\displaystyle{\sum_{i=2}^p} a_{1,i} B_i , \quad Y_1^{(2)} =\displaystyle{ \sum_{i=2}^p} a_{p+1, i} B_i , 
\label{eqn:5.28}
\end{equation}
where we put 
\begin{gather*} 
a_{1,2}={1 \over 2} (\beta - \alpha), \quad a_{p+1, 2} 
 = -{{\varepsilon} \over 2} (\alpha + \beta) \\
a_{1,i} = {1 \over 2}(\gamma_{1,i} + \delta_{1,i} ) , \quad 
a_{p+1, i} = {{\varepsilon} \over 2} ( \gamma_{1,i} - \delta_{1,i}) \quad (3 \leq i \leq p). 
\end{gather*}
Now the identity (\ref{eqn:5.27}) takes the form
\begin{equation}
X_i^{(1)} =(a_{1,i} + \varepsilon a_{p+1, i} ) 1_d , \quad X_i^{(3)} =(-a_{1,i} + \varepsilon a_{p+1, i} ) 1_d \quad (3 \leq i \leq p). 
\label{eqn:5.29}
\end{equation} 
Furthermore, by (\ref{eqn:5.21}), we have 
\[
X_2^{(1)} =-a_{1,2} 1_d -a_{p+1,2} \varepsilon 1_d , \quad 
X_2^{(3)} =a_{1,2} 1_d -a_{p+1,2} \varepsilon 1_d. 
\]
Then, putting $a_{2,1} =-a_{1,2}$, $a_{2,p+1}=-a_{p+1, 2}$, we obtain 
\begin{equation}
X_2^{(1)} =a_{2,1} 1_d + a_{2,p+1} \varepsilon 1_d , \quad 
X_2^{(3)} = - a_{2,1} 1_d + a_{2, p+1} \varepsilon 1_d . \label{eqn:5.30}
\end{equation}
Lemma \ref{lem:5.1} for Case A follows from Lemma \ref{lem:5.2}, (\ref{eqn:5.28}), (\ref{eqn:5.29}) and (\ref{eqn:5.30}).
\\[2ex]
{\bf Case B:}   We consider (\ref{eqn:5.17}) and (\ref{eqn:5.18}) for Case B, namely, under the assumption $q \ne 1$ or $A_1\ne\pm1_d$. 
Put 
\begin{gather*}
X_{1,s}^{(2)} ={1 \over 2} (X_1^{(2)} +{}^t X_1^{(2)} ) , \quad X_{1,a}^{(2)} ={1 \over 2} (X_1^{(2)} - {}^t X_1^{(2)} ) , \\
Y_{i,s}^{(2)} ={1 \over 2} (Y_i^{(2)} +{}^t Y_i^{(2)} ) , \quad Y_{i,a}^{(2)} ={1 \over 2} (Y_i^{(2)} - {}^t Y_i^{(2)} ).
\end{gather*}
Since, by the induction assumption, the condition $(\sufcond)$ holds for $d$-dimensional representations of $R_{1,q}$ for Case B,   
we have from (\ref{eqn:5.17}) and (\ref{eqn:5.18})
\begin{equation}
X_{1,s}^{(2)} +X_2^{(1)} = X_2^{(3)} -X_{1,s}^{(2)} 
 = \sum_{j=1}^q \alpha_j A_j , \quad 
Y_{i,s}^{(2)} =-\alpha_i 1_d + \sum_{j=1}^q \alpha_{ij} A_j \quad ( \alpha_{ij} =- \alpha_{ji} ). 
\label{eqn:5.31}
\end{equation}  
By substituting (\ref{eqn:5.31}) to (\ref{eqn:5.4}) and  (\ref{eqn:5.5}), we obtain
\begin{eqnarray}
\lefteqn{
(u,u) (u, X_{1,a}^{(2)} v) +(u,u) ( u, X_{1,s}^{(2)} v) -(u,v) X_{1,s}^{(2)} [u] 
}\nonumber \\
 & & {} + \sum_{i=3}^p (u,B_i v) X_i^{(1)} [u] 
          + \sum_{i=1}^q A_i [u] (u, Y_{i,a}^{(2)} v) 
          + \sum_{i,j=1}^q \alpha_{ij} A_i [u] (u, A_j v) = 0, 
\label{eqn:5.32} \\
\lefteqn{
 -(v,v) (u, X_{1,a}^{(2)} v) -(v,v) ( u, X_{1,s}^{(2)} v) +(u,v) X_{1,s}^{(2)} [v] 
} \nonumber\\
 & & {} + \sum_{i=3}^p (u,B_i v) X_i^{(3)} [v] 
          + \sum_{i=1}^q A_i [v] (u, Y_{i,a}^{(2)} v) 
          + \sum_{i,j=1}^q \alpha_{ij} A_i [v] (u, A_j v) =0. 
\label{eqn:5.33} 
\end{eqnarray} 
Let us prove that 
\[
\alpha_{ij} =0 \quad ( 1 \leq i,j \leq q).  
\]
Denote by $f_1(u,v)$ (resp.\ $f_2(u,v)$) the left-hand side of (\ref{eqn:5.32}) (resp.\ (\ref{eqn:5.33})) and consider $f_1(u,v) + f_1(v,u) + f_2(u,v) + f_2(v,u)$. 
Then we obtain
\begin{eqnarray*}
\lefteqn{
2\bigl((u,u)-(v,v)\bigr) (u,X_{1,a}^{(2)}v) 
} \\
 & & {} + \sum_{i=3}^p (u,B_i v)\{(X_i^{(1)}-X_i^{(3)})[u] - (X_i^{(1)}-X_i^{(3)})[v] \}  \\
 & & {} + 2\sum_{i,j=1}^q \alpha_{ij} (A_i[u]+A_i[v])(u,A_jv)= 0 
\end{eqnarray*}
This can be written as
\begin{gather*}
S_1[w]T_1[w]+\sum_{i=3}^p S_i[w]T_i[w]+\sum_{j=1}^q S_{p+j}[w]T_{p+j}[w]=0,\\
T_1 = \begin{pmatrix} 0 & X_{1,a}^{(2)} \\ {}^tX_{1,a}^{(2)} & 0 \end{pmatrix}, \quad
T_i = \begin{pmatrix} \frac12 \left(X_i^{(1)}-X_i^{(3)}\right) & 0 \\ 
                         0 & -\frac12 \left(X_i^{(1)}-X_i^{(3)}\right) \end{pmatrix}\ 
 (3 \leq i \leq p),  \\
T_{p+j} = \sum_{k=1}^q \alpha_{jk} 
          \begin{pmatrix} 0 & A_k \\ A_k & 0 \end{pmatrix}\ 
 (1 \leq j \leq q).
\end{gather*}
Since, by the induction assumption, the condition $(\sufcond)$ holds for $2d$-dimensional representations of $R_{p-1,q}$, 
every $T_i$ $(i\ne2)$ is a linear combination of $S_1,S_3,\ldots,S_{p+q}$. 
Then, for $j$ $(1 \leq j \leq q)$, the matrix $\sum_{k=1}^q \alpha_{jk} A_k$, the off-diagonal block of $T_{p+j}$,  is a linear combination of $B_3,\ldots,B_p$, which is skew-symmetric. 
Since $\sum_{k=1}^q \alpha_{jk} A_k$ is symmetric, this equals $0$. 
The linear independence of $A_1,\ldots,A_q$ implies that $\alpha_{jk} =0$ $( 1 \leq j,k \leq q)$.   
Moreover, from (\ref{eqn:5.31}), we get 
\begin{equation}
Y_{i,s}^{(2)} =a_{p+i, 2} 1_d,  \label{eqn:5.34}
\end{equation} 
where we write $a_{p+i, 2}$ for $-\alpha_i$. 
By $\alpha_{ij}=0$, the identities   (\ref{eqn:5.32}) and (\ref{eqn:5.33}) become 
\begin{eqnarray*}
\lefteqn{
 (u,v) X_{1,s}^{(2)} [u] - (u,u) ( u, X_{1,s}^{(2)} v) 
}\\ 
 & & = (u,u) (u, X_{1,a}^{(2)} v) + \sum_{i=3}^p (u,B_i v) X_i^{(1)} [u] 
          + \sum_{i=1}^q A_i [u] (u, Y_{i,a}^{(2)} v), 
\\
\lefteqn{
 -(u,v) X_{1,s}^{(2)} [v] + (v,v) ( u, X_{1,s}^{(2)} v)
} \nonumber\\
 & & =  -(v,v) (u, X_{1,a}^{(2)} v) + \sum_{i=3}^p (u,B_i v) X_i^{(3)} [v] 
          + \sum_{i=1}^q A_i [v] (u, Y_{i,a}^{(2)} v). 
\end{eqnarray*} 
We can apply Lemma \ref{lem:5.3} to these identities, since $s$ in Lemma \ref{lem:5.3} is not greater than $p+q-2$  and $d > 2(p+q-2) \geq 2s$. 
Therefore we have 
\begin{equation}
X_{1,s}^{(2)} =a_{12}1_d 
\label{eqn:5.35}
\end{equation}
for some constant $a_{12}$, and 
\begin{eqnarray*}
(u,u) (u, X_{1,a}^{(2)} v) + \sum_{i=3}^p (u,B_i v) X_i^{(1)} [u] 
          + \sum_{i=1}^q A_i [u] (u, Y_{i,a}^{(2)} v) = 0 
\\
-(v,v) (u, X_{1,a}^{(2)} v) + \sum_{i=3}^p (u,B_i v) X_i^{(3)} [v] 
          + \sum_{i=1}^q A_i [v] (u, Y_{i,a}^{(2)} v) = 0. 
\end{eqnarray*} 
Summing these two identities, we obtain
\begin{eqnarray*}
\bigl((u,u)-(v,v)\bigr) (u, X_{1,a}^{(2)} v) 
 + \sum_{i=3}^p (u,B_i v) (X_i^{(1)} [u] +  X_i^{(3)} [v]) 
          + \sum_{i=1}^q (A_i [u]+A_i [v]) (u, Y_{i,a}^{(2)} v) = 0. 
\end{eqnarray*} 
We can rewrite this as follows:
\begin{gather*}
S_1 [w] Z_1[w] + \sum_{i=3}^{p+q} S_i [w] Z_i [w] = 0, \\
Z_1=\left( \begin{array}{cc} 
0 & X_{1,a}^{(2)} \\
{}^t X_{1,a}^{(2)} & 0 
\end{array} \right) , \quad 
Z_i=\begin{pmatrix} 
X_i^{(1)} & 0 \\
0 & X_i^{(3)} 
\end{pmatrix} \quad (3 \leq i \leq p), \\
Z_{p+j} = \begin{pmatrix} 
0 & Y_{j,a}^{(2)} \\
{}^tY_{j,a}^{(2)} & 0
\end{pmatrix}  \quad (1 \leq j \leq q).
\end{gather*}
Since any $2d$-dimensional representation of $R_{p-1,q}$ satisfies the condition $(\sufcond)$ by the induction assumption, we obtain 
\[
Z_i = \sum_{\stackrel{j=1}{j \neq 2}}^{p+q}  a_{ij} S_j \quad 
(i=1,3,\ldots,p+q)
\]
for some constants $a_{ij}$ with $a_{ij} = -a_{ji}$. 
This identity together with Lemma \ref{lem:5.2}, (\ref{eqn:5.31}), (\ref{eqn:5.34}) and  (\ref{eqn:5.35}) implies Lemma \ref{lem:5.1} for Case B.

\section{Proof of Theorems \ref{thm:3.4} and \ref{thm:3.5}}

\label{section:h_{p.q}}

In this section we calculate the Lie algebra $\gerh_{p,q}(\rho)$ and prove Theorems \ref{thm:3.4} and \ref{thm:3.5}. 
Our calculation is based on the following Lemma.

\begin{lem}
\label{lem:6.1} 
Let 
$\rho : R_{p,q} \rightarrow M(m;{\Bbb R})$ be a representation of $R_{p,q}$ such that the basis matrices $S_i= \rho (e_i)$ $(1 \leq i \leq p+q)$ are symmetric matrices. 
Put 
\[
\mathcal{A} = \calA_{p,q}(\rho) :=\set{X \in M(m;{\Bbb R})}{\rho (Y) X = X \rho (Y)\ ( \forall Y \in R_{p,q}^+ )}.
\] 
If $ r\in ( R_{p,q})^{\times} $ is an odd element, then we have 
\[
\gerH_{p,q}(\rho)=\set{X \in \calA}{{}^tX \rho(r) + \rho(r)X=0}.
\]
\end{lem}

\begin{proof}
By the definition of $\mathfrak{h}_{p,q} ( \rho )$, 
\[
\gerH_{p,q}(\rho)=\set{X \in M(m;\R)}{{}^tX S_i + S_iX=0\ (1 \leq i \leq p+q)}
\]
If $X \in \mathfrak{h}_{p,q} (\rho)$, then
\[
S_iS_jX=-S_i{}^tXS_j=XS_iS_j.
\]
Since $S_i S_j \ ( 1 \leq i <j \leq p)$ generate $ \rho (R_{p,q}^+)$,  $X$ is in $\mathcal{A}$. 
Thus $\mathfrak{h}_{p,q} ( \rho ) \subset \mathcal{A}$. 
Put  $r'=e_i r$. Then $r'$ is an element of $(R_{p,q}^+)^{\times}$. 
Hence, for $X \in \mathcal{A}$, we have
\begin{eqnarray*}
{}^tX \rho(r)+\rho(r)X
 &=& {}^tX \rho(e_i)\rho(r')+\rho(e_i)\rho(r')X \\
 &=& {}^tX \rho(e_i)\rho(r')+\rho(e_i)X\rho(r') \\
 &=& \bigl({}^tX S_i+S_iX\bigr)\rho(r').
\end{eqnarray*}
Thus, if $X \in \calA$, we have 
\begin{eqnarray*}
{}^tX \rho(r)+\rho(r)X=0\ 
 &\Longleftrightarrow& {}^tX S_i+S_iX=0\ \text{for some $i$}\\ 
 &\Longleftrightarrow& {}^tX S_i+S_iX=0\ \text{for all $i$} 
 \Longleftrightarrow X \in \gerh_{p,q}(\rho).
\end{eqnarray*}
This proves the lemma.
\end{proof}

Note that, we may take $r=e_i$. 
Then we have $\rho(r)^2=1$ and 
the map $X \mapsto -\rho(r)\,{}^tX\rho(r)^{-1}=-S_i\,{}^tXS_i$ is an involutive automorphism of $\calA$. 
Hence $(\calA,\gerh_{p,q}(\rho))$ is a symmetric Lie algebra, if we regard $\calA$ as a Lie algebra (see $(\ref{eqn:calA as matrix algebra})$ below).  

Let $T'$ and $\mathbb K'$ be as in Theorem \ref{thm:3.4}. 
It is often convenient to consider the representation space $W=\R^m$ as a  $\mathbb K'$-vector space. 
Then the algebra $\calA$ is of the form 
\begin{equation}
\label{eqn:calA as matrix algebra}
\calA=\begin{cases} 
         M(k;\mathbb K') = \mathfrak{gl}(k,\mathbb K')& \text{if $R_{p,q}^+ = T'$,}  \\
         M(k_1;\mathbb K')\oplus M(k_2;\mathbb K') 
             = \mathfrak{gl}(k_1,\mathbb K') \oplus \mathfrak{gl}(k_2,\mathbb K')
                         & \text{if $R_{p,q}^+ = T'\oplus T'$,}  
         \end{cases}
\end{equation}
where $k$, $k_1$, $k_2$ are the multiplicities of irreducible representations corresponding to the simple components of $R^+_{p,q}$ in $\rho|_{R^+_{p,q}}$.  
In the sequel we denote the transpose of $X \in \calA$ as a matrix in $M(m;\R)$ by ${}^T\!X$ to distinguish it from the transpose as a matrix in $M(k;\mathbb K')$. 
Then ${}^T\!X$ corresponds to $X^*={}^t\bar X$ in $M(k;\mathbb K')$, where $X \mapsto \bar X$ denotes the conjugate in  $\mathbb K'$. 
 
We put
\[
\hat e_p = e_1\cdots e_p \quad \hat f_q=e_{p+1}\cdots e_{p+q}.
\]
Then it is easy to check the following identities 
\begin{gather}
\hat e_p^2 = \begin{cases} 1 & (p \equiv 0,1 \pmod 4), \\ -1 & (p \equiv 2,3 \pmod 4), \end{cases} \quad 
\hat f_q^2 = \begin{cases} 1 & (q \equiv 0,1 \pmod 4), \\ -1 & (q \equiv 2,3 \pmod 4), \end{cases}
\label{eqn:6.2}\\
e_i\hat e_p=(-1)^{p-1}\hat e_p e_i \quad (i=1,\ldots,p), \quad 
e_{p+i}\hat f_q=(-1)^{q-1}\hat f_q e_{p+i} \quad (i=1,\ldots,q), 
\label{eqn:6.3}\\
e_{p+i}\hat e_p=\hat e_p e_{p+i} \quad (i=1,\ldots,q), \quad 
e_{i}\hat f_q=\hat f_q e_{i} \quad (i=1,\ldots,p). 
\nonumber
\end{gather}

Now we are in a position to prove Theorems \ref{thm:3.4} by case by case examination. 
Theorem \ref{thm:3.5} on the action of $\gerg'_{p,q}(\rho)$ is an immediate consequence of the description of $\gerh_{p,q}(\rho)$ below.  


\subsection{The case: $(T,T')$, $({\Bbb K}, {\Bbb K}')=({\Bbb R}, {\Bbb C})$}

In this case, $\{\bar{p},\bar{q}\}=\{0,2\},\{4,6\}$ and  
\[
R_{p,q}=M(2^{n/2};\R) \supset R^+_{p,q} = M(2^{(n-2)/2};\C) \quad (n=p+q). 
\]
We may assume that $p \equiv 0 \pmod 4$ and $q \equiv 2 \pmod  4$. 
Then $\hat{e}_p^2=1, \hat{f}_q^2=-1$. 
The center of $R^+_{p,q}$, which is isomorphic to $\C$,  is given by $\R+\R \bf i$ (${\bf i}:=\hat{e}_p \hat{f}_q$), 
since ${\bf i}^2=-1$ and ${\bf i}$ is a central element of $R_{p,q}^+$. 
Let $W_0$ be a simple $R_{p,q}$-module (unique up to isomorphism). 
Then $W_0$ is simple also  as an $R_{p,q}^+$-module and naturally identified with ${\Bbb C}^{2^{(n-2)/2}}$.
More generally, a not necessarily simple $R_{p,q}$-module $W=\overbrace{W_0 \oplus \cdots \oplus W_0}^k$ is identified with $M(2^{(n-2)/2}, k; {\Bbb C})$. 
Under the identification, $\mathcal{A}=M(k,{\Bbb C})$ and the action of $X \in \mathcal{A}$ on $ w \in W$ is given by
\[
X\cdot w=(w_1,\ldots,w_k)\,{}^t\!\bar X \quad
 (X \in \calA=M(k,\C),\ w=(w_1,\ldots,w_k) \in W=M(2^{(n-2)/2},k,\C)).
\]
Let us take $e_1$ as  $r$ in Lemma \ref{lem:6.1} and calculate 
\[
\gerH_{p,q}=\set{X \in \calA}{{}^TX\rho(r)+\rho(r)X=0}.
\]
For $ \alpha =a + b {\bf i} \in {\Bbb C}$, by (\ref{eqn:6.3}) , we have 
\[
e_1 \alpha = a e_1 +b e_1 \hat{e}_p \hat{f}_q 
               = ae_1 -b \hat{e}_p \hat{f}_q e_1 
               = \bar{\alpha} e_1. 
\]
From this relation, we have
\begin{eqnarray*}
\left({}^T\!X \rho(e_1) + \rho(e_1)X\right)\cdot(w_1,\ldots,w_k) 
 &=& (e_1w_1,\ldots,e_1w_k)X + e_1\bigl((w_1,\ldots,w_k)\,{}^t\bar X\bigr) \\
 &=& (e_1w_1,\ldots,e_1w_k)X + (e_1w_1,\ldots,e_1w_k)\,{}^t X \\
 &=& (e_1w_1,\ldots,e_1w_k)(X + {}^t X). 
\end{eqnarray*}
Thus we have 
\[
{}^T\!X \rho(e_1) + \rho(e_1) X = 0 \quad 
 \Longleftrightarrow \quad {}^tX+X=0 \quad (\text{in $M(k;\C)$})
\]
and 
\[
\gerH_{p,q} = \set{X \in M(k;\C)}{{}^tX+X=0} = \gers\gero(k,\C).
\]


\subsection{The case: $(T,T')$, $({\Bbb K}, {\Bbb K}')=({\Bbb C}, {\Bbb R})$}

In this case, $\{\bar{p},\bar{q}\}=\{0,7\}, \{2,3\}, \{3,4\}, \{6,7\}$ and 
\[
R_{p,q}=M(2^{(n-1)/2};\C) \supset R^+_{p,q} = M(2^{(n-1)/2};\R). 
\]
We may assume that $p \equiv 3 \pmod  4$, $q \equiv 0 \pmod  2$. 
By (\ref{eqn:6.2}), $\hat{e}_p^2=-1$. 
Since $p$ is odd, ${\bf i}:=\hat{e}_p$ is a central element of $R_{p,q}$ and 
the center of $R_{p,q}$ is given by ${\Bbb C} 1_{2^{(n-1)/2} } ={\Bbb R} + {\Bbb R} {\bf i}$.  
Let $W_0$ be a simple $R_{p,q}$-module. Then 
$W_0 ={\Bbb C}^{2^{(n-1)/2}} ={\Bbb R}^{2^{(n-1)/2}}+ {\bf i} { \Bbb R}^{2^{(n-1)/2}} $ is a direct sum of simple $R_{p,q}^+$-modules.
In general a not necessarily simple $R_{p,q}$-module $W=\overbrace{W_0 \oplus \cdots \oplus W_0}^k$ is identified with $M(2^{(n-1)/2},2k; {\Bbb R})$. 
Under this identification, 
$\mathcal{A}=M(2k;{\Bbb R})$ and the action of $X \in \mathcal{A}$ on $ w \in W$ is given 
\[
X\cdot w=(w_1,\ldots,w_{2k}){}^t X \quad
 (X \in \calA=M(2k;\R), w=(w_1,\ldots,w_{2k}) \in W=M(2^{(n-1)/2},2k;\R))
\]
Since $p$ is odd, we may take $\mathbf i$ as $r$ in Lemma \ref{lem:6.1} 
and 
\[
\gerH_{p,q}=\set{X \in \calA}{{}^TX\rho(\mathbf i)+\rho(\mathbf i)X=0}.
\]
Since
\[
\mathbf i (u+\mathbf i v) = -v+\mathbf i u\quad (u,v \in \R^{2^{(n-1)/2}}),
\]
the action of $r={\bf i}$ on $W$ coincides with the action of 
\begin{equation}
\label{eqn:6.4}
J_k := 
\overbrace{
J_1 \perp \cdots \perp J_1
}^k
\in M(2k;\R) = \calA, \quad 
J_1 = \begin{pmatrix} 0 & -1 \\ 1 & 0 \end{pmatrix}
\end{equation}
on $W.$
Thus we have 
\[
{}^TX\rho(r)+\rho(r)X = 0 \quad \Longleftrightarrow \quad {}^tXJ_k+J_kX=0.
\]
Hence, 
\[
\gerH_{p,q} = \set{X \in M(2k;\R)}{{}^tXJ_k+J_kX=0} = \gers\gerp(k,\R).
\]


\subsection{The case: $(T,T')$,  $({\Bbb K}, {\Bbb K}')=({\Bbb C}, {\Bbb H})$}

In this case, $\{\bar{p},\bar{q}\}=\{0,3\}, \{2,7\}, \{3,6\}, \{4,7\}$, and 
\[
R_{p,q}=M(2^{(n-1)/2};\C) \supset R^+_{p,q} = M(2^{(n-3)/2};\H).
\]
We may assume that $p \equiv 3 \pmod 4$, $q \equiv 0 \pmod 2$. 
As in  the case of $({\Bbb K}, {\Bbb K}')=({\Bbb C}, {\Bbb R})$, the center of $R_{p,q}$ is given by ${\Bbb C} 1_{2^{(n-1)/2} }={\Bbb R} + {\Bbb R} {\bf i}$ with ${\bf i}:=\hat{e}_p$. 
We write $\mathbb H = \C+\C j$. Then $\alpha j = j \bar\alpha$ for $\alpha \in \C$. 

Let $W_0$ be a simple $R_{p,q}$-module. 
Then $W_0 ={\Bbb C}^{2^{(n-2)/2}} $ and $W_0$ is also simple as an $R_{p,q}^+$-module. 
We identify $W_0$ with  
${\Bbb H}^{2^{(n-3)/2}}={\Bbb C}^{2^{(n-3)/2}} + j{\Bbb C}^{2^{(n-3)/2}}$. 
Then a not necessarily simple $R_{p,q}$-module $W=\overbrace{W_0 \oplus \cdots \oplus W_0}^k$ is identified with $M(2^{(n-3)/2},2k; {\Bbb C})$. 
Then, $\mathcal{A}=M(k;,{\Bbb H})$ may be considered as a subalgebra of $M(2k;{\Bbb C})$. 
Since $\H$ is identified with 
\[
\set{\begin{pmatrix} \alpha & \bar\beta \\ -\beta & \bar\alpha \end{pmatrix}}{\alpha,\beta\in\C} 
 = \set{X \in M(2;\C)}{J_1\bar X J_1^{-1}=X},
\]
we have
\[
\calA = M(2k;\C)^\sigma = \set{X \in M(2k;\C)}{\sigma(X)=X}, \quad 
\sigma(X) = J_k \bar X J_k^{-1},
\]
where $J_k$ is the skew symmetric matrix given by (\ref{eqn:6.4}). 
The action of $X \in \mathcal{A}$ on $w \in W$ is then given by 
\[
X\cdot w=(w_1,\ldots,w_k)\,{}^t\!\bar X. 
\]
Since $p$ is odd, we may take $\bf i$ as $r$ in Lemma \ref{lem:6.1} and
\[
\gerH_{p,q}=\set{X \in \calA}{{}^TX\rho(\mathbf i)+\rho(\mathbf i)X=0}.
\]
Since 
\[
\mathbf i (u + j v) = (u -  j v)\sqrt{-1} \quad (u,v \in \C^{2^{(n-3)/2}}),
\]
the action of $r={\bf i}$ on $W$ coincides with the action of $\sqrt{-1}H_k \in \calA$, where
\[
H_k :=  
\overbrace{
\begin{pmatrix}
-1 & 0 \\
0 & 1 \end{pmatrix} 
\perp \cdots \perp 
\begin{pmatrix}
-1 & 0 \\
0 & 1 \end{pmatrix} 
}^k
\in M(2k;\C).
\]
Hence we have 
\[
{}^TX\rho(r)+\rho(r)X = 0 \quad \Longleftrightarrow \quad {}^t\bar XH_k+H_kX=0.
\]
Therefore 
\begin{eqnarray*}
\gerH_{p,q} 
  &=& \set{X \in \calA}{{}^t\bar XH_k+H_kX=0}  \\
  &=& \set{X \in \calA}{{}^t XH_kJ_k+H_kJ_kX=0}.
\end{eqnarray*}
Put 
\[
U_k=\overbrace{\frac{1}{\sqrt{2}} 
      \begin{pmatrix} 1 & 1 \\ \sqrt{-1} & -\sqrt{-1} \end{pmatrix}
      \perp \cdots \perp 
      \frac{1}{\sqrt{2}}\begin{pmatrix} 1 & 1 \\ \sqrt{-1} & -\sqrt{-1} \end{pmatrix}
      }^k.
\]
Then ${}^tU_kU_k=H_kJ_k$. 
The mapping $X \mapsto Y=U_kXU_k^{-1}$ stabilizes $\calA=M(k,\H)$ in $M(2k;\C)$ and gives an isomorphism
\begin{eqnarray*}
\gerH_{p,q} 
 &\cong& \set{Y \in \calA}{{}^tY+Y=0} \\
 &=& \set{Y \in M(2k;\C)}{{}^t Y+Y=0,\ {}^t\bar YJ_k+J_kY=0}
 = \gers\gero^*(2k).
\end{eqnarray*}


\subsection{The case: $(T,T')$,  $({\Bbb K}, {\Bbb K}')=({\Bbb H}, {\Bbb C})$}

In this case, $\{\bar{p},\bar{q}\}=\{0,6\}, \{2,4\}$ and 
\[
R_{p,q}=M(2^{(n-2)/2};\H) \supset R^+_{p,q} = M(2^{(n-2)/2};\C). 
\]
We may assume that $p \equiv 0 \pmod 4$ and $q \equiv 2 \pmod 4$. 
Then, $\hat{e}_p^2=1$,  $\hat{f}_q^2 =-1$. 
Hence ${\bf i}:=\hat{e}_p \hat{f}_q$ satisfies that ${\bf i}^2 =-1$ and  
the center of $R_{p,q}^+$ is given by 
${\Bbb C} 1_{2^{(n-1)/2} }={\Bbb R} + {\Bbb R} {\bf i}$. 
Let $W_0$ be a simple $R_{p,q}$-module. 
Then 
$W_0 ={\Bbb H}^{2^{(n-2)/2}}$ decomposes to the direct sum of two isomorphic simple $R_{p,q}^+$-modules: 
$W_0 ={\Bbb H}^{2^{(n-2)/2}} ={\Bbb C}^{2^{(n-2)/2}} + {\Bbb C}^{2^{(n-2)/2}} j$. 
In general, a not necessarily simple  $R_{p,q}$-module $W=\overbrace{W_0 \oplus \cdots \oplus W_0}^k$ is identified with $M(2^{(n-2)/2},2k; {\Bbb C})$. 
Under this identification, $\mathcal{A}=M(2k,{\Bbb C})$ and the action of $X \in \mathcal{A}$ on $ w \in W$ is given by 
\[
X\cdot w=(w_1,\ldots,w_k)\,{}^t \!\bar X. 
\]
Since ${\bf j}:=j 1_{2^{(n-2)/2}}$ anticommutes with ${\bf i}=\hat{e}_p \hat{f}_q$, ${\bf j}$ is an odd element of $R_{p,q}$. 
Hence we may take ${\bf j}$ as $r$ in Lemma \ref{lem:6.1} and 
\[
\gerH_{p,q}=\set{X \in \calA}{{}^TX\rho(\mathbf j)+\rho(\mathbf j)X=0}.
\]
Denote by ${\bf c}$ the complex conjugation on $W=M(2^{(n-2)/2},2k;{\Bbb C})$.  Then, since 
\[
\mathbf j (u+v j) = -\bar v + \bar u j  \quad (u,v \in \C^{2^{(n-2)/2}}), 
\]
 the action of $r={\bf j}$ on $W$ coincides with the action of $J_k \mathbf c$. 
Hence we have 
\[
({}^T \!X \rho(r) + \rho(r) X)\cdot w 
  = (\mathbf j w)X + \mathbf j (w\,{}^t\bar X)
  = \bar w\, {}^tJ_kX + \bar w\,{}^t X\, {}^tJ_k
  = \bar w ({}^tJ_kX + {}^t X\, {}^tJ_k).
\]
Then we have 
\[
{}^TX\rho(r)+\rho(r)X = 0 \quad \Longleftrightarrow \quad {}^tXJ_k+J_kX=0.
\]
Therefore
\[
\gerH_{p,q} = \gers\gerp(k,\C).
\]

\subsection{The case:  $(T, T' \oplus T')$}

In this case,  
\[
\{\bar p,\bar q\} 
 = \begin{cases}
    \{0,0\}, \{2,2\}, \{4,4\}, \{6,6\} & ({\Bbb K}={\Bbb K}'={\Bbb R})\\
    \{0,4\}, \{2,6\} & ({\Bbb K}={\Bbb K}'=\Bbb H)
    \end{cases}
\]
and 
\[
R_{p,q}=M(\ell;\mathbb K) \supset R_{p,q}^+=M(r;\mathbb K)\oplus M(r;\mathbb K), 
\]
where $\ell, r$ are as in Lemma \ref{lem:T and T'}. 

We have $p \equiv q \pmod 4$. Then 
\[
c^\pm=\frac 12\left(1 \pm \hat e_p\hat f_q\right)
\]
are central orthogonal idempotents of $R_{p,q}^+$. 
For an $R_{p,q}$-module $W$, put $W^{\pm} =c^{\pm}W$. 
Then $W^\pm$ are the isotypic components of $W$ as $R_{p,q}^+$-module and  $W=W^+ \oplus W^-$.
If $W_0$ is a simple $R_{p,q}$-module, the decomposition $W_0 =W_0^+ \oplus W_0^-$ gives two (non-isomorphic) simple $R_{p,q}^+$-modules. 
Then we have the following decomposition 
\[
W=\overbrace{W_0\oplus \cdots \oplus W_0}^k
 =\overbrace{W_0^+\oplus \cdots \oplus W_0^+}^k\oplus \overbrace{W_0^-\oplus \cdots \oplus W_0^-}^k
\]
and the action of $\mathcal{A}=M(k;{\Bbb K})$ on $W$ is given by 
\[
(X_1,X_2)\cdot (w_+,w_-) = (w_+{}^t\bar X_1,w_-{}^t \bar X_2) \quad 
((X_1,X_2) \in M(k;\mathbb K)\oplus M(k;\mathbb K)).
\]
Since $p$ is even, we have 
\[
e_1c^+=c^-e_1, \quad e_1c^-=c^+e_1. 
\]
Hence $\rho (e_1)$ induces a linear isomorphism 
\[
\phi:W_0^+\longrightarrow W_0^-, \quad \psi:W_0^-\longrightarrow W_0^+.
\]
Since $e_1^2=1$, we have $\psi =\phi^{-1}$. 
Hence the action of $e_1$ on $W$ is given by 
\[
e_1\cdot(w_1^+,\ldots,w_k^+,w_1^-,\ldots,w_k^-) 
 = (\phi^{-1}(w_1^-),\ldots,\phi^{-1}(w_k^-),\phi(w_1^+),\ldots,\phi(w_k^+)). 
\]
Taking $e_1$ as $r$ in Lemma \ref{lem:6.1}, we have, for $X=(X_1,X_2) \in \mathcal{A}$,  
\[
({}^TX\rho(r)+\rho(r)X)(w^+,w^-) 
 = (\phi^{-1}(w^-)X_1+\phi^{-1}(w^-\,{}^t\bar X_2),  \phi(w^+)X_2+\phi(w^+\,{}^t \bar X_1).
\] 
Hence, we have 
\[
{}^TX\rho(r)+\rho(r)X =0 \quad \Longleftrightarrow \quad 
X_2=-\phi \circ {}^t\bar X_1 \circ \phi^{-1}.
\]
Therefore
\[
\gerH_{p,q}(\rho)=\set{(X_1,-\phi \circ {}^t \bar X_1 \circ \phi^{-1})}{X_1 \in M(k,\mathbb K)} \cong \gerg\gerl(k,\mathbb K). 
\]


\subsection{The case: $(T \oplus T, T')$}

In this case, 
\[
\{\bar p,\bar q\} 
 = \begin{cases}
    \{0,1\}, \{1,2\}, \{4,5\}, \{5,6\} & (\mathbb K = \mathbb K' = \R) \\
    \{1,3\}, \{1,7\}, \{3,5\}, \{5,7\} & (\mathbb K = \mathbb K' = \C) \\
    \{0,5\}, \{1,4\}, \{1,6\}, \{2,5\} & (\mathbb K = \mathbb K' = \H)
   \end{cases}
\]
and 
\[
R_{p,q}=M(\ell;\mathbb K)\oplus M(\ell;\mathbb K) \supset R_{p,q}^+= M(r;\mathbb K), 
\]
where $\ell, r$ are as in Lemma \ref{lem:T and T'}. 
We may assume that $p \equiv 1 \pmod 4$.
Then 
\[
c^\pm=\frac 12\left(1 \pm \hat e_p\right)
\]
are central orthogonal idempotents of $R_{p,q}$. 
For an $R_{p,q}$-module $W$, put $W^{\pm} =c^{\pm} W$. 
Then $W^\pm$ are the isotypic components as $R_{p,q}$-module and $W=W^+ \oplus W^-$. 
Let $W_0^+$ (resp.\ $W_0^-$) be the simple $R_{p,q}$-module contained in $W^+$ (resp.\ $W^-$). Then we have 
\[
W^+=\overbrace{W_0^+\oplus\cdots\oplus W_0^+}^{k_1}, \quad 
W^-=\overbrace{W_0^-\oplus\cdots\oplus W_0^-}^{k_2}
\]
for some $k_1,k_2$. 
Since $W_0^+$ is isomorphic to $W_0^-$ as $R_{p,q}^+$-module, we have $\mathcal{A}=M(k_1 +k_2;{\Bbb K})$.
The action of $\mathcal{A}=M(k_1 +k_2;{\Bbb K})$ on $W$ is then given by 
\[
X\cdot w = w\,{}^t\!\bar X, \quad 
X \in M(k_1+k_2;\mathbb K).
\]
Since $p$ is odd, we may take $\hat{e}_p$ as $r$ in Lemma \ref{lem:6.1}. 
Then we have 
\[
r c^\pm = \pm c^\pm
\]
Hence, the action of $r$ on $W^+$ (resp. $W^-$) is $ +1$ (resp. $-1$) and the action of $r$ on $W$ coincides with the action of 
\[
I_{k_1,k_2} := 
\begin{pmatrix}
1_{k_1} & 0 \\ 
0 & -1_{k_2} 
\end{pmatrix}
\in \calA. 
\]
Hence we have
\[
{}^TX\rho(r)+\rho(r)X 
 = {}^t\bar X I_{k_1,k_2} + I_{k_1,k_2}X.
\] 
Therefore
\[
\gerH_{p,q}(\rho)=\set{X \in M(k_1+k_2;\mathbb H)}{{}^t\bar X I_{k_1,k_2} + I_{k_1,k_2}X=0} 
 \cong \begin{cases} 
       \gers\gero(k_1,k_2) & (\mathbb K = \R), \\
       \geru(k_1,k_2) & (\mathbb K = \C), \\
       \gers\gerp(k_1,k_2) & (\mathbb K = \H).
       \end{cases}
\]

\subsection{The case:  $(T \oplus T , T' \oplus T')$}

In this case, 
\[
\{\bar p,\bar q\}=\begin{cases} 
                \{3,3\}, \{7,7\} & ({\Bbb K}, {\Bbb K}')=({\Bbb C}, {\Bbb R}) \\
                \{3,7\} & ({\Bbb K}, {\Bbb K}')=({\Bbb C}, {\Bbb H})
                \end{cases}
\]
and 
\[
R_{p,q}=M(\ell;\mathbb K) \oplus M(\ell;\mathbb K) \supset R_{p,q}^+= M(r;\mathbb K')  \oplus M(r;\mathbb K'), 
\]
where $\ell, r$ are as in Lemma \ref{lem:T and T'}. 
Since $p \equiv q \equiv 3 \pmod  4$,  
we have  $\hat{e}_p^2 =\hat{f}_q^2=-1$ and 
\[
c^\pm = \frac 12\left(1 \pm \hat e_p \hat f_q\right)
\]
are central orthogonal idempotents of $R_{p,q}$ and of $R_{p,q}^+$. 
The algebra $R_{p,q}$ has two (non-isomorphic) simple modules $W_0^+, W_0^-$ 
which satisfy
\[
c^+ W^+_0 = W^+_0,\quad 
c^+ W^-_0 = \{0\},\quad 
c^- W^+_0 = \{0\},\quad 
c^- W^-_0 = W^-_0.
\]
Similarly $R_{p,q}^+$ has two(non-isomorphic)  simple modules $W_1^+, W_1^-$ which are not isomorphic which satisfy 
\[
c^+ W^+_1 = W^+_1,\quad 
c^+ W^-_1 = \{0\},\quad 
c^- W^+_1 = \{0\},\quad 
c^- W^-_1 = W^-_1.
\]
In the case where $\mathbb K'=\R$, the $R_{p,q}$-simple modules $W_0^\pm$ are the direct sum of two copies of the simple $R_{p,q}^+$-module $W_1^\pm$, and, in the case where $\mathbb K'=\H$, $W_0^\pm$ are simple as  $R_{p,q}^+$-module:
\[
W_0^\pm= \begin{cases} 
       W_1^\pm \oplus W_1^\pm & (\mathbb K' = \R), \\
       W_1^\pm  & (\mathbb K' = \H).
       \end{cases}
\]
A (not necessarily simple) $R_{p,q}$-module $W$ is written as 
\[
W = \overbrace{W_0^+\oplus\cdots\oplus W_0^+}^{k_1} \oplus 
     \overbrace{W_0^-\oplus\cdots\oplus W_0^-}^{k_2}
\]
and we have
\[
\calA= \begin{cases} 
       M(2k_1;\R) \oplus M(2k_2;\R) & (\mathbb K' = \R), \\
       M(k_1;\H) \oplus M(k_2;\H) & (\mathbb K' = \H).
       \end{cases}
\]
Since $p$ is odd, we may take $\hat{e}_p$ as $r$ in Lemma \ref{lem:6.1}.  
Then, by the calculation as in \S 6.2 or \S 6.3, the action of $r$ on $W$ coincides with the action of $(J_{k_1},J_{k_2}) \in \calA$ or $(H_{k_1},H_{k_2})\in \calA$ according as ${\Bbb K}' ={\Bbb R}$ or ${\Bbb K}' ={\Bbb H}$. 
Therefore
\[
\gerh_{p,q}(\rho) 
     = \begin{cases} 
       \gers\gerp(k_1,\R)\oplus \gers\gerp(k_2,\R) & (\mathbb K' = \R), \\
       \gers\gero^*(2k_1)\oplus \gers\gero^*(2k_2) & (\mathbb K' = \H).
       \end{cases}
\]


\subsection{The case:  $(T \oplus T \oplus T \oplus T , T' \oplus T')$}

In this case, 
\[
\{\bar p,\bar q\}=\begin{cases} 
                \{1,1\}, \{5,5\} & (\mathbb K = \mathbb K' = \R) \\
                \{1,5\} & (\mathbb K = \mathbb K' = \H)
                \end{cases}
\]
and 
\[
R_{p,q}=M(\ell;\mathbb K) \oplus M(\ell;\mathbb K)  \oplus M(\ell;\mathbb K)  \oplus M(\ell;\mathbb K) \supset R_{p,q}^+= M(r;\mathbb K)  \oplus M(r;\mathbb K), 
\]
where $\ell, r$ are as in Lemma \ref{lem:T and T'}. 
Since $p \equiv q \equiv 1 \mod 4$, we have $\hat{e}_p^2 =\hat{f}_q^2=1$ and  $\hat{e}_p, \hat{f}_q$ are central elements of $R_{p,q}$.
Put 
\[
c^\pm_p = \frac 12\left(1 \pm \hat e_p\right), \quad
c^\pm_q = \frac 12\left(1 \pm \hat f_q\right).
\]
Then the elements
\[
c^{++} = c^+_p c^+_q, \quad 
c^{+-} = c^+_p c^-_q, \quad 
c^{-+} = c^-_p c^+_q, \quad 
c^{--} = c^-_p c^-_q 
\]
give central orthogonal idempotents of $R_{p,q}$ and the elements 
\[
c^+ = c^{++}+c^{--} = \frac 12\left(1 + \hat e_p\hat f_q\right), \quad 
c^- = c^{+-}+c^{+-} = \frac 12\left(1 - \hat e_p\hat f_q\right)
\]
give central orthogonal idempotents of  $R^+_{p,q}$. 
The algebra $R_{p,q}$ has $4$ (non-isomorphic) simple modules $W_0^{\pm\pm}$ corresponding to the idempotents $c^{\pm\pm}$. 
The simple $R_{p,q}$-modules $W_0^{++}$ and $W_0^{--}$  (resp.\ $W_0^{+-}$ and $W_0^{-+}$) are isomorphic and simple as $R^+_{p,q}$-module.
A not necessarily simple $R_{p,q}$-module $W$ is written as  
\[
W = \overbrace{W_0^{++}\oplus\cdots\oplus W_0^{++}}^{k_1} \oplus 
    \overbrace{W_0^{--}\oplus\cdots\oplus W_0^{--}}^{k_2} \oplus 
    \overbrace{W_0^{+-}\oplus\cdots\oplus W_0^{+-}}^{k_3} \oplus 
    \overbrace{W_0^{-+}\oplus\cdots\oplus W_0^{-+}}^{k_4}. 
\]
Since $W_0^{++} \cong W_0^{--}$ and $W_0^{+-} \cong W_0^{-+}$ as $R_{p,q}^+$-module, we have
\[
\calA= \begin{cases} 
       M(k_1+k_2;\R) \oplus M(k_3+k_4;\R) & (\mathbb K' = \R), \\
       M(k_1+k_2;\H) \oplus M(k_3+k_4;\H) & (\mathbb K' = \H).
       \end{cases}
\]
Since $p$ is odd, we may take $\hat{e}_p$ as $r$ in Lemma \ref{lem:6.1}. 
Then, by the calculation as in \S 6.5, the action of $r$ on 
\begin{gather*}
W^+:=\overbrace{W_0^{++}\oplus\cdots\oplus W_0^{++}}^{k_1} \oplus 
    \overbrace{W_0^{--}\oplus\cdots\oplus W_0^{--}}^{k_2}, \\
W^-:=\overbrace{W_0^{+-}\oplus\cdots\oplus W_0^{+-}}^{k_3} \oplus 
    \overbrace{W_0^{-+}\oplus\cdots\oplus W_0^{-+}}^{k_4} 
\end{gather*}
coincides with the action of $I_{k_1 ,k_2}$, $I_{k_3, k_4}$, respectively.
Therefore 
\[
\gerh_{p,q}(\rho) 
     = \begin{cases} 
       \gers\gero(k_1,k_2)\oplus \gers\gero(k_3,k_4) & (\mathbb K' = \R), \\
       \gers\gerp(k_1,k_2)\oplus \gers\gerp(k_3,k_4) & (\mathbb K' = \H).
       \end{cases}
\]


\end{document}